\documentclass[reqno,12pt]{amsart}
\usepackage{geometry}
\usepackage{amsmath,amssymb,mathrsfs,color,stackengine}
\usepackage[colorlinks,
linkcolor=red,
anchorcolor=green,
citecolor=blue,
]{hyperref}
\usepackage[upint]{stix}
\usepackage{subfigure}
\usepackage{float}
\usepackage{times}
\usepackage{tikz}
\usetikzlibrary{intersections}
\usepackage{paralist}
\usepackage{tabularx}
\usepackage{booktabs}

\usepackage{comment}

\makeatletter
\def\setliststart#1{\setcounter{\@listctr}{#1}%
  \addtocounter{\@listctr}{-1}}
\makeatother

\makeatletter
\@addtoreset{figure}{section}
\makeatother

\usepackage{calc}
 \newtheorem{The}{Theorem}[section]
 \newtheorem{Cor}[The]{Corollary}
 \newtheorem{Lem}[The]{Lemma}
 \newtheorem{Pro}[The]{Proposition}
 \theoremstyle{definition}
 \newtheorem{defn}[The]{Definition}
 \newtheorem{Rem}[The]{Remark}
 
 \numberwithin{equation}{section}

\newcounter{Mr}
\newtheorem{Result}[Mr]{\textbf{Main Result}}

\newcommand{\R}{\mathbb{R}}

\newcommand{\N}{\mathbb{N}}

\newcommand{\SING}{\mbox{\rm Sing}\,}

\title[Variational construction of singular characteristics]{Variational construction of singular characteristics and propagation of singularities}
\author{Piermarco Cannarsa, Wei Cheng, Jiahui Hong and Kaizhi Wang}
\address{Dipartimento di Matematica, Universit\`a di Roma ``Tor Vergata'', Via della Ricerca Scientifica 1, 00133 Roma, Italy}
\email{cannarsa@mat.uniroma2.it}
\address{School of Mathematics, Nanjing University, Nanjing 210093, China}
\email{chengwei@nju.edu.cn}
\address{School of Mathematics, Nanjing University of Aeronautics and Astronautics, Nanjing 211106, China}
\email{jiahui.hong.nju@gmail.com}
\address{School of Mathematical Sciences, Shanghai Jiao Tong University, Shanghai 200240, China}
\email{kzwang@sjtu.edu.cn}
\date{\today}
\begin{document}
\maketitle

\begin{abstract}
On a smooth closed manifold $M$, we introduce a novel theory of maximal slope curves for any pair $(\phi,H)$ with $\phi$ a semiconcave function and $H$ a Hamiltonian.

By using the notion of maximal slope curve from gradient flow theory, the intrinsic singular characteristics constructed in [Cannarsa, P.; Cheng, W., \textit{Generalized characteristics and Lax-Oleinik operators: global theory}. Calc. Var. Partial Differential Equations 56 (2017), no. 5, 56:12], the smooth approximation method developed in [Cannarsa, P.; Yu, Y. \textit{Singular dynamics for semiconcave functions}. J. Eur. Math. Soc. 11 (2009), no. 5, 999--1024], and the broken characteristics studied in [Khanin, K.; Sobolevski, A., \textit{On dynamics of Lagrangian trajectories for Hamilton-Jacobi equations}. Arch. Ration. Mech. Anal. 219 (2016), no. 2, 861--885], we prove the existence and stability of such maximal slope curves and discuss certain new weak KAM features. We also prove that maximal slope curves for any pair $(\phi,H)$ are exactly broken characteristics which have right derivatives everywhere.

Applying this theory, we establish a global variational construction of strict singular characteristics and broken characteristics. Moreover, we prove a result on the global propagation of cut points along generalized characteristics, as well as a result on the propagation of singular points along strict singular characteristics, for weak KAM solutions. We also obtain the continuity equation along strict singular characteristics which clarifies the mass transport nature in the problem of propagation of singularities.
\end{abstract}

\tableofcontents

\section{Introduction}

The motivation for this paper is twofold. Firstly, we aim to develop a novel characteristic theory for arbitrary pairs $(\phi,H)$, where $\phi$ is a semiconcave function and $H$ is a Hamiltonian. Secondly, using the concept of ``forward characteristics'', we intend to investigate the propagation of singularities for general semiconcave functions and viscosity solutions to Hamilton-Jacobi equations. The interplay between these two areas lies in the development of the theory of propagation of singularities over the past two decades, as well as new perspectives emerging from the theory of abstract gradient flows. This insight also allows us to handle these two problems independently.

Let $M$ be a smooth connected and compact manifold without boundary and $TM$ and $T^*M$ its tangent and cotangent bundle respectively. We suppose $H:T^*M\to\R$ satisfies the following conditions:
\begin{enumerate}[(H1)]
	\item $H$ is locally Lipschitz;
	\item $p\mapsto H(x,p)$ is differentiable and $(x,p)\to H_p(x,p)$ is continuous;
	\item $p\mapsto H(x,p)$ is strictly convex, and $\lim_{|p|_x\to\infty}H(x,p)/|p|=+\infty$ uniformly.
\end{enumerate}
In this paper, for any semiconcave function $\phi$ with arbitrary modulus and a Hamiltonian satisfies conditions (H1)-(H3), we introduce a new theory of maximal slope curve for a pair $(\phi,H)$. This theory provides some new observation even in classical weak KAM theory, and also a natural variational construction of certain important objects in the study of propagation of singularities. We also study some mass transport aspect of propagation of singularities along singular characteristics. 

\subsection{Maximal slope curve for a pair $(\phi,H)$}

We borrow some idea from the abstract theory of gradient flow, the readers can refer to \cite{Ambrosio_Brue_Semola_book2021,Ambrosio_GigliNicola_Savare_book2008} for more information and the references therein. 

Given a semiconcave function $\phi$ on $M$ and a Hamiltonian $H$ satisfying conditions (H1)-(H3). Let $L:TM\to\R$ be the associated Lagrangian for $H$. Given an interval $I$ which can be the whole real line, we call a locally absolutely continuous curve $\gamma:I\to M$ an \emph{maximal slope curve for a pair $(\phi,H)$} and a Borel measurable selection $\mathbf{p}(x)$ of the Dini superdifferential $D^+\phi(x)$, if $\gamma$ is a solution of the EDI (Energy Dissipation Inequality)-type variational inequality 
\begin{equation}\label{eq:msc_H_phi}\tag{EDI}
	\phi(\gamma(t_2))-\phi(\gamma(t_1))\leqslant\int^{t_2}_{t_1}\Big\{L(\gamma(s),\dot{\gamma}(s))+H(\gamma(s),\mathbf{p}(\gamma(s)))\Big\} ds,\qquad\forall t_1,t_2\in I, t_1\leqslant t_2,
\end{equation}
i.e., the inequality above is indeed an equality for such a curve $\gamma$.
The most important Borel measurable selection of superdifferential $D^+\phi(x)$ is the minimal energy selection
\begin{align*}
	\mathbf{p}^\#_{\phi,H}(x)=\arg\min\{H(x,p): p\in D^+\phi(x)\},\qquad x\in M.
\end{align*}
By the condition of equality in Fenchel-Young inequality $\gamma$ is a maximal slope curve for a pair $(\phi,H)$ and the minimal energy selection $\mathbf{p}^{\#}_{\phi,H}$ if and only if $\gamma$ satisfies the following equation
\begin{equation}\label{eq:SGC1}\tag{SC}
	\dot{\gamma}(t)=H_p(\gamma(t),\mathbf{p}^{\#}_{\phi,H}(\gamma(t))),\qquad a.e.\ t\in I,
\end{equation}
We call each solution of \eqref{eq:SGC1} a \emph{strict singular characteristic}. It is worth noting that any maximal slope curve for the pair $(\phi,H)$ is exactly a strict singular characteristic (Proposition \ref{pro:msc is sc}).

\begin{Result}
	\hfill
	\begin{enumerate}[(1)]
		\item (Existence) For any $x\in M$, there exists a strict singular characteristic $\gamma:\R\to M$ with $\gamma(0)=x$ for the pair $(\phi,H)$. (Theorem \ref{thm:msc exs} and Theorem \ref{thm:pos neg})
		\item (Stability) Let $\{H_k\}$ be a sequence of Hamiltonians satisfying \text{\rm (H1)-(H3)}, $\{\phi_k\}$ be a sequence of $\omega$-semiconcave functions on $M$, and $\gamma_k:\R\to M$, $k\in\N$ be a sequence of strict singular characteristics for the pair $(\phi_k,H_k)$. We suppose the following condition:
		\begin{enumerate}[\rm (i)]
			\item $\phi_k$ converges to $\phi$ uniformly on $M$,
			\item $H_k$ converges to a Hamiltonian $H$ satisfying \text{\rm (H1)-(H3)} uniformly on compact subsets,
			\item $\gamma_k$ converges to $\gamma:\R\to M$ uniformly on compact subsets.
		\end{enumerate}
		Then $\gamma$ is a strict singular characteristic for the pair $(\phi,H)$. (Theorem \ref{thm:stability})
	\end{enumerate}
\end{Result}

The wellposedness of \eqref{eq:SGC1} is well understood only when $H$ is quadratic in $p$-variable because of uniqueness (see, \cite{Cannarsa_Yu2009,ACNS2013}), and in the 2D case because of the special topological property of surfaces (\cite{Cannarsa_Cheng2021b}), until now. 

For the existence issue of strict singular characteristic, we provide two proofs under different conditions.
\begin{enumerate}[--]
	\item For any general semiconcave function $\phi$ and Hamiltonian $H$ satisfying conditions (H1)-(H3), we prove by extending the approximation method in \cite{Yu2006,Cannarsa_Yu2009} together with the stability property of strict singular characteristics.
	\item If $\phi$ is semiconcave with linear modulus and $H$ is a Tonelli Hamiltonian, an alternative proof is based on some observations on intrinsic singular characteristics introduced in \cite{Cannarsa_Cheng3}, a theorem by Marie-Claude Arnaud (\cite{Arnaud2011}) and the property of the commutators of Lax-Oleinik operators discussed in \cite{Cannarsa_Cheng_Hong2023}. This appraoch also provides some insight on the connection of this theory of maximal slope curves and relevant topics in symplectic geometry. 
\end{enumerate}
We emphasize that both analytic and geometric approaches of this theory are interesting in their own right.

\subsection{New weak KAM aspect of maximal slope curve}

In the nineties, Albert Fathi developed a theory for Hamilton-Jacobi equations, called weak KAM theory, which has deep connection with Mather theory of Lagrangian dynamics of time-periodic Tonelli Lagrangian system. For any Tonelli Hamiltonian $H(x,p)$ with $L(x,v)$ the associated Tonelli Lagrangian, let
\begin{align*}
	A_t(x,y)=\inf_{\xi\in\Gamma^t_{x,y}}\int^t_0L(\xi,\dot{\xi})\ ds,\qquad t>0,\ x,y\in M,
\end{align*}
with $\Gamma^t_{x,y}=\{\xi\in W^{1,1}([0,t],M): \xi(0)=x, \xi(t)=y\}$. We call $A_t(x,y)$ the fundamental solution (w.r.t. the associated Hamilton-Jacobi equation). For any continuous function $\phi$ on $M$, we introduce the Lax-Oleinik semigroup $\{T^-_t\}_{t\geqslant0}$ and $\{T^+_t\}_{t\geqslant0}$ as
\begin{align*}
	T^-_t\phi(x)=\inf_{y\in M}\{\phi(y)+A_t(y,x)\},\qquad T^+_t\phi(x)=\sup_{y\in M}\{\phi(y)-A_t(x,y)\},\qquad t>0,x\in M,
\end{align*}
and set $T^{\pm}_0=id$. We call $\{T^-_t\}_{t\geqslant0}$ and $\{T^+_t\}_{t\geqslant0}$ the negative and positive Lax-Oleinik evolutions respectively. Consider the Hamilton-Jacobi equation 
\begin{equation}\label{eq:intro_WKM}\tag{HJs}
	H(x,D\phi(x))=0,\qquad x\in M,
\end{equation}
where $0$ on the right-side of \eqref{eq:intro_WKM} is Ma\~n\'e's critical value. We call $\phi$ a weak KAM solution of \eqref{eq:intro_WKM} if $T^-_t\phi=\phi$ for all $t\geqslant 0$. For more information on weak KAM theory see \cite{Fathi_book}. 

Let $\phi$ be a weak KAM solution of \eqref{eq:intro_WKM}. Then $\phi\in\text{\rm SCL}\,(M)$, where $\text{\rm SCL}\,(M)$ is the set of semiconcave functions on $M$ with linear modulus. The following fact is basic and important, see \cite{Cannarsa_Sinestrari_book,Rifford2008}. For $x\in M$ and $p\in D^*\phi(x)$, the set of reachable differentials of $\phi$ at $x$, there is a unique $C^1$ curve $\gamma:(-\infty,0]\to M$ with $\gamma(0)=x$ such that $L_v(x,\dot{\gamma}(0))=p$ and
\begin{align*}
	\phi(\gamma(0))-\phi(\gamma(-t))=\int^0_{-t}L(\gamma,\dot{\gamma})\ ds,\qquad\forall t\geqslant0.
\end{align*}
Such a curve $\gamma$ is called a $(\phi,H)$-calibrated curve on $(-\infty,0]$ in weak KAM theory, and $\phi(\cdot)$ is differentiable at $\gamma(s)$ for all $s\in(-\infty,0)$. It follows that $H(\gamma(s),\mathbf{p}^{\#}_{\phi,H}(\gamma(s)))\equiv0$ for all $s\in(-\infty,0)$. This implies that $\gamma$ is indeed a maximal slope curve from $x$ in the negative direction for the pair $(\phi,H)$.  One can regard the energy dissipation term $H(x,\mathbf{p}^{\#}_{\phi,H}(x))$ as an extension of Ma\~n\'e's critical value.

Now, for any Tonelli Hamiltonian $H$ with associated Lagrangian $L$ and $\phi\in\text{\rm SCL}\,(M)$, we introduce a new Lagrangian
\begin{align*}
	L^{\#}_{\phi}(x,v):=L(x,v)+H(x,\mathbf{p}^{\#}_{\phi,H}(x)),\qquad (x,v)\in TM
\end{align*}
with $H^\#_\phi$ the associated Hamiltonian. The fundamental solution with respect to $L^{\#}$ is
\begin{align*}
	A^{\#}_t(x,y)=\inf_{\xi\in\Gamma^t_{x,y}}\int^t_0L^{\#}_{\phi}(\xi(s),\dot{\xi}(s))\ ds,\qquad t>0,\ x,y\in M.
\end{align*}
We define new Lax-Oleinik operators: for any continuous function $u$ on $M$
\begin{align*}
	T^{\#}_t u(x)=\inf_{y\in M}\{u(y)+A^{\#}_t(y,x)\},\qquad
	\breve{T}^{\#}_t u(x)=\sup_{y\in M}\{u(y)-A^{\#}_t(x,y)\},\qquad t>0,x\in M.
\end{align*}
We call $u$ a weak KAM solution of negative (resp. positive) type of the Hamilton-Jacobi equation
\begin{equation}\label{eq:nwkam intro}
	H^\#_\phi(x,Du(x))=0,\qquad x\in M,
\end{equation}
if $T^{\#}_tu=u$ for any $t\geqslant0$ (resp. $\breve{T}^\#_tu=u$ for any $t\geqslant0$). 

Main Result 1 implies that, for any such pair $(\phi,H)$, 
$\phi$ is both a weak KAM solution of \eqref{eq:nwkam intro} of negative and positive type. However, if $\text{Sing}\,(\phi)\not=\varnothing$, then $\phi$ is not a viscosity subsolution of \eqref{eq:nwkam intro} (Theorem \ref{thm:new_WKM}). Here we use the classical notion of viscosity (sub)solutions for discontinuous Hamiltonians.

\subsection{Various generalized characteristics}

To explain the application of this new theory to the problem of propagation of singularities, we should recall various type singular characteristics in the literature. 

The notion of generalized characteristics, which plays an important role to describe the evolution of singularities of viscosity solutions, was first introduced in \cite{Albano_Cannarsa2002} in the frame of Hamilton-Jacobi equations ($M=\R^n$). We restrict our interests to the case when $H$ is a Tonelli Hamiltonian. For any $\phi\in\text{\rm SCL}\,(M)$, a Lipschitz curve $\gamma:[0,\infty)\to\R^n$ satisfying the following differential inclusion is called a \emph{generalized characteristic} from $x\in\R^n$:
\begin{equation}\label{eq:GC}\tag{GC}
	\begin{cases}
		\dot{\gamma}(t)\in\text{co}\,H_p(\gamma(t),D^+\phi(\gamma(t))),\quad a.e.\  t\in[0,+\infty)\\
		\gamma(0)=x.
	\end{cases}
\end{equation}
In \cite{Albano_Cannarsa2002}, the authors proved if $\phi$ is a viscosity solution of the Hamilton-Jacobi equation
\begin{align*}
	H(x,D\phi(x))=0,\qquad x\in\Omega,
\end{align*}
where $\Omega\subset\R^n$ is an open set, then there exists a Lipschitz continuous solution $\mathbf{x}$ of \eqref{eq:GC} and $\delta>0$ such that $\gamma(t)\in\text{Sing}\,(\phi)$ for $t\in[0,\delta]$ provided $x\in \text{Sing}\,(\phi)$. Here, $\text{Sing}\,(\phi)$ is the set of the points of non-differentiability of $\phi$ and it is called the \emph{singular set} of $\phi$. 
Using an approximation method developed in \cite{Yu2006,Cannarsa_Yu2009}, this result has been further extended to the case where $H$ is of class $C^1$ and strictly convex in $p$-variable, and $\phi$ is an  arbitrary semiconcave function. 
From the control theory point of view, the generalized characteristic differential inclusion comes from some relaxation problem, the readers can understand this point from the original proof in \cite{Albano_Cannarsa2002}, and also in the current paper (Section 4). Due to the presence of the convex hull in the right-side of \eqref{eq:GC}, no evidence ensures the uniqueness of the solution of \eqref{eq:GC}.


Invoking a celebrated paper by Khanin and Sobolevski (\cite{Khanin_Sobolevski2016}, see also \cite{Stromberg_Ahmadzadeh2014}), we consider the equation
\begin{equation}\label{eq:SGC2}\tag{BC}
	\begin{cases}
		\dot{\gamma}^+(t)=H_p(\gamma(t),\mathbf{p}^{\#}_{\phi,H}(\gamma(t))),\qquad\forall t\in[0,+\infty),\\
		\gamma(0)=x.
	\end{cases}
\end{equation}
A solution of \eqref{eq:SGC2} is called a \emph{broken characteristic} in the literature. In \cite{Cannarsa_Cheng2021b} we introduce the notion of strict singular characteristic, which is strictly related to \eqref{eq:SGC2}.

For the local existence of solutions for \eqref{eq:SGC2} and \eqref{eq:SGC1}, there are several proofs (\cite{Khanin_Sobolevski2016,Cannarsa_Cheng2021a,Cheng_Hong2022a,Cheng_Hong2023}). The original one \cite{Khanin_Sobolevski2016} used vanishing viscosity approximation methods in case $\phi$ is a viscosity solution of an evolutionary Hamilton-Jacobi equation to prove the existence result for \eqref{eq:SGC2}, and \cite{Cannarsa_Cheng2021a} used a standard approximation by convolution for the static equation. In \cite{Cheng_Hong2022a} and \cite{Cheng_Hong2023}, some local existence results for \eqref{eq:SGC2} for viscosity solutions to evolutionary Hamilton-Jacobi equations with smooth initial data and weak KAM solution in 2D were obtained, using an intrinsic approach based on the analysis of the underlying characteristic systems.



Applying our method of maximal slope curves for strict singular characteristic, we can also give a clear explanation of the notion of broken characteristics. We prove that the two notions, broken characteristic and strict singular characteristic, coincide for any pair $(\phi,H)$.

\begin{Result}
Suppose $\phi$ is a semiconcave function on $M$, $H$ is a Hamiltonian satisfying conditions {\rm (H1)-(H3)}, and $\gamma:\R\to M$ is a strict singular characteristic for the pair $(\phi,H)$.
\begin{enumerate}[\rm (1)]
	\item  The right derivative $\dot{\gamma}^+(t)$ exists for all $t\in\R$ and $\gamma$ satisfies the broken characteristic equation \eqref{eq:SGC2}. (Theorem \ref{thm:right_derivative})
	\item The one-dimensional families of measures determined by the curves $(\gamma,\mathbf{p}^\#_{\phi,H}(\gamma))$ and $(\gamma,\dot{\gamma}^+)$, respectively, are weakly right continuous. (Theorem \ref{thm:measure_H} and Theorem \ref{thm:measure_L})
	\item Moreover, if $\phi\in\text{\rm SCL}\,(M)$ and $H$ is a Tonelli Hamiltonian, then for any $x\in M$ there exists a strict singular characteristic $\gamma:\R\to M$ with $\gamma(0)=x$, for the pair $(\phi,H)$ such that the right derivative $\dot{\gamma}^+(t)$ is right continuous for all $t\in\R$. (Theorem \ref{thm:right_continuous2})
\end{enumerate}
\end{Result}

We emphasize that, in the case when $H$ is quadratic in the $p$-variable, item (1) and (2) above are obvious in the EVI frame of gradient flow theory for every strict singular characteristic (see, for instance, \cite{Muratori_Savare2020}). This problem for non-quadratic Hamiltonian is open in the literature.

As Khanin and Sobolevski pointed out, the notion of broken characteristic is also closely related to some problem for certain stochastic Hamiltonian dynamical systems and viscous Hamilton-Jacobi equation. See \cite{Bogaevsky2002,Bogaevski2006,Khanin_Sobolevski2010,Khanin_Sobolevski2016} for more on these topics.

\subsection{Intrinsic construction of strict singular characteristics}

The notion of \emph{intrinsic singular characteristic} was introduced first in \cite{Cannarsa_Cheng3} and then developed in \cite{Cannarsa_Cheng_Fathi2017,Cannarsa_Cheng_Fathi2021} with some important topological applications to weak KAM theory and geometry. Suppose $\phi\in\text{\rm SCL}\,(M)$ and $H$ is a Tonelli Hamiltonian. For any $x\in\R^n$ we define a curve $\mathbf{y}_x:[0,\tau(\phi)]\to\R^n$ by
\begin{equation}\label{eq:curve_y_x_intro}
	\mathbf{y}_x(t):=
	\begin{cases}
		x,& t=0,\\
		\arg\max\{\phi(y)-A_t(x,y): y\in\R^n\},& t\in(0,\tau(\phi)],
	\end{cases}
\end{equation}
where $\tau(\phi)>0$ is a constant (see the beginning of Section 4). We call $\mathbf{y}_x$ an \emph{intrinsic singular characteristic} from $x$. One can prove that $\mathbf{y}_x$ can be extended to $[0,+\infty)$ since $\tau(\phi)$ is independent of $x$. By a theorem of Marie-claude Arnaud (\cite{Arnaud2011}), for any $t\in(0,\tau(\phi)]$, if $\xi_t\in\Gamma^t_{x,\mathbf{y}_x(t)}$ is the minimal curve for $A_t(x,\mathbf{y}_x(t))$, then $\xi_t$ satisfies the differential equation (Proposition \ref{pro:curve_y})
\begin{equation}\label{eq:vector_field1}
	\begin{cases}
		\dot{\xi}_t(s)=H_p(\xi_t(s),DT^+_{t-s}\phi(\xi_t(s))),& s\in[0,t),\\
		\xi_t(0)=x.
	\end{cases}
\end{equation}
Thus, it is natural to introduce a time-dependent vector field on $[-T,T]\times\R^n$ as follows: fix $T>0$ and let
\begin{align*}
	\Delta: -T=\tau_0<\tau_1<\cdots<\tau_{N-1}<\tau_N=T
\end{align*}
be a partition of the interval $[-T,T]$ with $|\Delta|=\max\{\tau_i-\tau_{i-1}: 1\leqslant i\leqslant N\}\leqslant\tau(\phi)$. For $t\in[-T,T)$, let $\tau_{\Delta}^{+}(t)=\inf\{\tau_i|\,\tau_i>t\}$. We define the vector field
\begin{align*}
	W_{\Delta}(t,x)=H_p(x,\nabla T^+_{\tau_{\Delta}^{+}(t)-t}\phi(x)),\qquad t\in[-T,T), x\in M.
\end{align*}
Notice that, for any $x\in M$, the differential equation
\begin{equation}\label{eq:ODE_W_Delta_intro}
	\begin{cases}
		\dot{\gamma}(t)=W_{\Delta}(t,\gamma(t)),\qquad t\in[-T,T],\\
		\gamma(0)=x,
	\end{cases}
\end{equation}
admits a unique piecewise $C^1$ solution $\gamma$. We prove that the vector field $W_{\Delta}(t,x)$ converges to $H_p(x,\mathbf{p}^{\#}_{\phi,H}(x))$ as $|\Delta|\to 0$, as well as the solution curves\footnote{During the annual conference of Chinese Mathematical Society in Wuhan, February 2023, Prof. Bangxian Han made a comment to the talk by Wei Cheng on the early results of this intrinsic approach. He pointed out the connection between our construction and the theory of minimizing movements initiated by De Giorgi.}.

\begin{Result}
\hfill
\begin{enumerate}[\rm (1)]
	\item For all partitions $\{\Delta\}$ of the interval $[-T,T]$, we have (Theorem \ref{thm:v limit})
	\begin{align*}
		\lim_{|\Delta|\to 0} W_{\Delta}(t,x)=H_p(x,\mathbf{p}^{\#}_{\phi,H}(x)),\qquad \forall t\in[-T,T),\ x\in M.
	\end{align*}
	\item Suppose $\{\Delta_k\}$ is a sequence of partitions of $[-T,T]$, and each $\gamma_k:[-T,T]\to M$ is a solution of the equation \eqref{eq:ODE_W_Delta_intro} with $\Delta=\Delta_k$. If $\lim_{k\to\infty}|\Delta_k|=0$ and $\gamma_k$ converges uniformly to $\gamma:[-T,T]\to M$, then $\gamma$ is a strict singular characteristic for $(\phi,H)$. (Theorem \ref{thm:solution limit})
\end{enumerate}
\end{Result}

We can conclude from this result that any limiting curve of intrinsic singular characteristic $\mathbf{y}_x(t)$ is a strict singular characteristic for the pair $(\phi,H)$ (Theorem \ref{thm:in sc}).


\subsection{Propagation of singularities}

Recall that for every weak KAM solution $\phi$ of \eqref{eq:intro_WKM}, we have that $\text{\rm Sing}\,(\phi)\subset\text{\rm Cut}\,(\phi)\subset\overline{\text{\rm Sing}\,(\phi)}$, where $\text{\rm Sing}\,(\phi)$ denotes the set of singular points of $\phi$ and $\text{\rm Cut}\,(\phi)$ the cut locus of $\phi$. Using the key observation that $M\setminus\text{Cut}\,(\phi)$ has a certain local $C^{1,1}$ property (Theorem \ref{thm:C11}), we prove a global propagation result for the cut locus even along generalized characteristics. 

\begin{Result}
Suppose $H$ is a Tonelli Hamiltonian, $\phi$ is a weak KAM solution of \eqref{eq:intro_WKM}.
\begin{enumerate}[(1)]
	\item For any $x\in M\setminus\text{\rm Cut}\,(\phi)$, let $\gamma_x:(-\infty,\tau_\phi(x)]\to M$ be the unique $(\phi,H)$-calibrated curve with $\gamma_x(0)=x$, where $\tau_\phi(x)>0$ is the cut time of $\phi$ at $x$. Then $\gamma_x$ is the unique solution of the differential inclusion \eqref{eq:GC} with $\gamma_x(0)=x$ on $(-\infty,\tau_\phi(x)]$. (Theorem \ref{thm:propagation_cut})
	\item If $\gamma:[0,+\infty)\to M$ satisfies \eqref{eq:GC} and $\gamma(0)\in \text{\rm Cut}\,(\phi)$, then $\gamma(t)\in\text{\rm Cut}\,(\phi)$ for all $t\geqslant0$. (Theorem \ref{thm:propagation_cut})
\end{enumerate}
\end{Result}

Applying the aforementioned theory of maximal slope curves, we obtain  local and global propagation results for singular points along strict singular characteristics.

\begin{Result}
Suppose $H$ is a Tonelli Hamiltonian, $\phi$ is a weak KAM solution of \eqref{eq:intro_WKM}.
\begin{enumerate}[\rm (1)]
	\item For any $x\in\text{\rm Sing}\,(\phi)$, there exists a strict singular characteristic $\gamma:[0,+\infty)\to M$ with $\gamma(0)=x$ for the pair $(\phi,H)$, such that for some $\delta>0$ we have (Theorem \ref{thm:local_prop})
	\begin{align*}
		\gamma(t)\in\text{\rm Sing}\,(\phi),\qquad \forall t\in[0,\delta].
	\end{align*}
    \item If $\gamma:[0,+\infty)\to M$ is a strict singular characteristic for $(\phi,H)$ and $\gamma(0)\in\text{\rm Cut}\,(\phi)$, then $\text{\rm supp}\,(\chi_{\text{\rm Sing}\,(\phi)}(\gamma)\mathscr{L}^1)=[0,+\infty)$, where $\chi_{\text{\rm Sing}\,(\phi)}$ is the indicator for the set $\text{\rm Sing}\,(\phi)$, and $\mathscr{L}^1$ stands for the Lebesgue measure on $[0,+\infty)$. (Theorem \ref{thm:global pro})
    \item For any $x\in\text{\rm Cut}\,(\phi)$, there exists a strict singular characteristic $\gamma:[0,+\infty)\to M$ with $\gamma(0)=x$ for $(\phi,H)$, such that  $\text{\rm int}\,(\{t\in[0,+\infty):\gamma(t)\in\text{\rm Sing}\,(\phi)\})$ is dense in $[0,+\infty)$, where $\text{\rm int}\,(A)$ stands for the interior point of $A\subset\R$. (Theorem \ref{thm:global pro})
\end{enumerate}
\end{Result} 

We remark that Albano proved in \cite{Albano2016_1} that under our assumption in Main Result 5 with the Hamiltonian $H(x,p)$ quadratic in $p$-variable, the unique strict singular characteristic propagates singularities of $\phi$ globally provided the initial point is a singular point of $\phi$. However, the problem if such global propagation result holds for general Tonelli Hamiltonian is still open.

\subsection{Mass transport}

The vector field $H_p(x,\mathbf{p}^\#_{\phi,H}(x))$ discussed in this paper does not determine a \emph{Regular Lagrangian Flow} (see, for instance,  \cite{Ambrosio_Trevisan2014,Ambrosio_Trevisan2017}) since the collision and focus of forward characteristics. One cannot use the classical DiPerna-Lions theory to deal with mass transport along strict singular characteristics. On the other hand, no evidence shows that the vector field $H_p(x,\mathbf{p}^\#_{\phi,H}(x))$ ensures the well-posedness of the associated equation $\dot{x}=H_p(x,\mathbf{p}^\#_{\phi,H}(x))$ in the forward direction. However, we can use a lifting method to consider a dynamical system on the set of strict singular characteristics even without uniqueness. 

\begin{Result}
Suppose $\phi$ is a semiconcave function on $M$, $H$ is a Hamiltonian satisfying \text{\rm (H1)-(H3)}. For any $T>0$, let $\mathscr{S}_T$ be the family of strict singular characteristics $\gamma:[0,T]\to M$ for $(\phi,H)$. Then the following holds true:
\begin{enumerate}[(1)]
	\item There exists a measurable map $\gamma:M\to\mathscr{S}_T$, $x\mapsto\gamma(x,\cdot)$ such that $\gamma(x,0)=x$ for all $x\in M$ (Proposition \ref{pro:measuable}). 

	\item Let $\Phi_\gamma^t(x)=\gamma(x,t)$ for $(t,x)\in[0,T]\times M$. For any $\bar{\mu}\in\mathscr{P}(M)$, the set of Borel probability measures on $M$, let $\mu_t=(\Phi^t_\gamma)_{\#}\bar{\mu}$ for $t\in[0,T]$. Then the curve $\mu_t$ in $\mathscr{P}(M)$ satisfies the continuity equation (Theorem \ref{thm:CE})
		\begin{equation}\label{eq:intro_CE}\tag{CE}
			\begin{cases}
				\frac d{dt}\mu+\text{\rm div}(H_p(x,\mathbf{p}^{\#}_{\phi,H}(x))\cdot\mu)=0,\\\mu_0=\bar{\mu}.
			\end{cases}
		\end{equation}
	\item For any solution,$\mu_t$ of  \eqref{eq:intro_CE}  given by the construction in point (2) above, $\frac{d^+}{dt}\mu_t$ exists for all $t>0$ in a weak sense (Theorem \ref{thm:CE}).
	\item Iif $H$ is a Tonelli Hamiltonian and $\phi$ is a weak KAM solution of \eqref{eq:intro_WKM}, then we have (Theorem \ref{thm:mass sing})
	\begin{align*}
		\mu_{t_1}(\text{\rm Cut}\,(\phi))\leqslant\mu_{t_2}(\text{\rm Cut}\,(\phi)),\qquad \mu_{t_1}(\overline{\text{\rm Sing}\,(\phi)})\leqslant\mu_{t_2}(\overline{\text{\rm Sing}\,(\phi)}),\qquad\forall 0\leqslant t_1\leqslant t_2\leqslant T.
	\end{align*}
\end{enumerate}
\end{Result}

For the propagation of singularities in the frame of optimal transport along \eqref{eq:intro_CE}, we refer to a recent paper \cite{CCSW2024}.

\medskip

Finally, we emphasize that the idea of maximal slope curves can be also adapted to time-dependent functions $\phi:\R\times M\to\R$. We explain this point in Appendix A.

\medskip

The paper is organized as follows. In Section 2, we have collected some basic facts from weak KAM theory and semiconcave functions. In Section 3, we first introduce the notion of maximal slope curve for a pair $(\phi,H)$ and prove the existence and stability property. Then, we prove that broken characteristics and strict singular characteristics coincide. We also prove the existence of broken characteristics from any initial points, for which the right-derivative is right continuous everywhere. We also build a new weak KAM setting using the Hamiltonian/Lagrangian with an extra energy dissipation term $H(x,\mathbf{p}^\#_{\phi,H}(x))$. In Section 4, for any Tonelli Hamiltonian $H$ and $\phi\in\text{\rm SCL}\,(M)$, we construct a sequence of piecewise Lipschitz vector fields $W_k(t,x)$ using intrinsic singular characteristics. We prove that such vector fields $W_k(t,x)$ converge to the strict singular characteristics system in the aspect of both vector field and solution curve. Using this intrinsic approach, we obtain an alternative proof of the existence of maximal slope curves, with quite natural geometric intuition. In Section 5, we study global propagation of singularities  along strict singular characteristics of a given weak KAM solution. We also obtain a new  result on the global propagation of cut points along any generalized characteristics. In Section 6, we introduce the continuity equation for strict singular characteristics and its relation with the cut locus and $C^{1,1}$-singular support of weak KAM solutions. There is also an appendix on the extension of our main results of this paper to the time-dependent case. 

\medskip

\noindent\textbf{Acknowledgements.} Piermarco Cannarsa was supported by the PRIN 2022 PNRR Project P20225SP98 ``Some mathematical approaches to climate change and its impacts'  (funded by the European Community-Next Generation EU), CUP E53D2301791 0001, by the INdAM  (Istituto Nazionale di Alta Matematica) Group for Mathematical Analysis, Probability and Applications, and by the MUR Excellence Department Project awarded to the Department of Mathematics, University of Rome Tor Vergata, CUP E83C23000330006. Wei Cheng was partly supported by National Natural Science Foundation of China (Grant No. 12231010). Kaizhi Wang was partly supported by National Natural Science Foundation of China (Grant No. 12171315). Wei Cheng thanks Konstantin Khanin and Jianlu Zhang for their patience when talking on the early version of this paper in the Chinese Academy of Sciences in April 2023. Wei Cheng is also obliged to Marie-Claude Arnaud for the exchange of ideas on the results of this paper and her suggestion to connect these results to symplectic geometry, during his visit to Paris in April 2024. The authors are also grateful to Qinbo Chen, Shiyong Chen, Tianqi Shi, Wenxue Wei and Zhixiang Zhu for helpful discussions.

\section{Preliminaries}

\begin{table}[h]
	\caption{Notation}
	\begin{tabularx}{\textwidth}{p{0.22\textwidth}X}
		\toprule
		$M$ & compact and connected smooth manifold without boundary\\
		$TM/T^*M$ & tangent/cotangent bundle of $M$\\
		$\text{SCL}\,(M)$ & the set of semiconcave functions with linear modulus on $M$\\
		$C^{1,1}(M)$ & the set of $C^1$ functions on $M$ with Lipschitz continuous differentials\\
		$D^{+}\phi(x)$ & the superdifferential of a function $\phi$\\
		$D^*\phi(x)$ & the set of reachable gradients of a function $\phi$\\
		$\mathbf{p}^{\#}_{\phi,H}(x)$ & the minimal-energy element of $D^{+}\phi(x)$ w. r. t. the Hamiltonian $H$\\
		$\mbox{\rm co}\,A$ & the convex hull of $A$ in some linear space\\
		$\text{Sing}\,(\phi)$ & the set of non-differentiability points of a function $\phi$\\
		$\text{Cut}\,(\phi)$ & the cut locus of a weak KAM solution $\phi$\\
		$\tau_{\phi}(x)$ & the cut time function of a given weak KAM solution $\phi$\\
		$\Phi^t_H$ & the Hamiltonian flow associated to the Hamiltonian $H$\\
		\bottomrule
	\end{tabularx}
\end{table}

\subsection{Semiconcave functions}
Let $M$ be a compact smooth manifold without boundary. Usually, $M$ is endowed with a Riemannian metric $g$, with $d$ the associated Riemannian distance. A function $\phi:M\to\R$ is called semiconcave with semiconcavity modulus $\omega$ (with respect to the metric $g$) or $\omega$-semiconcave if
\begin{equation}\label{eq:semiconcave}
	\lambda\phi(x)+(1-\lambda)\phi(y)-\phi(\gamma(\lambda))\leqslant \lambda(1-\lambda)d(x,y)\cdot\omega(d(x,y))
\end{equation}
for any $x,y\in M$, $\lambda\in[0,1]$ and a geodesic $\gamma:[0,1]\to M$ with $\gamma(0)=x$ and $\gamma(1)=y$, where $\omega:[0,\infty)\to[0,\infty)$ is a nondecreasing upper semicontinuous function such that $\lim_{r\to0^+}\omega(r)=0$. If $\omega(r)=Cr$ for some constant $C>0$, we call such a function a semiconcave function with linear modulus and $C$ is a semiconcavity constant for $\phi$. We denote by $\text{\rm SCL}\,(M)$ the family of all semiconcave functions with linear modulus on $M$. We call $x\in M$ a singular point of $\phi$ if $\phi$ is not differentiable at $x$ and denote by $\SING(\phi)$ the set of singular points of $\phi$. The superdifferential of a semiconcave function $\phi$ is defined as
\begin{align*}
	D^{+}\phi(x)=\{D\psi(x):r>0,\ \psi\in C^1(B_r(x)),\ \psi\geqslant\phi,\ \psi(x)=\phi(x)\},\quad x\in M.
\end{align*}
and the set of reachable gradients of $\phi$ is defined as
\begin{align*}
	D^{*}\phi(x)=\{\lim_{k\to\infty}D\phi(x_k):x_k\to x,\ x_k\notin\SING(\phi)\},\quad x\in M.
\end{align*}

\begin{Pro}[\cite{Cannarsa_Sinestrari_book} Theorem 2.1.7, Proposition 3.1.5, Proposition 3.3.4, Theorem 3.3.6]\label{pro:scsv}
Let $\phi$ be a semiconcave function on a compact manifold $M$. Then we have
\begin{enumerate}[\rm (1)]
	\item $\phi$ is Lipschitz continuous on $M$.
	\item for every $x\in M$, $D^{+}\phi(x)$ is a nonempty, convex, compact subset of $T^{*}_{x}M$.
	\item $D^{+}\phi(x)=\mbox{\rm co}\,D^{*}\phi(x)$ for all $x\in M$.
    \item the directional derivative of $\phi$ satisfies
    \begin{align*}
    	\partial\phi(x,v)=\min_{p\in D^+\phi(x)}\langle p,v\rangle,\quad \forall x\in M,\ v\in T_{x}M.
    \end{align*}
\end{enumerate}
\end{Pro}

The following lemma shows the map $(\phi,x)\mapsto D^{+}\phi(x)$ is upper semicontinuous.

\begin{Lem}\label{lem:usc2}
Let $\{\phi_k\}$ be a sequence of $\omega$-semiconcave functions on $M$, $\{x_k\}\subset M$ and $p_k\in D^+\phi_k(x_k)$ for each $k\in\N$. If $\phi_k$ converges uniformly to $\phi$ on $M$ as $k\to\infty$ and $\lim_{k\to\infty}x_k=x$, then $\phi$ is also $\omega$-semiconcave and $\lim_{k\to\infty}d_{D^+\phi(x)}(p_k)=0$.
\end{Lem}

\begin{proof}
It is clear that $\phi$ is still $\omega$-semiconcave by definition \eqref{eq:semiconcave} and the uniform convergence. To show $\lim_{k\to\infty}d_{D^+\phi(x)}(p_k)=0$ it is sufficient to show that if $\{p_{k_i}\}$ is subsequence such that $\lim_{i\to\infty}p_{k_i}=p$ then $p\in D^+\phi(x)$. Because of the local nature we work on Euclidean space. Indeed, since $p_{k_i}\in D^+\phi_k(x_{k_i})$ we have (see \cite[Proposition 3.3.1]{Cannarsa_Sinestrari_book})
\begin{align*}
	\phi_{k_i}(y)\leqslant\phi_{k_i}(x_{k_i})+\langle p_{k_i},y-x_{k_i}\rangle+|y-x_{k_i}|\omega(|y-x_{k_i}|),\qquad \forall i\in\N, y\in \R^n.
\end{align*}
Taking $i\to\infty$ we have $p\in D^+\phi(x)$.
\end{proof}

\begin{Lem}[\cite{Cannarsa_Sinestrari_book} Proposition 3.3.15]\label{lem:usc dir}
Let $\phi:\R^n\to\R$ be semiconcave and $x\in\R^n$, $p_0\in\R^n$. Suppose that there exists sequences $\{x_k\}\subset(\R^n\setminus\{x\})$ and $\{p_k\}$ such that
\begin{align*}
	\lim_{k\to\infty}x_k=x,\quad \lim_{k\to\infty}\frac{x_k-x}{|x_k-x|}=v,\quad p_k\in D^+\phi(x_k),\quad \lim_{k\to\infty}p_k=p_0
\end{align*}
for some unit vector $v\in\R^n$. Then we have $p_0\in D^+\phi(x)$ and
\begin{align*}
	\langle p_0,v\rangle=\min_{p\in D^+\phi(x)}\langle p,v\rangle.
\end{align*}
\end{Lem}

\subsection{Measurable selection}
In this section, we recall fundamental results about measurable selection.

\begin{defn}
Let $X$ be a topological space and $(S,\Sigma)$ a measurable space. A set-valued map $F:S\rightrightarrows X$ is called \emph{weakly measurable} (resp. \emph{measurable}) if $\{s\in S: F(s)\cap E\neq\emptyset\}\in\Sigma$ for any open (resp. closed) subset $E$ of $X$.
\end{defn}

\begin{Pro}[\cite{Clarke_book1990} Theorem 3.1.1]\label{pro:selection}
Let $X$ be a measurable space, let $Y$ be a Polish space, and let $F:X\rightrightarrows Y$ be a non-empty closed-valued measurable set-valued map. Then $F$ admits a measurable selection.
\end{Pro}

\begin{Pro}[\cite{Aliprantis_Border_book2006} Lemma 18.2]
Let $X$ be a metric space and let $(S,\Sigma)$ a measurable space. For any set-valued map $F:(S,\Sigma)\rightrightarrows X$ we have that:
	\begin{enumerate}[\rm (1)]
		\item If $F$ is measurable, then it is also weakly measurable.
		\item If $F$ is compact-valued and weakly measurable, then it is measurable.
	\end{enumerate}
\end{Pro}

\begin{defn}
	Let $(S,\Sigma)$ be a measurable space, and let $X$ and $Y$ be topological spaces. A function $f:S\times X\to Y$ is a Carath\'eodory function if it satisfies the following conditions:
	\begin{enumerate}[\rm (1)]
		\item for each $x\in X$, the function $f^x(\cdot)=f(\cdot,x)$ is $(\Sigma,\mathscr{B}(Y))$-measurable;
		\item for each $s\in S$, the function $f_s(\cdot)=f(s,\cdot)$ is continuous.
	\end{enumerate}
\end{defn}

\begin{Pro}[\cite{Aliprantis_Border_book2006} Theorem 18.19]\label{pro:max_meas}
	Let $X$ be a separable metrizable space and $(S,\Sigma)$ a measurable space. Let $F:S\rightrightarrows X$ be a weakly measurable set-valued map with nonempty compact values, and suppose $f:S\times X\to\R$ is a Carath\'eodory function. Define the marginal function $m:S\to\R$ by
	\begin{align*}
		m(s)=\max_{x\in F(s)}f(s,x)
	\end{align*}
	and the set-valued map of maximizers $\Lambda:S\rightrightarrows X$ by
	\begin{align*}
		\Lambda(s)=\{x\in F(s): f(s,x)=m(s)\}.
	\end{align*}
	Then the following holds:
	\begin{enumerate}[\rm (1)]
		\item The value function $m$ is measurable.
		\item The argmax set-valued map $\Lambda$ has nonempty and compact values.
		\item  The argmax set-valued map $\Lambda$ is measurable and admits a measurable selection.
	\end{enumerate}
\end{Pro}

\subsection{Lax-Oleinik evolution}

In this paper, we suppose $H:T^*M\to\R$ satisfies conditions (H1)-(H3). Thus, the associated Lagrangian $L:TM\to\R$ defined by
\begin{align*}
	L(x,v)=\sup_{p\in T^*_xM}\{\langle p,v\rangle-H(x,p)\},\qquad x\in M, v\in T_xM
\end{align*}
satisfies conditions
\begin{enumerate}[(L1)]
	\item $L$ is locally Lipschitz;
	\item $v\mapsto L(x,v)$ is differentiable and $(x,v)\to L_v(x,v)$ is continuous;
	\item $v\mapsto L(x,v)$ is strictly convex, and uniformly superlinear.
\end{enumerate}
We say $H:T^*M\to\R$ is a Tonelli Hamiltonian if $H$ is of class $C^2$ and satisfies conditions
\begin{enumerate}[(H1')]
	\item $H_{pp}(x,p)>0$ for all $(x,p)\in T^*M$
	\item $p\mapsto H(x,p)$ is uniformly superlinear.
\end{enumerate}
If $H$ is a Tonelli Hamiltonian, we call the associated Lagrangian $L$ a Tonelli Lagrangian.

For any function $\phi\in C(M,\R)$ and $L$ satisfying conditions (L1)-L(3), as in weak KAM theory, we define the Lax-Oleinik semigroups $\{T^{\pm}_t\}_{t\geqslant0}$ as operators
\begin{align*}
	T^+_t\phi(x)=\sup_{y\in M}\{\phi(y)-A_t(x,y)\},\quad T^-_t\phi(x)=\inf_{y\in M}\{\phi(y)+A_t(y,x)\},\quad x\in M,
\end{align*}
where $A_t(x,y)$ is the fundamental solution given by
\begin{align*}
	A_t(x,y)=\inf_{\xi\in\Gamma^t_{x,y}}\int^t_0L(\xi(s),\dot{\xi}(s))\ ds,\quad t>0,\ x,y\in M,
\end{align*}
with 
\begin{align*}
\Gamma^t_{x,y}=\{\xi:[0,t]\to M|\,\xi \text{ is absolutely continuous, and }\xi(0)=x, \xi(t)=y\}
\end{align*}
Specially, we set $T^{\pm}_0=id$. A function $\phi:M\to\R$ is called a \emph{weak KAM solution} for the Hamilton-Jacobi equation
\begin{equation}\label{eq:HJ_wk}\tag{HJs}
	H(x,D\phi(x))=0,\qquad x\in M,
\end{equation}
if $T^-_t\phi=\phi$ for all $t\geqslant0$. Here we suppose $0$ on the right-hand side of \eqref{eq:HJ_wk} is the Ma\~n\'e critical value, for convenience. The readers can find some details on weak KAM theory under our conditions \text{(H1)-(H3)} in \cite{Fathi_Siconolfi2005}.

If $H$ is a Tonelli Hamiltonian and $\phi$ is a weak KAM solution of \eqref{eq:HJ_wk}, then $\phi\in\text{SCL}\,(M)$. We define the \emph{cut time function} of $\phi$ as
\begin{align*}
	\tau_\phi(x)=\sup\{t\geqslant0: \exists\gamma\in C^1([0,t],M), \gamma(0)=x, \phi(\gamma(t))-\phi(x)=A_t(x,\gamma(t))\}.
\end{align*}
Then $\text{Cut}\,(\phi)=\{x\in M: \tau_\phi(x)=0\}$ is called the \emph{cut locus} of $\phi$.

\subsection{Lasry-Lions regularization and Arnaud's theorem}

The following proposition is a collection of some useful Lasry-Lions type facts of the the Lax-Oleinik commutators for small time (\cite{Bernard2007,Bernard2010,Arnaud2011}) by Patrick Bernard and Marie-Claude Arnaud. For any $\phi\in\text{\rm SCL}\,(M)$, set
\begin{align*}
	\text{\rm graph}\,(D^+\phi)=\{(x,p): x\in M, p\in D^+\phi(x)\subset T^*_xM\}.
\end{align*}
We denote by $\{\Phi_H^t\}_{t\in\R}$ the Hamiltonian flow of $H$

\begin{Pro}\label{pro:graph}
Suppose $H:T^{*}M\to\R$ is a Tonelli Hamiltonian and $\phi\in\text{\rm SCL}\,(M)$. Then there exists $\tau_1(\phi)>0$ such that the following properties hold.
\begin{enumerate}[\rm (1)]
	\item $T^+_t\phi=T^-_{\tau_1(\phi)-t}\circ T^+_{\tau_1(\phi)}\phi$ for all $t\in[0,\tau_1(\phi)]$, and $T^+_t\phi\in C^{1,1}(M)$ for all $t\in(0,\tau_1(\phi)]$.
	\item {\rm (Arnaud)} For all $t\in(0,\tau_1(\phi))$ we have
	\begin{align*}
		\text{\rm graph}\,(D T^+_t\phi)=\Phi_H^{-t}(\text{\rm graph}\,(D^+\phi)).
	\end{align*} 
	\item Let $u(t,x)=T^+_t\phi(x)$ for $(t,x)\in[0,\tau_1(\phi)]\times M$. Then $u$ is of class $C^{1,1}_{\rm loc}$ on $(0,\tau_1(\phi))\times M$  and it is a viscosity solution of the Hamilton-Jacobi equation
	\begin{align*}
		\begin{cases}
			u_t-H(x,\nabla u)=0,\qquad(t,x)\in(0,\tau_1(\phi))\times M;\\
			u(0,x)=\phi(x).
		\end{cases}
	\end{align*}
\end{enumerate}
\end{Pro}

\subsection{Regularity properties of $A_t(x,y)$}

The following proposition on the regularity properties of $A_t(x,y)$ is standard. For the proof the readers can refer to \cite{Cannarsa_Cheng3,Cannarsa_Cheng2021a,Bernard2012}. 

\begin{Pro}\label{pro:regularity}
Suppose $L$ is a Tonelli Lagrangian. Then for any $\lambda>0$
\begin{enumerate}[\rm (1)]
	\item there exists a constant $C_\lambda>0$ such that for any $x\in\R^n$, $t\in(0,2/3)$ the function $y\mapsto A_t(x,y)$ defined on $B(x,\lambda t)$ is semiconcave with constant $\frac{C_{\lambda}}{t}$;
	\item there exist $C'_\lambda>0$ and $t_{\lambda}>0$ such that for any $x\in\R^n$ the function $y\mapsto A_t(x,y)$ is convex with constant $\frac{C'_{\lambda}}{t}$ on $B(x,\lambda t)$ with $0<t\leqslant t_{\lambda}$.
	\item there exists $t'_{\lambda}>0$ such that for any $x\in\R^n$ the function $y\mapsto A_t(x,y)$ is of class $C^2$ on $B(x,\lambda t)$ with $0<t\leqslant t'_{\lambda}$. Moreover, 
	\begin{align*}
		D_yA_t(x,y)=&L_v(\xi(t),\dot{\xi}(t)),\\
		D_xA_t(x,y)=&-L_v(\xi(0),\dot{\xi}(0)),\\
		D_tA_t(x,y)=&-E(\xi(s),\dot{\xi}(s)),\label{eq:diff_A_t_t}
	\end{align*}
	where $\xi\in\Gamma^t_{x,y}$ is the unique minimizer for $A_t(x,y)$. We remark that
	\begin{align*}
		E(x,v):=L_v(x,v)\cdot v-L(x,v),\quad (x,v)\in\R^n\times\R^n,
	\end{align*}
	is the energy function in the Lagrangian formalism, and
	\begin{align*}
		E(\xi(s),\dot{\xi}(s))=H(\xi(s),p(s)),\quad s\in[0,t],
	\end{align*}
	for the dual arc $p(s)=L_v(\xi(s),\dot{\xi}(s))$;
\end{enumerate} 
\end{Pro}

\section{Maximal slope curve and strict singular characteristic}

In this section, we will deal with the new notion of Maximal Slope Curve for a pair $(\phi,H)$, where $\phi:M\to\R$ is semiconcave and $H$ satisfies condition (H1)-(H3). 

\subsection{Maximal slope curves and strict singular characteristics of $(\phi,H)$}

The notion of maximal slope curve can be understood as the steepest descent curve of a function, which plays an important role in the theory of gradient flows in metric space. Our treatment of maximal slope curves has the same spirit as the classical one. The readers can refer to the monograph \cite{Ambrosio_GigliNicola_Savare_book2008,Ambrosio_Brue_Semola_book2021} and the references therein for more details. 

Let $\phi$ be any semiconcave function on $M$. By Proposition \ref{pro:scsv} (2) and the strict convexity of $H(x,\cdot)$, the set $\arg\min\{H(x,p): p\in D^{+}\phi(x)\}$ is a singleton for any $x\in M$. We set the \emph{minimal energy selection} 
\begin{align*}
\mathbf{p}^{\#}_{\phi,H}(x)=\arg\min\{H(x,p): p\in D^{+}\phi(x)\},\quad x\in M.
\end{align*}

\begin{Lem}\label{lem:Borel}
Suppose $\phi$ is a semiconcave function on $M$ and $H$ satisfies condition (H1)-(H3). Then the map $x\mapsto\mathbf{p}^\#_{\phi,H}(x)$ is Borel measurable. 
\end{Lem}


\begin{proof}
In local coordinate neighborhood, we take $S=X=\R^n$, $F(x)=D^{+}\phi(x)$, $f=-H$ in Proposition \ref{pro:max_meas} and endow $S$ with the $\sigma$-algebra $\Sigma$ of the Borel class of $\R^n$. By Proposition \ref{pro:scsv} (2) and Lemma \ref{lem:usc2}, the set-valued map $x\rightrightarrows D^{+}\phi(x)$ is measurable, nonempty and compact. Thus, the Borel measurability of $\mathbf{p}^\#_{\phi,H}$ is a direct consequence of Proposition \ref{pro:max_meas}.
\end{proof}

\begin{Pro}\label{pro:msc}
Let $\phi$ be a semiconcave function on $M$, let $H$ be a Hamiltonian satisfying \text{\rm (H1)-(H3)}, and let $\mathbf{p}(x)$ be a Borel measurable selection of the superdifferential $D^+\phi(x)$. Then for any absolutely continuous curve $\gamma:[0,t]\to M$, we have that
\begin{equation}\label{eq:msc ineq}
	\phi(\gamma(t))-\phi(\gamma(0))\leqslant\int_{0}^{t}\Big\{L(\gamma(s),\dot{\gamma}(s))+H(\gamma(s),\mathbf{p}(\gamma(s)))\Big\}\ ds,
\end{equation}
and $\gamma$ satisfies the equality in \eqref{eq:msc ineq} if and only if
\begin{equation}
	\dot{\gamma}(s)=H_{p}(\gamma(s),\mathbf{p}(\gamma(s))),\quad a.e.\ s\in[0,t].
\end{equation}
\end{Pro}

\begin{proof}
	By Proposition \ref{pro:scsv} (3), for almost all $s\in[0,t]$ we have
	\begin{align*}
		\min_{p\in D^+\phi(\gamma(s))}\langle p,\dot{\gamma}(s)\rangle=\frac {d^+}{ds}\phi(\gamma(s))=\frac {d^-}{ds}\phi(\gamma(s))=-\min_{p\in D^+\phi(\gamma(s))}\langle p,-\dot{\gamma}(s)\rangle=\max_{p\in D^+\phi(\gamma(s))}\langle p,\dot{\gamma}(s)\rangle.
	\end{align*}
	Combing this with Young's inequality, it follows that
	\begin{align*}
		\phi(\gamma(t))-\phi(\gamma(0))&=\int_{0}^{t}\frac{d}{ds}\phi(\gamma(s))\ ds=\int_{0}^{t}\langle\mathbf{p}(\gamma(s)),\dot{\gamma}(s)\rangle\ ds\\
		&\leqslant \int_{0}^{t}\Big\{L(\gamma(s),\dot{\gamma}(s))+H(\gamma(s),\mathbf{p}(\gamma(s)))\Big\}\ ds,
	\end{align*}
where equality holds if and only if
\begin{align*}
	\dot{\gamma}(s)=H_{p}(\gamma(s),\mathbf{p}(\gamma(s))),\quad a.e.\ s\in[0,t].
\end{align*}
\end{proof}

Now, thanks to Proposition \ref{pro:msc}, we introduce the notion of maximal slope curve and strict singular characteristic for a pair $(\phi,H)$.

\begin{defn}\label{defn:msc}
Let $\phi$ be a semiconcave function on $M$, let $H$ be a Hamiltonian satisfying \text{\rm (H1)-(H3)}, and let $\mathbf{p}$ be a Borel measurable selection of the superdifferential $D^+\phi$.
\begin{enumerate}[(1)]
	\item We call a locally absolutely continuous curve $\gamma:I\to M$ a \emph{maximal slope curve for the pair $(\phi,H)$ and the selection $\mathbf{p}$}, where $I$ is any interval which can be the whole real line, if $\gamma$ satisfies
	\begin{equation}\label{eq:msc_H_phi2}\tag{VI}
		\phi(\gamma(t_2))-\phi(\gamma(t_1))=\int_{t_1}^{t_2}\Big\{L(\gamma(s),\dot{\gamma}(s))+H(\gamma(s),\mathbf{p}(\gamma(s)))\Big\} ds,\qquad\forall t_1,t_2\in I,\ t_1<t_2,
	\end{equation}
	or, equivalently,
	\begin{equation}
		\dot{\gamma}(t)=H_{p}(\gamma(t),\mathbf{p}(\gamma(t))),\quad a.e.\ t\in I.
	\end{equation}
	\item For the minimal energy selection $\mathbf{p}^{\#}_{\phi,H}$, we call any associated maximal slope curve $\gamma:I\to M$ for $(\phi,H)$ a \emph{strict singular characteristic} for the pair $(\phi,H)$. We use the term \emph{singular} in the definition since it is essentially connected to the phenomenon of propagation of singularities of $\phi$ in forward direction when the initial point is a singular point of $\phi$.
\end{enumerate}
\end{defn}

\begin{Pro}\label{pro:msc is sc}
Let $\phi$ be a semiconcave function on $M$, let $H$ be a Hamiltonian satisfying \text{\rm (H1)-(H3)},  and let $\mathbf{p}$ be a Borel measurable selection of $D^+\phi$. If $\gamma:I\to M$ is a maximal slope curve for the pair $(\phi,H)$ and the selection $\mathbf{p}$, where $I$ is any interval, then $\gamma$ must be a strict singular characteristic for the pair $(\phi,H)$. In other words, if $\gamma$ is a solution of the differential inclusion
\begin{align*}
	\dot{\gamma}(t)\in H_p(\gamma(t),D^+\phi(\gamma(t))),\qquad a.e.\ t\in I,
\end{align*}
then $\gamma$ is a solution of the differential equation
\begin{align*}
	\dot{\gamma}(t)=H_p(\gamma(t),\mathbf{p}^{\#}_{\phi,H}(\gamma(t))),\qquad a.e.\ t\in I.
\end{align*}
\end{Pro}

\begin{proof}
Since $\gamma:I\to M$ is a maximal slope curve for $(\phi,H)$ and $\mathbf{p}$, we have
\begin{align*}
	\phi(\gamma(t_2))-\phi(\gamma(t_1))=&\,\int^{t_2}_{t_1}\Big\{L(\gamma(s),\dot{\gamma}(s))+H(\gamma(s),\mathbf{p}(\gamma(s))\Big\}\ ds\\
	\geqslant&\,\int^{t_2}_{t_1}\Big\{L(\gamma(s),\dot{\gamma}(s))+H(\gamma(s),\mathbf{p}^\#_{\phi,H}(\gamma(s))\Big\}\ ds,\quad \forall t_1,t_2\in I,\ t_1<t_2.
\end{align*}
Since the converse inequality is true for any selection of $D^+\phi$, it follows that $\gamma$ is a strict singular characteristic for the pair $(\phi,H)$.
\end{proof}

The proposition above implies that any maximal slope curve for the pair $(\phi,H)$ is exactly a strict singular characteristic. This means that the minimal energy selection $\mathbf{p}^\#_{\phi,H}(x)$ plays a crucial role in this theory. In this case, the notions of maximal slope curve and strict singular characteristic coincide.

\subsection{Stability of strict singular characteristics}

We will adress the existence of maximal slope curves in the next subsection. In this subsection, we study the stability issue for strict singular characteristics at first, which is known only for quadratic Hamiltonians (see \cite{ACNS2013}). This will help us to obtain the existence of strict singular characteristics in the general case.

The following lemma implies that the minimal-energy function is lower semicontinuous, which is essential for the proof of our stability result.

\begin{Lem}\label{lem:lsc1}
	Let $\{H_k\}$ be a sequence of Hamiltonians satisfying \text{\rm (H1)-(H3)}, $\{\phi_k\}$ be a sequence of $\omega$-semiconcave functions on $M$, and $\{x_k\}\subset M$. If
	\begin{enumerate}[\rm (i)]
		\item $\phi_k$ converges to $\phi$ uniformly on $M$,
		\item $H_k$ converges to a Hamiltonian $H$ satisfying \text{\rm (H1)-(H3)} uniformly on compact subset,
		\item $\lim_{k\to\infty}x_k=x$,
	\end{enumerate}
	then
	\begin{align*}
		\liminf_{k\to\infty}H_k(x_k,\mathbf{p}^{\#}_{\phi_k,H_k}(x_k))\geqslant H(x,\mathbf{p}^\#_{\phi,H}(x)).
	\end{align*}
\end{Lem}

\begin{proof}
	Because of the local nature, we work on Euclidean space. For any subsequence $\mathbf{p}^{\#}_{\phi_{k_i},H_{k_i}}(x_{k_i})$ converging to some $p$ as $i\to\infty$, we have that $p\in D^+\phi(x)$ by Lemma \ref{lem:usc2}. Therefore,
	\begin{align*}
		\lim_{i\to\infty}H_{k_i}(x_{k_i},\mathbf{p}^{\#}_{\phi_{k_i},H_{k_i}}(x_{k_i}))=H(x,p)\geqslant H(x,\mathbf{p}^\#_{\phi,H}(x)).
	\end{align*}
	This completes the proof.
\end{proof}

\begin{Lem}\label{lem:P_k_P}
	Under the assumptions of Lemma \ref{lem:lsc1}, if $\lim_{k\to\infty}H_k(x_k,\mathbf{p}^\#_{\phi_k,H_k}(x_k))=H(x,\mathbf{p}^\#_{\phi,H}(x))$, then we have that
	\begin{align*}
		\lim_{k\to\infty}\mathbf{p}^\#_{\phi_k,H_k}(x_k)=\mathbf{p}^\#_{\phi,H}(x).
	\end{align*}
\end{Lem}

\begin{proof}
	We work on Euclidean space. Given $\varepsilon>0$, if $D^+\phi(x)\subset B(\mathbf{p}^\#_{\phi,H}(x),\varepsilon)$, by Lemma \ref{lem:usc2} we get
	\begin{align*}
		\limsup_{k\to\infty}|\mathbf{p}^\#_{\phi_k,H_k}(x_k)-\mathbf{p}^\#_{\phi,H}(x)|\leqslant\varepsilon.
	\end{align*}	
	If $D^+\phi(x)\setminus B(\mathbf{p}^\#_{\phi,H}(x),\varepsilon)\not=\varnothing$, there exists $\delta>0$ such that
	\begin{align*}
		\min_{p\in D^+\phi(x)\setminus B(\mathbf{p}^\#_{\phi,H}(x),\varepsilon)}H(x,p)>H(x,\mathbf{p}^\#_{\phi,H}(x))+\delta.
	\end{align*}
	In this case, if there exists a subsequence $\{x_{k_i}\}$ such that
	\begin{align*}
		|\mathbf{p}^\#_{\phi_{k_i},H_{k_i}}(x_{k_i})-\mathbf{p}^\#_{\phi,H}(x)|>\varepsilon,\qquad\forall i\in\N,
	\end{align*}
	then by Lemma \ref{lem:usc2} we conclude that
	\begin{align*}
		\lim_{i\to\infty}H_{k_i}(x_{k_i},\mathbf{p}^\#_{\phi_{k_i},H_{k_i}}(x_{k_i}))=&\,\lim_{i\to\infty}H(x,\mathbf{p}^\#_{\phi_{k_i},H_{k_i}}(x_{k_i}))\\
		\geqslant&\,\min_{p\in D^+\phi(x)\setminus B(\mathbf{p}^\#_{\phi,H}(x),\varepsilon)}H(x,p)>H(x,\mathbf{p}^\#_{\phi,H}(x))+\delta.
	\end{align*}
	This contradicts the assumption that $\lim_{i\to\infty}H_{k_i}(x_{k_i},\mathbf{p}^\#_{\phi_{k_i},H_{k_i}}(x_{k_i}))=H(x,\mathbf{p}^\#_{\phi,H}(x))$. In sum, we have that
	\begin{align*}
		\limsup_{k\to\infty}|\mathbf{p}^\#_{\phi_{k},H_{k}}(x_{k})-\mathbf{p}^\#_{\phi,H}(x)|\leqslant\varepsilon.
	\end{align*}
	Our conclusion follows.
\end{proof}

\begin{The}[Stability of strict singular characteristics]\label{thm:stability}
	Let $\{H_k\}$ be a sequence of Hamiltonians satisfying \text{\rm (H1)-(H3)}, $\{\phi_k\}$ be a sequence of $\omega$-semiconcave functions on $M$, and $\gamma_k:\R\to M$, $k\in\N$ be a sequence of strict singular characteristics for the pair $(\phi_k,H_k)$. We suppose the following condition:
	\begin{enumerate}[\rm (i)]
		\item $\phi_k$ converges to $\phi$ uniformly on $M$,
		\item $H_k$ converges to a Hamiltonian $H$ satisfying \text{\rm (H1)-(H3)} uniformly on compact subset,
		\item $\gamma_k$ converges to $\gamma:\R\to M$ uniformly on compact subset.
	\end{enumerate}
	Then $\gamma$ is a strict singular characteristic for the pair $(\phi,H)$.
	
	Moreover, there exists a subsequence of strict singular characteristics $\{\gamma_{k_i}\}$ such that
	\begin{equation}\label{eq:P-converge}
		\lim_{i\to\infty}\mathbf{p}^\#_{\phi_{k_i},H_{k_i}}(\gamma_{k_i}(t))=\mathbf{p}^\#_{\phi,H}(\gamma(t)),\qquad a.e.\ t\in \R.
	\end{equation}
\end{The}

%
%

\begin{proof}
	Fix any $T>0$. By \eqref{eq:msc_H_phi2}, any strict singular characteristic $\gamma_k$ satisfies
	\begin{align*}
		\phi_k(\gamma_k(T))-\phi_k(\gamma_k(-T))=\int^T_{-T}\Big\{L_k(\gamma_k(s),\dot{\gamma}_k(s))+H_k(\gamma_k(s),\mathbf{p}^{\#}_{\phi_k,H_k}(\gamma_k(s)))\Big\}\ ds.
	\end{align*}
	Since $\gamma_k$ converges to $\gamma$ uniformly on $[-T,T]$, $\dot{\gamma}_k$ weakly converges to $\dot{\gamma}$ in $L^1$, by a standard lower semicontinuity result (see \cite[Theorem 3.6]{Buttazzo_Giaquinta_Hildebrandt_book} or \cite[Section 3.4]{Buttazzo_book}) we conclude
	\begin{equation}\label{eq:L2}
		\liminf_{k\to\infty}\int^T_{-T}L_k(\gamma_k(s),\dot{\gamma}_k(s))\ ds\geqslant \int^T_{-T}L(\gamma(s),\dot{\gamma}(s))\ ds.
	\end{equation}
	Invoking Fatou lemma and Lemma \ref{lem:lsc1}, we conclude that
	\begin{equation}\label{eq:H2}
		\liminf_{k\to\infty}\int^T_{-T}H_k(\gamma_k(s),\mathbf{p}^{\#}_{\phi_k,H_k}(\gamma_k(s)))\ ds\geqslant\int^T_{-T}H(\gamma(s),\mathbf{p}^{\#}_{\phi,H}(\gamma(s)))\ ds.
	\end{equation}
	Summing up \eqref{eq:L2} and \eqref{eq:H2} we have that
	\begin{align*}
		&\,\phi(\gamma(T))-\phi(\gamma(-T))=\lim_{k\to\infty}\phi_k(\gamma_k(T))-\phi_k(\gamma_k(0))\\
		=&\,\lim_{k\to\infty}\int^T_{-T}\Big\{L_k(\gamma_k(s),\dot{\gamma}_k(s))+H_k(\gamma_k(s),\mathbf{p}^{\#}_{\phi_k,H_k}(\gamma_k(s)))\Big\}\ ds\\
		\geqslant&\,\liminf_{k\to\infty}\int^T_{-T}L_k(\gamma_k(s),\dot{\gamma}_k(s))\ ds+\liminf_{k\to\infty}\int^T_{-T}H_k(\gamma_k(s),\mathbf{p}^{\#}_{\phi_k,H_k}(\gamma_k(s)))\ ds\\
		\geqslant&\,\int^T_{-T}\Big\{L(\gamma(s),\dot{\gamma}(s))+H(\gamma(s),\mathbf{p}^{\#}_{\phi,H}(\gamma(s)))\Big\}\ ds\\
		\geqslant&\,\phi(\gamma(T))-\phi(\gamma(-T)).
	\end{align*}
	It follows that, for any $T>0$,
	\begin{align*}
		\phi(\gamma(T))-\phi(\gamma(-T))=\int^T_{-T}\Big\{L(\gamma(s),\dot{\gamma}(s))+H(\gamma(s),\mathbf{p}^{\#}_{\phi,H}(\gamma(s)))\Big\}\ ds,
	\end{align*}
	which implies $\gamma$ is a strict singular characteristic for the pair $(\phi,H)$.
	
	Finally, we turn to prove \eqref{eq:P-converge}. By \eqref{eq:L2}, \eqref{eq:H2} and the equality
	\begin{align*}
		&\,\lim_{k\to\infty}\int^T_{-T}\Big\{L_k(\gamma_k(s),\dot{\gamma}_k(s))+H_k(\gamma_k(s),\mathbf{p}^\#_{\phi_k,H_k}(\gamma_k(s)))\Big\}\ ds\\
		=&\,\int^T_{-T}\Big\{L(\gamma(s),\dot{\gamma}(s))+H(\gamma(s),\mathbf{p}^\#_{\phi,H}(\gamma(s)))\Big\}\ ds
	\end{align*}
	we find
	\begin{align*}
		\lim_{k\to\infty}\int^T_{-T}H_k(\gamma_k(s),\mathbf{p}^\#_{\phi_k,H_k}(\gamma_k(s)))\ ds=\int^T_{-T}H(\gamma(s),\mathbf{p}^\#_{\phi,H}(\gamma(s)))\ ds.
	\end{align*}
	By Lemma \ref{lem:lsc1} we conclude that
	\begin{align*}
		\lim_{k\to\infty}\min\Big\{0,H_k(\gamma_k(s),\mathbf{p}^{\#}_{\phi_k,H_k}(\gamma_k(s)))-H(\gamma(s),\mathbf{p}^{\#}_{\phi,H}(\gamma(s)))\Big\}=0,\qquad\forall s\in[-T,T].
	\end{align*}
	Then, Lebesgue's theorem and the above two identities yield
	\begin{align*}
		&\,\lim_{k\to\infty}\int^T_{-T}\Big|H_k(\gamma_k(s),\mathbf{p}^{\#}_{\phi_k,H_k}(\gamma_k(s)))-H(\gamma(s),\mathbf{p}^{\#}_{\phi,H}(\gamma(s)))\Big|\ ds\\
		=&\,\lim_{k\to\infty}\int^T_{-T}\Big\{H_k(\gamma_k(s),\mathbf{p}^{\#}_{\phi_k,H_k}(\gamma_k(s)))-H(\gamma(s),\mathbf{p}^{\#}_{\phi,H}(\gamma(s)))\Big\}\ ds\\
		&\,-2\lim_{k\to\infty}\int^T_{-T}\min\Big\{0,H_k(\gamma_k(s),\mathbf{p}^{\#}_{\phi_k,H_k}(\gamma_k(s)))-H(\gamma(s),\mathbf{p}^{\#}_{\phi,H}(\gamma(s)))\Big\}\ ds\\
		=&\,\lim_{k\to\infty}\int^T_{-T}\Big\{H_k(\gamma_k(s),\mathbf{p}^{\#}_{\phi_k,H_k}(\gamma_k(s)))-H(\gamma(s),\mathbf{p}^{\#}_{\phi,H}(\gamma(s)))\Big\}\ ds\\
		=&\,0.
	\end{align*}
	Therefore, there exists a subsequence $\{\gamma_{k_i}\}$ such that
	\begin{align*}
		\lim_{i\to\infty}H_{k_i}(\gamma_{k_i}(s),\mathbf{p}^{\#}_{\phi_{k_i},H_{k_i}}(\gamma_{k_i}(s)))=H(\gamma(s),\mathbf{p}^{\#}_{\phi,H}(\gamma(s))),\qquad a.e.\ s\in[-T,T].
	\end{align*}
	By Lemma \ref{lem:P_k_P}, we conclude that
	\begin{align*}
		\lim_{i\to\infty}\mathbf{p}^\#_{\phi_{k_i},H_{k_i}}(\gamma_{k_i}(s))=\mathbf{p}^\#_{\phi,H}(\gamma(s)),\qquad a.e.\ s\in[-T,T].
	\end{align*}
    Thus, \eqref{eq:P-converge} follows by the arbitrariness of $T$.
\end{proof}

\begin{Cor}\label{cor:stability}
	Let $\phi$ be a semiconcave function on $M$, $H$ be a Hamiltonian satisfying \text{\rm (H1)-(H3)}. Then the family of strict singular characteristics $\gamma:\R\to M$ for the pair $(\phi,H)$ is a closed subset of $C(\R,M)$ under the topology of uniform convergence on compact subset.	
\end{Cor}

\subsection{Existence of strict singular characteristics: general cases}
Now we will prove the existence of strict singular characteristics in the general case by an approximation method firstly introduced in \cite{Yu2006,Cannarsa_Yu2009} for semiconcave functions with linear modulus.

Let us recall the approximation method used in the paper \cite{Cannarsa_Yu2009}. It is known that for any semiconcave function $\phi$ on $M$ with Lipschitz constant $L_0$, semiconcavity constant $C_0$ and any fixed $x\in M$, there exists a sequence of smooth functions $\{\phi_k\}\subset C^{\infty}(M)$ such that:
\begin{enumerate}[(i)]
	\item Each $\|D\phi_k\|_{C^0}\leqslant L_0$ and $D^2\phi_k$ is bounded above by $C_0I$ uniformly.
	\item $\phi_k$ converges to $\phi$ uniformly on $M$ as $k\to\infty$.
	\item $\lim_{k\to\infty}D\phi_k(x)=\mathbf{p}^{\#}_{\phi,H}(x)$ for given fixed $x\in M$. 
\end{enumerate}
Consider the differential equation
\begin{equation}\label{eq:SSC_k}
	\begin{cases}
		\dot{\gamma}(t)=H_p(\gamma(t),D\phi_k(\gamma(t))),\qquad t\in[0,\infty),\\
		\gamma(0)=x,
	\end{cases}
\end{equation}
and denote by $\gamma_k$ the unique solution of \eqref{eq:SSC_k}. The family $\{\gamma_k\}$ is equi-Lipschitz since $D\phi_k$ is uniformly bounded. Invoking the Ascoli-Arzel\`a theorem, by taking a subsequence if necessary, we can suppose that $\gamma_k$ converges to a Lipschitz curve $\gamma:[0,\infty)\to M$ uniformly on all compact subsets. In \cite{Cannarsa_Yu2009}, it is proved that there exists such a limiting curve $\gamma$ which is a generalized characteristic and it propagates singularities locally. Moreover, such a generalized characteristic satisfies the property that $\dot{\gamma}^+(0)$ exists and
\begin{equation}\label{eq:right_cont}
	\lim_{t\to0^+}\operatorname*{ess\ sup}_{s\in[0,t]}|\dot{\gamma}(s)-\dot{\gamma}^+(0)|=0.
\end{equation}
From Theorem \ref{thm:stability} we conclude that $\gamma$ is indeed a strict singular characteristic which means there is no convex hull in the right side of the generalized characteristic differential inclusion \eqref{eq:GC}. 

To discuss strict singular characteristics for general semiconcave functions using the approximation method, we need an approximation lemma for general semiconcave functions (see \cite[Lemma 2.1]{Cannarsa_Yu2009} for linear modulus).

\begin{Lem}\label{lem:appr_omega}
	Let $U,V\subset\R^n$ be bounded open convex set, $\overline{U}\subset V$. Suppose $\phi:V\to\R$ is an $\omega$-semiconcave function with Lipschitz constant $L_0>0$. Then, there exists a sequence $\{\phi_k\}\subset C^{\infty}(U,\R)$ such that
	\begin{enumerate}[\rm (i)]
		\item $\phi_k$ converges to $\phi$ uniformly on $U$;
		\item all functions $\phi_k$ are equi-Lipschitz with Lipschitz constant $L_0$;
		\item all functions $\phi_k$ are $\omega$-semiconcave.
	\end{enumerate}
	
	For any $\omega$-semiconcave function $\phi$ on a compact manifold $M$ with Lipschitz constant $L_0>0$, there exists a sequence $\{\phi_k\}\subset C^{\infty}(M,\R)$, with $\phi_k$ $\omega$-semiconcave and equi-Lipschitz of constant $L_0$, such that $\phi_k$ converges to $\phi$ uniformly on $M$.
\end{Lem}

\begin{proof}
	We let $\varepsilon=\text{\rm dist}\,(\overline{U},\R^n\setminus V)>0$. Taking a mollifier $\eta\in C^\infty(\R^n,\R)$ satisfying
	\begin{align*}
		0\leqslant\eta\leqslant1,\qquad \eta\vert_{B(0,1)}=1,\qquad \eta\vert_{\R^n\setminus B(0,2)}=0,\qquad \int_{\R^n}\eta(x)\ dx=1,
	\end{align*}
	for each $k\in\N$ we define $\phi_k:U\to\R$ by
	\begin{align*}
		\phi_k(x)=\Big(\frac{2k}{\varepsilon}\Big)^n\int_{\R^n}\phi(x-y)\cdot\eta\Big(\frac{2k}{\varepsilon}y\Big)\ dy,\qquad x\in U.
	\end{align*} 
	Items (i) and (ii) follows from the properties of mollifiers. Now, we turn to the proof of (iii). 
	
	For any $k\in\N$, $x_1,x_2\in U$ and $\lambda\in[0,1]$ we have that
	\begin{align*}
		&\,\lambda\phi_k(x_1)+(1-\lambda)\phi_k(x_2)-\phi_k(\lambda x_1+(1-\lambda)x_2)\\
		=&\,\Big(\frac{2k}{\varepsilon}\Big)^n\int_{\R^n}[\lambda\phi(x_1-y)+(1-\lambda)\phi(x_2-y)-\phi(\lambda x_1+(1-\lambda)x_2-y)]\cdot\eta\Big(\frac{2k}{\varepsilon}y\Big)\ dy\\
		\leqslant&\,\Big(\frac{2k}{\varepsilon}\Big)^n\int_{\R^n}\lambda(1-\lambda)|x_1-x_2|\omega(|x_1-x_2|)\cdot\eta\Big(\frac{2k}{\varepsilon}y\Big)\ dy\\
		=&\,\lambda(1-\lambda)|x_1-x_2|\omega(|x_1-x_2|).
	\end{align*}
	Taking $k\to\infty$, this completes the proof of (iii). The last assertion is a consequence of the facts we have shown in Euclidean case, since $M$ is compact.
\end{proof}

\begin{The}\label{thm:msc exs}
	Let $\phi$ be a semiconcave function on $M$, $H$ be a Hamiltonian satisfying \text{\rm (H1)-(H3)}. Then for any $x\in M$, there exists a strict singular characteristic $\gamma:\R\to M$ with $\gamma(0)=x$ for the pair $(\phi,H)$. In other words, 
	\begin{equation}\tag{SC}
		\begin{cases}
			\dot{\gamma}(t)=H_p(\gamma(t),\mathbf{p}^{\#}_{\phi,H}(\gamma(t))),\qquad a.e. \ t\in\R,\\
			\gamma(0)=x,
		\end{cases}
	\end{equation}
admits a Lipschitz solution.
\end{The}

\begin{proof}
By Lemma \ref{lem:appr_omega}, there exists a sequence of smooth functions $\{\phi_k\}$ approximating $\phi$ on $M$ satisfying (i)-(iii). For every $k\in\N$, we let $\gamma_k:\R\to M$ be the classical characteristic \eqref{eq:SSC_k} with $\gamma_k(0)=x$ for the pair $(\phi_k,H)$. Without loss of generality, we suppose that $\gamma_k$ converges to $\gamma$ uniformly on all compact subsets. From Theorem \ref{thm:stability}, the stability of strict singular characteristics, we conclude $\gamma:\R\to M$ is also a strict singular characteristic with $\gamma(0)=x$ for the pair $(\phi,H)$. This completes our proof.
\end{proof}


In fact, all the results in this paper have a similar statement for semiconvex functions, by almost the same method of proof. 


\subsection{Broken characteristics}

This section is devoted to clarify the notion of strict singular characteristics and the broken characteristics determined by \eqref{eq:SGC2}.

\begin{Lem}[\cite{Cheng_Hong2023} Lemma 2.6]\label{lem:epf}
	Suppose $\phi$ is a semiconcave function on $M$ and $H$ is a Hamiltonian satisfying conditions {\rm (H1)-(H3)}. Then for any $x\in M$ and $p_0\in D^+\phi(x)$, we have $p_0=\mathbf{p}^\#_{\phi,H}(x)$ if and only if
	\begin{align*}
		\langle p_0,H_p(x,p_0)\rangle=\,\min_{p\in D^+\phi(x)}\langle p,H_p(x,p_0)\rangle.
	\end{align*}
\end{Lem}

\begin{Lem}\label{lem:directional}
Suppose $\phi$ is a semiconcave function on $M$ and $H:T^*M\to\R$ is a Hamiltonian satisfying conditions {\rm (H1)-(H3)}. Then, for any $x\in M$
\begin{equation}\label{eq:ineq_direction}
	L(x,v)+H(x,\mathbf{p}^\#_{\phi,H}(x))\geqslant\partial\phi(x,v),\qquad\forall v\in T_xM.
\end{equation} 
Moreover, the equality holds in \eqref{eq:ineq_direction} if and only if $v=H_p(x,\mathbf{p}^\#_{\phi,H}(x))$.
\end{Lem}

\begin{proof}
In view of Young's inequality and Proposition \ref{pro:scsv} (3), we obtain that for any $v\in T_xM$
\begin{align*}
	L(x,v)+H(x,\mathbf{p}^\#_{\phi,H}(x))\geqslant\langle\mathbf{p}^\#_{\phi,H}(x),v\rangle\geqslant\min_{p\in D^+\phi(x)}\langle p,v\rangle=\partial\phi(x,v).
\end{align*}
The first inequality above is an equality if and only if $v=H_p(x,\mathbf{p}^\#_{\phi,H}(x))$. By Lemma \ref{lem:epf}, the second inequality automatically is an equality if $v=H_p(x,\mathbf{p}^\#_{\phi,H}(x))$. Thus, the equality holds in \eqref{eq:ineq_direction} if and only if $v=H_p(x,\mathbf{p}^\#_{\phi,H}(x))$.
\end{proof}

\begin{The}\label{thm:right_derivative}
Suppose $\phi$ is a semiconcave function on $M$, $H$ is a Hamiltonian satisfying conditions {\rm (H1)-(H3)}, and $\gamma:[0,T]\to M$ is a strict singular characteristic for the pair $(\phi,H)$.
\begin{enumerate}[\rm (1)]
	\item The right derivative $\dot{\gamma}^+(t)$ exists for all $t\in[0,T)$ and
	\begin{equation}\tag{BC}
		\dot{\gamma}^+(t)=H_p(\gamma(t),\mathbf{p}^\#_{\phi,H}(\gamma(t))),\qquad \forall t\in[0,T).
	\end{equation}
	\item For all $t\in[0,T)$
	\begin{align*}
		\lim_{\tau\to t^+}\frac 1{\tau-t}\int^\tau_tH(\gamma(s),\mathbf{p}^\#_{\phi,H}(\gamma(s)))\ ds=H(\gamma(t),\mathbf{p}^\#_{\phi,H}(\gamma(t))).
	\end{align*}
\end{enumerate}
\end{The}

\begin{proof}
Since the local nature we work on Euclidean space. For any $t\in[0,T)$, due to the Lipschitz continuity of $\gamma$, the family 
\begin{align*}
	\Big\{\frac{\gamma(\tau)-\gamma(t)}{\tau-t}\Big\}_{\tau\in(t,T]}
\end{align*}
is uniformly bounded. Picking up any sequence $\tau_i\searrow0^+$ as $i\to\infty$ such that
\begin{equation}\label{eq:sub_lim}
	\lim_{i\to\infty}\frac{\gamma(\tau_i)-\gamma(t)}{\tau_i-t}=v\in\R^n,
\end{equation}
we have
\begin{equation}\label{eq:sub_lim2}
	\lim_{i\to\infty}\frac{\phi(\gamma(\tau_i))-\phi(\gamma(t))}{\tau_i-t}=\partial\phi(\gamma(t),v).
\end{equation}
Since $\gamma$ is a strict singular characteristic and $L(x,\cdot)$ is convex, for any $\tau\in(t,T]$
\begin{align*}
	&\,\phi(\gamma(\tau))-\phi(\gamma(t))\\
	=&\,\int^\tau_t\Big\{L(\gamma(s),\dot{\gamma}(s))+H(\gamma(s),\mathbf{p}^\#_{\phi,H}(\gamma(s)))\Big\}\ ds\\
	=&\,\int^\tau_t\Big\{L(\gamma(t),\dot{\gamma}(s))+[L(\gamma(s),\dot{\gamma}(s))-L(\gamma(t),\dot{\gamma}(s))]\\
	&\,\quad\quad\qquad +H(\gamma(t),\mathbf{p}^\#_{\phi,H}(\gamma(t)))+[H(\gamma(s),\mathbf{p}^\#_{\phi,H}(\gamma(s)))-H(\gamma(t),\mathbf{p}^\#_{\phi,H}(\gamma(t)))]\Big\}\ ds\\
	\geqslant&\,\int^\tau_t\Big\{L(\gamma(t),v)+\langle L_v(\gamma(t),v),\dot{\gamma}(s)-v\rangle+H(\gamma(t),\mathbf{p}^\#_{\phi,H}(\gamma(t)))\Big\}\ ds\\
	&\,\qquad+\int^\tau_t\Big\{L(\gamma(s),\dot{\gamma}(s))-L(\gamma(t),\dot{\gamma}(s))\Big\}\ ds+\int^\tau_t\Big\{H(\gamma(s),\mathbf{p}^\#_{\phi,H}(\gamma(s)))-H(\gamma(t),\mathbf{p}^\#_{\phi,H}(\gamma(t)))\Big\}\ ds\\
	=&\,\int^\tau_t\Big\{L(\gamma(t),v)+H(\gamma(t),\mathbf{p}^\#_{\phi,H}(\gamma(t)))\Big\}\ ds+\int^\tau_t\langle L_v(\gamma(t),v),\dot{\gamma}(s)-v\rangle\ ds\\
	&\,\qquad+\int^\tau_t\Big\{L(\gamma(s),\dot{\gamma}(s))-L(\gamma(t),\dot{\gamma}(s))\Big\}\ ds+\int^\tau_t\Big\{H(\gamma(s),\mathbf{p}^\#_{\phi,H}(\gamma(s)))-H(\gamma(t),\mathbf{p}^\#_{\phi,H}(\gamma(t)))\Big\}\ ds.
\end{align*}
Thank to the Lipschitz property of $L$ and $\gamma$,
\begin{equation}\label{eq:right1}
	\lim_{\tau\to t^+}\frac 1{\tau-t}\int^\tau_t|L(\gamma(s),\dot{\gamma}(s))-L(\gamma(t),\dot{\gamma}(s))|\ ds\leqslant\lim_{\tau\to t^+}\frac 1{\tau-t}\int^\tau_tC(\tau-t)\ ds=0.
\end{equation}
By Lemma \ref{lem:lsc1} we obtain that
\begin{equation}\label{eq:right2}
	\liminf_{\tau\to t^+}\frac 1{\tau-t}\int^\tau_t\Big\{H(\gamma(s),\mathbf{p}^\#_{\phi,H}(\gamma(s)))-H(\gamma(t),\mathbf{p}^\#_{\phi,H}(\gamma(t)))\Big\} ds\geqslant0,
\end{equation}
and \eqref{eq:sub_lim} yields
\begin{equation}\label{eq:right3}
	\lim_{i\to\infty}\frac 1{\tau_i-t}\int^{\tau_i}_t\langle L_v(\gamma(t),v),\dot{\gamma}(s)-v\rangle\ ds=\lim_{i\to\infty}\Big\langle L_v(\gamma(t),v),\frac{\gamma(\tau_i)-\gamma(t)}{\tau_i-t}-v\Big\rangle=0.
\end{equation}
Combining \eqref{eq:right1}, \eqref{eq:right2} and \eqref{eq:right3} we obtain
\begin{align*}
	\lim_{i\to\infty}\frac{\phi(\gamma(\tau_i))-\phi(\gamma(t))}{\tau_i-t}\geqslant L(\gamma(t),v)+H(\gamma(t),\mathbf{p}^\#_{\phi,H}(\gamma(t))).
\end{align*}
Together with \eqref{eq:sub_lim2} and Lemma \ref{lem:directional} it yields $v=H_p(\gamma(t),\mathbf{p}^\#_{\phi,H}(\gamma(t)))$, and 
\begin{align*}
	\dot{\gamma}^+(t)=\lim_{\tau\to t^+}\frac{\gamma(\tau)-\gamma(t)}{\tau-t}=H_p(\gamma(t),\mathbf{p}^\#_{\phi,H}(\gamma(t))).
\end{align*}
This completes the proof of (1).

The proof of (2) need the same reasoning as above. For any $t\in[0,T)$ we have that
\begin{align*}
	&\,L(\gamma(t),\dot{\gamma}^+(t))+H(\gamma(t),\mathbf{p}^\#_{\phi,H}(\gamma(t)))=\partial\phi(\gamma(t),\dot{\gamma}^+(t))\\
	=&\,\lim_{\tau\to t^+}\frac 1{\tau-t}(\phi(\gamma(\tau))-\phi(\gamma(t)))\\
	\geqslant&\,\limsup_{\tau\to t^+}\frac 1{\tau-t}\bigg\{\int^\tau_t\Big[L(\gamma(t),\dot{\gamma}^+(t))+H(\gamma(t),\mathbf{p}^\#_{\phi,H}(\gamma(t)))\Big]\ ds+\int^\tau_t\langle L_v(\gamma(t),\dot{\gamma}^+(t)),\dot{\gamma}(s)-\dot{\gamma}^+(t)\rangle\ ds\\
	&\,\qquad+\int^\tau_t\Big[L(\gamma(s),\dot{\gamma}(s))-L(\gamma(t),\dot{\gamma}(s))\Big]\ ds+\int^\tau_t\Big[H(\gamma(s),\mathbf{p}^\#_{\phi,H}(\gamma(s)))-H(\gamma(t),\mathbf{p}^\#_{\phi,H}(\gamma(t)))\Big]\ ds\bigg\}\\
	=&\,L(\gamma(t),\dot{\gamma}^+(t))+H(\gamma(t),\mathbf{p}^\#_{\phi,H}(\gamma(t)))\\
	&\qquad +\limsup_{\tau\to t^+}\frac 1{\tau-t}\int^\tau_t\big[H(\gamma(s),\mathbf{p}^\#_{\phi,H}(\gamma(s)))-H(\gamma(t),\mathbf{p}^\#_{\phi,H}(\gamma(t)))\big]\ ds,
\end{align*}
where we have used \eqref{eq:right1}, the convexity of $L(x,\cdot)$ and the fact that $\dot{\gamma}^+(t)=H_p(\gamma(t),\mathbf{p}^\#_{\phi,H}(\gamma(t)))$ from (1). It follows that
\begin{align*}
	\limsup_{\tau\to t^+}\frac 1{\tau-t}\int^\tau_t\big[H(\gamma(s),\mathbf{p}^\#_{\phi,H}(\gamma(s)))-H(\gamma(t),\mathbf{p}^\#_{\phi,H}(\gamma(t)))\big]\ ds\leqslant0.
\end{align*}
Together with Lemma \ref{lem:lsc1}, the above inequality leads to our conclusion. 
\end{proof}

\subsection{Weakly right-continuous curves in space of probability measures}

Theorem \ref{thm:right_derivative} implies, at least in the Euclidean case, that for any strict singular characteristic $\gamma$ the curve $\dot{\gamma}^+:[0,T)\to\R^n$ is well defined. Moreover, the relation
\begin{align*}
	\dot{\gamma}^+(t)=\lim_{\tau\to t^+}\frac 1{\tau-t}\int^\tau_t\dot{\gamma}^+(s)\ ds
\end{align*}
implies that each $t\in[0,T]$ is a density point of $\dot{\gamma}^+$ from the right-hand side. In fact $(\gamma,\mathbf{p}^\#_{\phi,H}(\gamma))$ and $(\gamma,\dot{\gamma}^+)$ determine two curves of probability measures on $T^*M$ and $TM$ respectively, and Theorem \ref{thm:right_derivative} implies that such curves have some weak right continuity property.

\begin{The}\label{thm:measure_H}
	Under the assumption of Theorem \ref{thm:right_derivative}, for any $t\in[0,T)$ we define probability measures $\mu^*_{[t,\tau]}$, $\tau\in(t,T]$, on $T^*M$ by
	\begin{align*}
		\int_{T^*M}f\ d\mu^*_{[t,\tau]}=\frac 1{\tau-t}\int^\tau_tf(\gamma(s),\mathbf{p}^\#_{\phi,H}(\gamma(s)))\ ds,\qquad\forall f\in C_c(T^*M,\R).
	\end{align*}
	Then $\mu^*_{[t,\tau]}$ weakly converges to $\delta^*_{\gamma(t)}:=\delta^*_{\gamma(t),\mathbf{p}^\#_{\phi,H}(\gamma(t))}$ as $\tau\to t^+$, where $\delta^*_{x,p}$ is the Dirac measure supported on $(x,p)\in T^*M$.
\end{The}

\begin{proof}
	We only prove in Euclidean case. We first claim for any $\varepsilon>0$, 
	\begin{equation}\label{eq:I}
		\lim_{\tau\to t^+}\frac{\mathscr{L}^1(I_{t,\tau,\varepsilon})}{\tau-t}=0
	\end{equation}
	where $\mathscr{L}^1$ stands for the Lebesgue measure on $\R$ and
	\begin{align*}
		I_{t,\tau,\varepsilon}:=\{s\in[t,\tau]: |(\gamma(s),\mathbf{p}^\#_{\phi,H}(\gamma(s)))-(\gamma(t),\mathbf{p}^\#_{\phi,H}(\gamma(t)))|>\varepsilon\}.
	\end{align*}
	In fact, by Lemma \ref{lem:P_k_P} for any $\varepsilon>0$ there exists $\varepsilon_1>0$ and $\delta>0$ such that if $y-\gamma(t)\leqslant\delta$ and $H(y,\mathbf{p}^\#_{\phi,H}(y))\leqslant H(\gamma(t),\mathbf{p}^\#_{\phi,H}(\gamma(t)))+\varepsilon_1$, then $|\mathbf{p}^\#_{\phi,H}(y)-\mathbf{p}^\#_{\phi,H}(\gamma(t))|\leqslant\varepsilon/2$. Denote by $C_0$ a Lipschitz constant of $\gamma$ and set $\delta_1=\min\{\frac{\delta}{C_0},\frac{\varepsilon}{2C_0}\}$. Thus, if $\tau\in(t,t+\delta_1)$ then
	\begin{align*}
		|\gamma(s)-\gamma(t)|\leqslant C_0(s-t)\leqslant C_0\delta_1\leqslant\min\big\{\delta,\frac{\varepsilon}{2}\big\},\qquad\forall s\in[t,\tau].
	\end{align*}
	In addition, if $H(\gamma(s),\mathbf{p}^\#_{\phi,H}(\gamma(s)))\leqslant H(\gamma(t),\mathbf{p}^\#_{\phi,H}(\gamma(t)))+\varepsilon_1$, then
	\begin{align*}
		&\,|(\gamma(s),\mathbf{p}^\#_{\phi,H}(\gamma(s)))-(\gamma(t),\mathbf{p}^\#_{\phi,H}(\gamma(t)))|\\
		\leqslant&\,|\gamma(s)-\gamma(t)|+|\mathbf{p}^\#_{\phi,H}(\gamma(s))-\mathbf{p}^\#_{\phi,H}(\gamma(t))|\leqslant\frac{\varepsilon}{2}+\frac{\varepsilon}{2}=\varepsilon.
	\end{align*}
	This implies that
	\begin{align*}
		&\,\{s\in[t,\tau]: H(\gamma(s),\mathbf{p}^\#_{\phi,H}(\gamma(s)))\leqslant H(\gamma(t),\mathbf{p}^\#_{\phi,H}(\gamma(t)))+\varepsilon_1\}\\
		\subset&\,\{s\in[t,\tau]: |(\gamma(s),\mathbf{p}^\#_{\phi,H}(\gamma(s)))-(\gamma(t),\mathbf{p}^\#_{\phi,H}(\gamma(t)))|\leqslant\varepsilon\}.
	\end{align*}
	Equivalently,
	\begin{align*}
		J_{t,\tau,\varepsilon_1}:=\{s\in[t,\tau]: H(\gamma(s),\mathbf{p}^\#_{\phi,H}(\gamma(s)))>H(\gamma(t),\mathbf{p}^\#_{\phi,H}(\gamma(t)))+\varepsilon_1\}\supset I_{t,\tau,\varepsilon}.
	\end{align*}
	Together with Lemma \ref{lem:lsc1} and Theorem \ref{thm:right_derivative}, we obtain
	\begin{align*}
		&\,H(\gamma(t),\mathbf{p}^\#_{\phi,H}(\gamma(t)))=\lim_{\tau\to t^+}\frac 1{\tau-t}\int^\tau_tH(\gamma(s),\mathbf{p}^\#_{\phi,H}(\gamma(s)))\ ds\quad(\text{Theorem \ref{thm:right_derivative}})\\
		=&\,\lim_{\tau\to t^+}\Big\{\frac 1{\tau-t}\int_{J_{t,\tau,\varepsilon_1}}H(\gamma(s),\mathbf{p}^\#_{\phi,H}(\gamma(s))\ ds+\frac 1{\tau-t}\int_{[t,\tau]\setminus J_{t,\tau,\varepsilon_1}}H(\gamma(s),\mathbf{p}^\#_{\phi,H}(\gamma(s))\ ds\Big\}\\
		\geqslant&\,\limsup_{\tau\to t^+}\Big\{\frac {\mathscr{L}^1(J_{t,\tau,\varepsilon_1})}{\tau-t}(H(\gamma(t),\mathbf{p}^\#_{\phi,H}(\gamma(t)))+\varepsilon_1)+\Big(1-\frac {\mathscr{L}^1(J_{t,\tau,\varepsilon_1})}{\tau-t}\Big)\inf_{s\in[t,\tau]}H(\gamma(s),\mathbf{p}^\#_{\phi,H}(\gamma(s)))\Big\} \\
		=&\,\limsup_{\tau\to t^+}\Big\{\frac {\mathscr{L}^1(J_{t,\tau,\varepsilon_1})}{\tau-t}\big(H(\gamma(t),\mathbf{p}^\#_{\phi,H}(\gamma(t)))-\inf_{s\in[t,\tau]}H(\gamma(s),\mathbf{p}^\#_{\phi,H}(\gamma(s)))\big)\\
		&\qquad\qquad\qquad\qquad\qquad\qquad\qquad +\frac {\mathscr{L}^1(J_{t,\tau,\varepsilon_1})}{\tau-t}\cdot\varepsilon_1+\inf_{s\in[t,\tau]}H(\gamma(s),\mathbf{p}^\#_{\phi,H}(\gamma(s)))\Big\} \\
		\geqslant&\,\limsup_{\tau\to t^+}\frac {\mathscr{L}^1(J_{t,\tau,\varepsilon_1})}{\tau-t}\cdot\varepsilon_1+\lim_{\tau\to t^+}\inf_{s\in[t,\tau]}H(\gamma(s),\mathbf{p}^\#_{\phi,H}(\gamma(s))) \\
		=&\,\limsup_{\tau\to t^+}\frac {\mathscr{L}^1(J_{t,\tau,\varepsilon_1})}{\tau-t}\cdot\varepsilon_1+H(\gamma(t),\mathbf{p}^\#_{\phi,H}(\gamma(t)))\qquad(\text{Lemma \ref{lem:lsc1}})
	\end{align*}
	It follows
	\begin{align*}
		\limsup_{\tau\to t^+}\frac{\mathscr{L}^1(I_{t,\tau,\varepsilon})}{\tau-t}\leqslant\limsup_{\tau\to t^+}\frac{\mathscr{L}^1(J_{t,\tau,\varepsilon_1})}{\tau-t}\leqslant0.
	\end{align*}
	This leads to \eqref{eq:I}. 
	
	Now we note that there exists a compact subset $K\subset T^*M$ such that
	\begin{align*}
		(x,\mathbf{p}^\#_{\phi,H}(x))\in K,\qquad\forall x\in M.
	\end{align*}
	For any $f\in C_c(T^*M,\R)$ we set $M_f:=\max_{(x,p)\in K}f(x,p)$. Given $\varepsilon>0$, there exists $\delta>0$ such that 
	\begin{align*}
		|(x',p')-(x,p)|<\delta\qquad\Longrightarrow\qquad |f(x',p')-f(x,p)|<\varepsilon.
	\end{align*}
	By \eqref{eq:I} we conclude that
	\begin{align*}
		&\,\limsup_{\tau\to t^+}\Big|\int_{T^*M}f\ d\mu^*_{[t,\tau]}-\int_{T^*M}f\ d\delta^*_{\gamma(t)}\Big|\\
		\leqslant&\,\limsup_{\tau\to t^+}\frac 1{\tau-t}\int^\tau_t|f(\gamma(s),\mathbf{p}^\#_{\phi,H}(\gamma(s)))-f(\gamma(t),\mathbf{p}^\#_{\phi,H}(\gamma(t)))|\ ds\\
		\leqslant&\,\limsup_{\tau\to t^+}\frac 1{\tau-t}\cdot\mathscr{L}^1(I_{t,\tau,\delta})\cdot 2M_f+\varepsilon=\varepsilon.
	\end{align*}
	Since $\varepsilon>0$ is arbitrary, we have that
	\begin{align*}
		\lim_{\tau\to t^+}\int_{T^*M}f\ d\mu^*_{[t,\tau]}=\int_{T^*M}f\ d\delta^*_{\gamma(t)},\qquad\forall f\in C_c(T^*M,\R).
	\end{align*}
	This means $\mu^*_{[t,\tau]}$ weakly converges to $\delta^*_{\gamma(t)}$ as $\tau\to t^+$.
\end{proof}

Recall that the Legendre transform $(x,p)\mapsto(x,H_p(x,p))$ is a $C^1$-diffeomorphism from $T^*M$ to $TM$. We have an equivalent form on Theorem \ref{thm:measure_H} for measures on $TM$.

\begin{The}\label{thm:measure_L}
	Under the assumption of Theorem \ref{thm:right_derivative}, for any $t\in[0,T)$ we define probability measures $\mu_{[t,\tau]}$, $\tau\in(t,T]$, on $TM$ by
	\begin{align*}
		\int_{TM}f\ d\mu_{[t,\tau]}=\frac 1{\tau-t}\int^\tau_tf(\gamma(s),\dot{\gamma}^+(s))\ ds,\qquad\forall f\in C_c(TM,\R).
	\end{align*}
	Then $\mu_{[t,\tau]}$ weakly converges to $\delta_{\gamma(t)}:=\delta_{\gamma(t),\dot{\gamma}^+(t)}$ as $\tau\to t^+$, where $\delta_{x,v}$ is the Dirac measure supported on $(x,v)\in TM$.
\end{The}

\subsection{Energy evolution along strict singular characteristics} 
In this section, we estimate the energy evolution along strict singular characteristics for any $\phi\in\text{\rm SCL}\,(M)$ and $H$ a Tonelli Hamiltonian. We obtain a family of strict singular characteristics with finite rate of energy increase, which leads to the right continuity of right derivative.

\begin{Lem}\label{lem:lambda}
Under the assumptions of Theorem \ref{thm:stability}, if there exists $\lambda>0$ such that
\begin{align*}
	H_k(\gamma_k(t_2),\mathbf{p}^\#_{\phi_k,H_k}(\gamma_k(t_2)))-H_k(\gamma_k(t_1),\mathbf{p}^\#_{\phi_k,H_k}(\gamma_k(t_1)))\leqslant\lambda(t_2-t_1),\qquad\forall t_1\leqslant t_2, k\in\N,
\end{align*}
then
\begin{align*}
	H(\gamma(t_2),\mathbf{p}^\#_{\phi,H}(\gamma(t_2)))-H(\gamma(t_1),\mathbf{p}^\#_{\phi,H}(\gamma(t_1)))\leqslant\lambda(t_2-t_1),\qquad\forall t_1\leqslant t_2.
\end{align*}
\end{Lem}

\begin{proof}
In view of Theorem \ref{thm:stability}, we suppose
\begin{align*}
	\lim_{k\to\infty}\mathbf{p}^\#_{\phi_k,H_k}(\gamma_k(s))=\mathbf{p}^\#_{\phi,H}(\gamma(s)),\qquad a.e.\ s\in\R,
\end{align*}
without loss of generality. Then, by Lemma \ref{lem:lsc1}, for any $t_2\in\R$ we have
\begin{align*}
	&\,H(\gamma(t_2),\mathbf{p}^\#_{\phi,H}(\gamma(t_2))-H(\gamma(s),\mathbf{p}^\#_{\phi,H}(\gamma(s))\\
	\leqslant&\,\liminf_{k\to\infty}H_k(\gamma_k(t_2),\mathbf{p}^\#_{\phi_k,H_k}(\gamma_k(t_2)))-\lim_{k\to\infty}H_k(\gamma_k(s),\mathbf{p}^\#_{\phi_k,H_k}(\gamma_k(s)))\\
	\leqslant&\,\lambda(t_2-s),\qquad a.e.\ s\in(-\infty,t_2].
\end{align*}
The above estimate, together with Theorem \ref{thm:right_derivative} (2) implies that for any $t_1\in(-\infty,t_2)$
\begin{align*}
	H(\gamma(t_1),\mathbf{p}^\#_{\phi,H}(\gamma(t_1))=&\,\lim_{t\to t_1^+}\frac 1{t-t_1}\int^t_{t_1}H(\gamma(s),\mathbf{p}^\#_{\phi,H}(\gamma(s))\ ds\\
	\geqslant&\,\lim_{t\to t_1^+}\frac 1{t-t_1}\int^t_{t_1}\big[H(\gamma(t_2),\mathbf{p}^\#_{\phi,H}(\gamma(t_2))-\lambda(t_2-s)\big]\ ds\\
	=&\,H(\gamma(t_2),\mathbf{p}^\#_{\phi,H}(\gamma(t_2))-\lambda(t_2-t_1).
\end{align*}
This completes the proof.
\end{proof}

\begin{The}\label{thm:right_continuous2}
Let $\phi\in\text{\rm SCL}\,(M)$ and let $H:T^*M\to\R$ be a Tonelli Hamiltonian. Then, there exists $\lambda=\lambda_{\phi,H}>0$ such that the following property is satisfied: for any $x\in M$ there exists a strict singular characteristic $\gamma:\R\to M$ for the pair $(\phi,H)$ with $x=\gamma(0)$ satisfying
\begin{equation}\label{eq:lambda_phi_H}
	H(\gamma(t_2),\mathbf{p}^\#_{\phi,H}(\gamma(t_2)))-H(\gamma(t_1),\mathbf{p}^\#_{\phi,H}(\gamma(t_1)))\leqslant\lambda(t_2-t_1),\qquad\forall t_1\leqslant t_2.
\end{equation} 
Moreover, there holds
\begin{equation}\label{eq:right_cont2}
	\lim_{s\to t^+}\mathbf{p}^\#_{\phi,H}(\gamma(s))=\mathbf{p}^\#_{\phi,H}(\gamma(t)),\qquad\forall t\in\R.
\end{equation} 
Furthermore, $\dot{\gamma}^+$ is right-continuous on $\R$.
\end{The}

\begin{proof}
Recall the method of approximation used in the paper \cite{Cannarsa_Yu2009}. There exists a sequence of smooth functions $\{\phi_k\}\subset C^{\infty}(M)$ such that 
\begin{enumerate}[(i)]
	\item Each $\|D\phi_k\|_{C^0}\leqslant L_0=\text{Lip}\,(\phi)$ and $D^2\phi_k$ is uniformly bounded above by $C_0I$ where $C_0$ is the semiconcavity constant of $\phi$.
	\item $\phi_k$ converges to $\phi$ uniformly on $M$ as $k\to\infty$.
\end{enumerate}
Consider the following differential equation 
\begin{align*}
	\begin{cases}
		\dot{\gamma}(t)=H_p(\gamma(t),D\phi_k(\gamma(t))),\qquad t\in\R,\\
		\gamma(0)=x,
	\end{cases}
\end{align*}
and denote by $\gamma_k$ its unique solution for each $k\in\N$. To estimate the energy evolution along $\gamma_k$, for any $t\in\R$ and $k\in\N$ we have that
\begin{align*}
	\frac d{dt}H(\gamma_k(t),D\phi_k(\gamma_k(t)))=&\,H_x\cdot H_p+H_p\cdot D^2\phi_k\cdot H_p\\
	\leqslant&\, C^2+C_0C^2:=\lambda=\lambda_{\phi,H}
\end{align*}
where $C=\sup_{x\in M,|p|_x\leqslant L_0}|DH(x,p)|$. This implies that
\begin{align*}
	H(\gamma_k(t_2),\mathbf{p}^\#_{\phi_k,H}(\gamma_k(t_2)))-H(\gamma_k(t_1),\mathbf{p}^\#_{\phi_k,H}(\gamma_k(t_1)))\leqslant\lambda(t_2-t_1),\qquad\forall t_1\leqslant t_2, k\in\N.
\end{align*}
Without loss of generality, we suppose that $\gamma_k$ converges to $\gamma$ uniformly on all compact subsets as $k\to\infty$. By Theorem \ref{thm:stability} and Lemma \ref{lem:lambda}, $\gamma$ is a strict generalized characteristic on $\R$ for the pair $(\phi,H)$ with $\gamma(0)=x$ and \eqref{eq:lambda_phi_H} holds. Together with Lemma \ref{lem:lsc1}, this yields 
\begin{align*}
	\lim_{s\to t^+}H(\gamma(s),\mathbf{p}^\#_{\phi,H}(\gamma(s)))=H(\gamma(t),\mathbf{p}^\#_{\phi,H}(\gamma(t))),\qquad\forall t\in\R.
\end{align*}
Now, \eqref{eq:right_cont2} holds by Lemma \ref{lem:P_k_P}, and the last assertion is a consequence of Theorem \ref{thm:right_derivative}.
\end{proof}

\begin{Cor}\label{cor:stability_H}
	Let $\phi\in\text{\rm SCL}\,(M)$ and let $H:T^*M\to\R$ be a Tonelli Hamiltonian. Then the family of strict singular characteristics $\gamma:\R\to M$ for the pair $(\phi,H)$ satisfying \eqref{eq:lambda_phi_H} as in Theorem \ref{thm:right_continuous2} is a closed subset of $C(\R,M)$ under the topology of uniform convergence on compact subsets.	
\end{Cor}

\subsection{New weak KAM aspects of maximal slope curves}

For any $\phi\in\text{SCL}\,(M)$ and Tonelli Hamiltonian $H:T^*M\to\R$ with associated Tonelli Lagrangian $L:TM\to\R$, inspired by the maximal slope curves for the pair $(\phi,H)$, we introduce a new Lagrangian
\begin{align*}
	L^{\#}_{\phi}(x,v):=L(x,v)+H(x,\mathbf{p}^{\#}_{\phi,H}(x)),\quad x\in M, v\in T_xM.
\end{align*}
Let $H^{\#}_{\phi}$ be the associated Hamiltonian of $L^{\#}_{\phi}$, more precisely,
\begin{align*}
	H^{\#}_{\phi}(x,p)=\sup_{v\in T_xM}\{\langle p,v\rangle-L^{\#}_{\phi}(x,v)\}=H(x,p)-H(x,\mathbf{p}^{\#}_{\phi,H}(x)),\quad x\in M, p\in T^*_xM.
\end{align*}
Observe that the Lagrangian $L^{\#}_{\phi}$ (resp. the Hamiltoninan $H^{\#}_{\phi}$) is strictly convex and superlinear with respect to the $v$-variable (resp. $p$-variable) but only lower-semicontinuous (resp. upper-semicontinuous) in $x$-variable. We define 
\begin{align*}
	A^{\#}_t(x,y)=\inf_{\xi\in\Gamma^t_{x,y}}\int^t_0L^{\#}_{\phi}(\xi(s),\dot{\xi}(s))\ ds,\qquad t>0,\ x,y\in M,
\end{align*}
the fundamental solution associated with $L^{\#}_{\phi}$. Due to the lower semicotinuity of $L^{\#}_{\phi}$, the above infimum is achieved and is Lipschitz by a classical result in \cite{Dal_Maso_Frankowska_2003}. 

We introduce the Lax-Oleinik operators with respect to this new Hamiltonian: for any continuous function $u:M\to\R$
\begin{align*}
	T^{\#}_t u(x)=\,\inf_{y\in M}\{u(y)+A^{\#}_t(y,x)\},\qquad \breve{T}^{\#}_t u(x)=\,\sup_{y\in M}\{u(y)-A^{\#}_t(x,y)\},\qquad t>0,x\in M.
\end{align*}

\begin{defn}
We call $u$ a weak KAM solution of negative (resp. positive) type of the Hamilton-Jacobi equation
\begin{equation}\label{eq:sharp-}
		H^\#_\phi(x,Du(x))=0,\qquad x\in M,
\end{equation}
if $T^{\#}_tu=u$ for any $t\geqslant0$ (resp. $\breve{T}^\#_tu=u$) for any $t\geqslant0$.
\end{defn}

\begin{The}\label{thm:new_WKM}
Suppose $\phi\in\text{\rm SCL}\,(M)$, $H:T^*M\to\R$ is a Tonelli Hamiltonian and $H^\#_\phi$ is defined as above.
\begin{enumerate}[\rm (1)]
	\item $\phi$ is weak KAM solution of both negative and positive type of \eqref{eq:sharp-}.
	\item $\phi$ is a viscosity solution\footnote{Here we use the standard definition of viscosity (sub)solutions even for discontinuous Hamiltonians.} of the Hamilton-Jacobi equation
	\begin{equation}\label{eq:sharp+}
		-H^\#_\phi(x,Du(x))=0,\qquad x\in M.
	\end{equation}
	\item If $\text{\rm Sing}\,(\phi)\not=\varnothing$, then $\phi$ is not a viscosity subsolution of \eqref{eq:sharp-}.
\end{enumerate}	
\end{The}

\begin{proof}
The proof of (1) is direct from \eqref{eq:msc_H_phi2} and Theorem \ref{thm:msc exs}. Indeed, for any locally absolutely continuous curve $\gamma:\R\to M$, the following inequality holds by Proposition \ref{pro:msc} 
\begin{align*}
	\phi(\gamma(t_2))-\phi(\gamma(t_1))\leqslant\int_{t_1}^{t_2}L(\gamma(s),\dot{\gamma}(s))+H(\gamma(s),\mathbf{p}^\#_{\phi,H}(\gamma(s)))\ ds,\qquad\forall t_1,t_2\in\R,\ t_1<t_2.
\end{align*}
The inequality above can be read as
\begin{align*}
	\breve{T}^\#_t\phi(x)\leqslant\phi(x)\leqslant T^\#_t\phi(x),\qquad\forall t\geqslant0, x\in M,
\end{align*}
Theorem \ref{thm:msc exs} ensures the existence of maximal slope curves defined by \eqref{eq:msc_H_phi2}. This implies that the two inequalities above are equalities.

To prove (2), for any $x\in M$ and $p\in D^+\phi(x)$, the inquality
\begin{align*}
	-H^\#_\phi(x,p)=-H(x,p)+H(x,\mathbf{p}^{\#}_{\phi,H}(x))\leqslant0
\end{align*}
shows $\phi$ is a viscosity subsolution of \eqref{eq:sharp+}. On the other hand, for any differentiable point $x$ of $\phi$, we have
\begin{align*}
	-H^\#_\phi(x,D\phi(x))=-H(x,D\phi(x))+H(x,\mathbf{p}^{\#}_{\phi,H}(x))=0.
\end{align*}
Then, $\phi$ turns out to be a viscosity supersolution and so a viscosity solution of \eqref{eq:sharp+}. 

To prove (3), suppose $x\in \text{Sing}\,(\phi)$, then for any $p\in D^+u(x)\setminus\{\mathbf{p}^\#_{\phi,H}(x)\}$, by the strict convexity of $H(x,\cdot)$ and convexity of $D^+u(x)$ we obtain
\begin{align*}
	H^\#_\phi(x,p)=H(x,p)-H(x,\mathbf{p}^\#_{\phi,H}(x))>0.
\end{align*}
This implies $\phi$ is not a viscosity subsolution of \eqref{eq:sharp-}. 
\end{proof}

\begin{Cor}
There exists a discontinuous Hamiltonian $H:T^*M\to\R$ and $\phi\in\text{\rm SCL}\,(M)$ such that $\phi$ is a weak KAM solution but not a viscosity solution of the stationary Hamilton-Jacobi equation for $H$.
\end{Cor}

\begin{Rem}
In the definition of new Lax-Oleinik operators $T^\#_t$ and $\breve{T}^\#_t$, the term $H(x,\mathbf{p}^\#_{\phi,H}(x))$ is an analogy to Ma\~n\'e'scritical value. Maximal slope curves for the pair $(\phi,H)$ and the selection $\mathbf{p}^\#_{\phi,H}$ are the counterpart of calibrated curves in the classical theory.
\end{Rem}

\section{Intrinsic construction of strict singular characteristics}

In this section, we always suppose that $H$ is a Tonelli Hamiltonian and $\phi\in\text{\rm SCL}\,(M)$. We will continue our analysis of the intrinsic singular characteristics developed in \cite{Cannarsa_Cheng3,Cannarsa_Cheng_Fathi2017,Cannarsa_Cheng_Fathi2021} and construct strict singular characteristics by this method.

\subsection{More on intrinsic singular characteristics}

Recall the construction of intrinsic singular characteristics (see \cite{Cannarsa_Cheng3} for more detail). For any $x\in M$ we define a curve $\mathbf{y}_x:[0,\tau(\phi)]\to M$ with $\tau(\phi)>0$ determined below
\begin{equation}\label{eq:curve_y_x}
	\mathbf{y}_x(t):=
	\begin{cases}
		x,& t=0,\\
		\arg\max\{\phi(y)-A_t(x,y): y\in M\},& t\in(0,\tau(\phi)].
	\end{cases}
\end{equation}
The construction of the curve $\mathbf{y}_x(t)$ is explained as follows:
\begin{enumerate}[--]
	\item The supremum in \eqref{eq:curve_y_x} can be achieved for any $t>0$. Indeed, there exists a constant $\lambda_0>0$ depending on $\text{Lip}\,(\phi)$ such that if $y$ is a maximizer of $\phi(\cdot)-A_t(x,\cdot)$ then $y\subset B(x,\lambda_0 t)$.
	\item Taking $\lambda=\lambda_0+1$ and applying some regularity result for $A_t(x,y)$, we conclude that, there exists $t_0>0$ such that for any $t\in(0,t_0)$ and $x\in M$, the functions $y\mapsto A_t(x,y)$, $y\in B(x,\lambda t)$ is of class $C^2$ and convex with constant $C_2/t$.
	\item Since $\phi$ is semiconcave with constant $C_1$, the function $\phi(\cdot)-A_t(x,\cdot)$ is strictly concave provided $C_1-C_2/t<0$. Take $0<\tau(\phi)\leqslant t_0$ such that $C_1-C_2/\tau(\phi)<0$.
	\item Since $T^+\phi$ is naturally semiconvex, we conclude that if $t\in(0,\tau(\phi))$, then $T^+_t\phi\in C^{1,1}(M)$ and the maps $x\mapsto A_t(x,y)$, $x\in B(y,\lambda t)$ and $y\mapsto A_t(x,y)$, $y\in B(x,\lambda t)$ are convex with constant $C_2/t$ by Proposition \ref{pro:regularity}.
\end{enumerate}
We call $\mathbf{y}_x$ an \emph{intrinsic singular characteristic} from $x$. Observe that the above construction can be extended to $[0,\infty)$. In the following, we take $\tau(\phi)\leqslant\tau_1(\phi)$ in Proposition \ref{pro:graph}. 


\begin{Pro}\label{pro:curve_y}
Suppose $H$ is a Tonelli Hamiltonian and $\phi\in\text{\rm SCL}\,(M)$. Then for every $x\in M$, the following holds true. 
\begin{enumerate}[\rm (1)]
	\item For any $t\in(0,\tau(\phi)]$, if $\xi_t\in\Gamma^t_{x,\mathbf{y}_x(t)}$ is the minimal curve for $A_t(x,\mathbf{y}_x(t))$, then $\xi_t$ satisfies the differential equation
	\begin{equation}\label{eq:vector_field1_4}
		\begin{cases}
			\dot{\xi}_t(s)=H_p(\xi_t(s),DT^+_{t-s}\phi(\xi_t(s))),& s\in[0,t),\\
			\xi_t(0)=x.
		\end{cases}
	\end{equation}
	\item We have that
	\begin{equation}
		\mathbf{y}_x(t)=\pi_x\Phi_H^t(x,DT^+_t\phi(x)),\qquad \forall t\in(0,\tau(\phi)],
	\end{equation}
	where $\pi_x:T^*M\to M$ is the canonical projection.
	\item For any $t\in(0,\tau(\phi)]$, let $\xi_t$ be the minimal curve for $A_t(x,\mathbf{y}_x(t))$. Then there exists $C=C_{(\phi,H)}>0$ such that
	\begin{align*}
		|\xi_t(s)-\mathbf{y}_x(s)|\leqslant Ct,\qquad \forall s\in[0,t].
	\end{align*}
\end{enumerate}
\end{Pro}

\begin{proof}
Assertions (1) and (2) follow directly from Proposition \ref{pro:graph} (2) since we take $\tau(\phi)\leqslant\tau_1(\phi)$. Now we turn to prove (3) that working on Euclidean space. For any $s\in[0,t]$, by the dynamic programming principle we have that
\begin{align*}
	T^+_t\phi(x)=T^+_{t-s}\phi(\xi_t(s))-A_s(x,\xi_t(s)).
\end{align*}
Let $\mathbf{p}(s)=L_v(\xi_s(s),\dot{\xi}_s(s))$, $s\in[0,t]$. Then $\mathbf{p}(s)\in D^+\phi(\mathbf{y}(s))\cap D_yA_s(x,\mathbf{y}(s))$. By the semiconcavity of $\phi$ and the convexity of the fundamental solution $A_t(x,y)$ we deduce that
\begin{align*}
	0\leqslant&\,[T^+_{t-s}\phi(\xi_t(s))-A_s(x,\xi_t(s))]-[T^+_{t-s}\phi(\mathbf{y}(s))-A_s(x,\mathbf{y}(s))]\\
	=&\,[T^+_{t-s}\phi(\xi_t(s))-\phi(\xi_t(s))]+[\phi(\xi_t(s))-\phi(\mathbf{y}(s))]+[\phi(\mathbf{y}(s))-T^+_{t-s}\phi(\mathbf{y}(s))]\\
	&\,\quad-[A_s(x,\xi_t(s))-A_s(x,\mathbf{y}(s))]\\
	\leqslant&\,2\|T^+_{t-s}\phi-\phi\|_{\infty}+[\langle\mathbf{p}(s),\xi_t(s)-\mathbf{y}(s)\rangle+C_1|\xi_t(s)-\mathbf{y}(s)|^2]\\
	&\,\quad-\left[\langle\mathbf{p}(s),\xi_t(s)-\mathbf{y}(s)\rangle+\frac{C_2}s|\xi_t(s)-\mathbf{y}(s)|^2\right]\\
	\leqslant&\,2\|T^+_{t-s}\phi-\phi\|_{\infty}-\frac{C_2-C_1\tau(\phi)}{s}|\xi_t(s)-\mathbf{y}(s)|^2.
\end{align*}
Therefore,
\begin{align*}
	|\xi_t(s)-\mathbf{y}(s)|\leqslant C_3s^{\frac 12}\|T^+_{t-s}\phi-\phi\|^{\frac 12}_{\infty}\leqslant C_4(s(t-s))^{\frac 12}\leqslant C_4t.
\end{align*}
The proof is complete.
\end{proof}

\subsection{Intrinsic construction of strict singular characteristics}
In this section, we use the intrinsic method introduced in Section 4.1 to construct strict singular characteristics for Tonelli Hamiltonian $H$ and $\phi\in\text{SCL}\,(M)$. This also leads to an alternative proof of the existence of such characteristics with more geometric intuition.

Suppose $\phi\in\text{\rm SCL}\,(M)$ and $H:T^*M\to\R$ is a Tonelli Hamiltonian. Invoking Proposition \ref{pro:curve_y} (1), we observe that for given $t\in[0,\tau(\phi)]$,
\begin{equation}\label{eq:vf_W}
	W(s,x)=H_p(x,DT^+_{t-s}\phi(x)),\qquad (s,x)\in [0,t)\times M
\end{equation}
defines a time-dependent vector field on $[0,t)\times M$. Now, for any fixed $T>0$, we try to extend the vector field $W$ to $[-T,T]\times M$. Let
\begin{align*}
	\Delta: -T=\tau_0<\tau_1<\cdots<\tau_{N-1}<\tau_N=T
\end{align*}
be a partition of the interval $[-T,T]$ with $|\Delta|\leqslant\tau(\phi)$, where $|\Delta|=\max\{\tau_i-\tau_{i-1}: 1\leqslant i\leqslant N\}$ is the width of the partition. For any $t\in[-T,T)$, we let
\begin{align*}
\tau_{\Delta}^{+}(t)=\inf\{\tau_i|\,\tau_i>t\},\qquad \tau_{\Delta}^{-}(t)=\sup\{\tau_i|\,\tau_i\leqslant t\}.
\end{align*}
Then one can define a vector field $W_{\Delta}(t,x)$ on $[-T,T)\times M$ as
\begin{equation}\label{eq:vf_W_Delta}
	W_{\Delta}(t,x)=H_p(x,DT^+_{\tau_{\Delta}^{+}(t)-t}\phi(x)),\qquad t\in[-T,T), x\in M.
\end{equation}
Notice that the vector fields $W_{\Delta}$ is uniformly bounded for any partition $\Delta$. Moreover, $W_{\Delta}(\cdot,x)$ is piecewise continuous and $W_{\Delta}(t,\cdot)$ is Lipschitz continuous. Thus, for any $x\in M$, the differential equation
\begin{equation}\label{eq:ODE_W_Delta}
	\begin{cases}
		\dot{\gamma}(t)=W_{\Delta}(t,\gamma(t)),\qquad t\in[-T,T],\\
		\gamma(0)=x,
	\end{cases}
\end{equation}
admits a unique piecewise $C^1$ solution $\gamma:[-T,T]\to M$, which is of $C^1$ class on each partition interval, by Cauchy-Lipschitz theorem, and $|\dot{\gamma}|$ is uniformly bounded for any partition $\Delta$.

The following Theorem \ref{thm:v limit} and Theorem \ref{thm:solution limit} show that  $W_{\Delta}(t,x)$ converges to $H_p(x,\mathbf{p}^{\#}_{\phi,H}(x))$ as $|\Delta|\to 0$, and that the corresponding solution curves converges as well.

\begin{Pro}\label{pro:limit min energy}
	Suppose $H:T^*M\to\R$ is a Tonelli Hamiltonian and $\phi\in\text{SCL}\,(M)$, then
	\begin{align*}
		\lim_{t\to 0^+}DT^+_{t}\phi(x)=\mathbf{p}^{\#}_{\phi,H}(x),\qquad \forall x\in M.
	\end{align*}
\end{Pro}

\begin{proof}
Because of the local nature we work on Euclidean space. For any subsequence $t_k\to 0^+$ such that $DT^+_{t_k}\phi(x)$ converges to some $p_0$ as $k\to\infty$, let $(x_k,p_k)=\Phi_H^{t_k}((x,DT^+_{t_k}\phi(x)))$. It follows that
	\begin{align*}
		\lim_{k\to\infty}x_k=x,\quad \lim_{k\to\infty}p_k=p_0,\quad \lim_{k\to\infty}\frac{x_k-x}{t_k}=H_p(x,p_0),
	\end{align*}
and Proposition \ref{pro:graph} (2) implies that $p_k\in D^+\phi(x_k)$. By Lemma \ref{lem:usc dir}, we also have that $p_0\in D^+\phi(x)$ and
\begin{align*}
	\langle p_0,H_p(x,p_0)\rangle=\min_{p\in D^+\phi(x)}\langle p,H_p(x,p_0)\rangle.
\end{align*}
Thus, Lemma \ref{lem:epf} yields $p_0=\mathbf{p}^{\#}_{\phi,H}(x)$. This completes our proof.
\end{proof}

\begin{The}\label{thm:v limit}
	For all partitions $\{\Delta\}$ of the interval $[-T,T]$, we have that
	\begin{align*}
		\lim_{|\Delta|\to 0} W_{\Delta}(t,x)=H_p(x,\mathbf{p}^{\#}_{\phi,H}(x)),\qquad \forall t\in[-T,T),\ x\in M.
	\end{align*}
\end{The}

\begin{proof}
	For any $t\in[-T,T)$ and $x\in M$, Proposition \ref{pro:limit min energy} implies that
	\begin{align*}
		\lim_{|\Delta|\to 0}W_{\Delta}(t,x)=\lim_{|\Delta|\to 0}H_p(x,DT^+_{\tau_{\Delta}^{+}(t)-t}\phi(x))=\lim_{s\to 0^+}H_p(x,DT^+_{s}\phi(x))=H_p(x,\mathbf{p}^{\#}_{\phi,H}(x))
	\end{align*}
	thus completing the proof.
\end{proof}

\begin{Pro}\label{pro:piecewise msc}
	Let $\Delta$ be any partition of $[-T,T]$ with $|\Delta|\leqslant\tau(\phi)$, and $\gamma:[-T,T]\to M$ be a solution of \eqref{eq:ODE_W_Delta}. Then we have
	\begin{align*}
	\phi(\gamma(T))-\phi(\gamma(-T))=\,\int^T_{-T}\Big\{L(\gamma(s),\dot{\gamma}(s))+H(\gamma(\tau_{\Delta}^{-}(s)),DT^{+}_{s-\tau_{\Delta}^{-}(s)}(\gamma(\tau_{\Delta}^{-}(s))))\Big\}\ ds.
	\end{align*}
\end{Pro}

\begin{proof}
	For every $i=0,\ldots,N-1$, Proposition \ref{pro:curve_y} (1) implies that
	\begin{equation}\label{eq:4-01}
		T^+_{\tau_{i+1}-\tau_{i}}\phi(\gamma(\tau_{i}))=\phi(\gamma(\tau_{i+1}))-\int^{\tau_{i+1}}_{\tau_{i}}L(\gamma,\dot{\gamma})\ ds.
	\end{equation}
    By Proposition \ref{pro:graph} (3), we also have that
	\begin{equation}\label{eq:4-02}
		T^+_{\tau_{i+1}-\tau_i}\phi(\gamma(\tau_{i}))-\phi(\gamma(\tau_{i}))=\,\int^{\tau_{i+1}}_{\tau_i}\frac d{ds}T^+_{s-\tau_i}\phi(\gamma(\tau_{i}))\ ds=\,\int^{\tau_{i+1}}_{\tau_i}H(\gamma(\tau_{i}),DT^+_{s-\tau_i}\phi(\gamma(\tau_{i})))\ ds
	\end{equation}
    Invoking \eqref{eq:4-01} and \eqref{eq:4-02}, it follows that
	\begin{equation}\label{eq:calibrate1}
		\begin{split}
			&\,\phi(\gamma(\tau_{i+1}))-\phi(\gamma(\tau_{i}))\\
			=&\,(\phi(\gamma(\tau_{i+1}))-T^+_{\tau_{i+1}-\tau_i}\phi(\gamma(\tau_{i})))+(T^+_{\tau_{i+1}-\tau_i}\phi(\gamma(\tau_{i}))-\phi(\gamma(\tau_{i})))\\
			=&\,\int^{\tau_{i+1}}_{\tau_i}L(\gamma(s),\dot{\gamma}(s))\ ds+\int^{\tau_{i+1}}_{\tau_i}H(\gamma(\tau_{i}),DT^+_{s-\tau_i}\phi(\gamma(\tau_{i})))\ ds.
		\end{split}
	\end{equation}
    Summing up \eqref{eq:calibrate1} for $i=0,\ldots,N-1$ we obtain
	\begin{align*}
			\phi(\gamma(T))-\phi(\gamma(-T))=&\,\int^T_{-T}L(\gamma(s),\dot{\gamma}(s))\ ds+\sum^{N-1}_{i=0}\int^{\tau_{i+1}}_{\tau_i}H(\gamma(\tau_{i}),DT^+_{s-\tau_i}\phi(\gamma(\tau_{i})))\ ds\\
			=&\,\int^T_{-T}\Big\{L(\gamma(s),\dot{\gamma}(s))+H(\gamma(\tau_{\Delta}^{-}(s)),DT^{+}_{s-\tau_{\Delta}^{-}(s)}(\gamma(\tau_{\Delta}^{-}(s))))\Big\}\ ds.
	\end{align*}
\end{proof}

\begin{The}\label{thm:solution limit}
Suppose $\{\Delta_k\}$ is a sequence of partitions of $[-T,T]$ with $|\Delta_k|\leqslant\tau(\phi)$ for all $k\in\N$, and each $\gamma_k:[-T,T]\to M$, $k\in\N$ is a solution of the equation
	\begin{equation}\label{eq:ODE}
		\dot{\gamma}_k(t)=W_{\Delta_k}(t,\gamma_k(t)),\qquad a.e.\ t\in[-T,T].
	\end{equation}
Further assume that:
\begin{enumerate}[\rm (i)]
	\item $\lim_{k\to\infty}|\Delta_k|=0$,
	\item $\gamma_k$ converges uniformly to $\gamma:[-T,T]\to M$.
\end{enumerate}
Then $\gamma$ is a strict singular characteristic for the pair $(\phi,H)$, that is,
\begin{equation}
	\dot{\gamma}(t)=H_p(\gamma(t),\mathbf{p}^{\#}_{\phi,H}(\gamma(t))),\qquad a.e.\ t\in[-T,T].
\end{equation}
\end{The}

\begin{proof}
For each $k\in\N$, Proposition \ref{pro:piecewise msc} implies
\begin{equation}\label{eq:4p-1}
	\begin{split}
		&\,\phi(\gamma_k(T))-\phi(\gamma_k(-T))\\
		=&\,\int^T_{-T}\Big\{L(\gamma_k(s),\dot{\gamma}_k(s))+H(\gamma_k(\tau_{\Delta_k}^{-}(s)),DT^{+}_{s-\tau_{\Delta_k}^{-}(s)}(\gamma_k(\tau_{\Delta_k}^{-}(s))))\Big\}\ ds.
	\end{split}
\end{equation}
By a classical result from calculus of variation (see \cite[Theorem 3.6]{Buttazzo_Giaquinta_Hildebrandt_book} or \cite[Section 3.4]{Buttazzo_book}), 
	\begin{equation}\label{eq:4p-2}
		\liminf_{k\to\infty}\int^T_{-T}L(\gamma_k(s),\dot{\gamma}_k(s))\ ds\geqslant\int^T_{-T}L(\gamma(s),\dot{\gamma}(s))\ ds.
	\end{equation}
Invoking Fatou's lemma, Proposition \ref{pro:graph}, and Lemma \ref{lem:lsc1} we have that
	\begin{equation}\label{eq:4p-3}
		\begin{split}
		&\,\liminf_{k\to\infty}\int^T_{-T} H(\gamma_k(\tau_{\Delta_k}^{-}(s)),DT^{+}_{s-\tau_{\Delta_k}^{-}(s)}(\gamma_k(\tau_{\Delta_k}^{-}(s))))\ ds\\
		\geqslant&\,\int^T_{-T}\liminf_{k\to\infty}H(\gamma_k(\tau_{\Delta_k}^{-}(s)),DT^{+}_{s-\tau_{\Delta_k}^{-}(s)}(\gamma_k(\tau_{\Delta_k}^{-}(s))))\ ds\\
		\geqslant&\,\int^T_{-T}H(\gamma(s),\mathbf{p}^{\#}_{\phi,H}(\gamma(s)))\ ds.
		\end{split}
	\end{equation}
Now \eqref{eq:4p-1}, \eqref{eq:4p-2} and \eqref{eq:4p-3} implies that
\begin{align*}
\phi(\gamma(T))-\phi(\gamma(-T))=&\,\lim_{k\to\infty}\phi(\gamma_k(T))-\phi(\gamma_k(-T))\\
=&\,\lim_{k\to\infty}\int^T_{-T}\Big\{L(\gamma_k(s),\dot{\gamma}_k(s))+H(\gamma_k(\tau_{\Delta_k}^{-}(s)),DT^{+}_{s-\tau_{\Delta_k}^{-}(s)}(\gamma_k(\tau_{\Delta_k}^{-}(s))))\Big\}\ ds\\
\geqslant&\,\int^T_{-T}\Big\{L(\gamma(s),\dot{\gamma}(s))+H(\gamma(s),\mathbf{p}^{\#}_{\phi,H}(\gamma(s)))\Big\}\ ds	
\end{align*}
Thus, by \eqref{eq:msc_H_phi2}, $\gamma$ is a strict singular characteristic for the pair $(\phi,H)$.
\end{proof}

\begin{The}\label{thm:pos neg}
	Let $\phi\in\text{\rm SCL}\,(M)$ and let $H:T^*M\to\R$ be a Tonelli Hamiltonian. Then, for any $x\in M$, there exists a strict singular characteristic $\gamma:\R\to M$ with $\gamma(0)=x$ for the pair $(\phi,H)$. In other words, the equation
	\begin{equation}
		\begin{cases}
			\dot{\gamma}(t)=H_p(\gamma(t),\mathbf{p}^{\#}_{\phi,H}(\gamma(t))),\qquad a.e.\ t\in\R, \\
			\gamma(0)=x
		\end{cases}
	\end{equation}
	admits a Lipschitz solution.
\end{The}

\begin{proof}
	Fix any $T>0$. Take a sequence of partitions $\{\Delta_k\}$ of $[-T,T]$ satisfying $\lim_{k\to\infty}|\Delta_k|=0$. For each $k\in\N$, let $\gamma_k:[-T,T]\to M$ be the solution of equation \eqref{eq:ODE_W_Delta} with $\Delta=\Delta_k$. Since $|\dot{\gamma}_k|$ is uniformly bounded, up to taking a subsequence we can suppose $\gamma_k$ converges uniformly to $\gamma:[-T,T]\to M$, $\gamma(0)=x$. Then Theorem \ref{thm:solution limit} implies that $\gamma$ is a strict singular characteristic for the pair $(\phi,H)$. The arbitrariness of $T$ leads to our conclusion.
\end{proof}

Invoking our previous intrinsic construction, for any partition $\Delta$ of $[0,T]$ with $|\Delta|\leqslant\tau(\phi)$ and $x\in M$, we define a curve $\mathbf{y}_{\Delta,x}:[0,T]\to M$ inductively by
\begin{align*}
	\mathbf{y}_{\Delta,x}(0)=x,\quad \mathbf{y}_{\Delta,x}(t)=\arg\max\{\phi(y)-A_{t-\tau_{\Delta}^{-}(t)}(\mathbf{y}_{\Delta,x}(\tau_{\Delta}^{-}(t)),y):y\in M\},\quad 
\end{align*}

\begin{The}\label{thm:in sc}
Suppose $\{\Delta_k\}$ is a sequence of partitions of $[0,T]$ with $\lim_{k\to\infty}|\Delta_k|=0$, and $\{x_k\}\subset M$. Then any limiting curve $\mathbf{y}$ of the intrinsic singular characteristics $\mathbf{y}_{\Delta_k,x_k}$ is a strict singular characteristic for the pair $(\phi,H)$.
\end{The}

\begin{proof}
For every $k\in\N$, we let $\gamma_k:[0,T]\to M$ be the solution of \eqref{eq:ODE_W_Delta} with $\Delta=\Delta_k$, $x=x_k$. Then Proposition \ref{pro:curve_y} (2) implies
\begin{align*}
	\limsup_{k\to\infty}\|\gamma_k-\mathbf{y}\|_{\infty}\leqslant\,\limsup_{k\to\infty}\|\gamma_k-\mathbf{y}_{\Delta_k,x_k}\|_{\infty}+\limsup_{k\to\infty}\|\mathbf{y}_{\Delta_k,x_k}-\mathbf{y}\|_{\infty}\leqslant\,\limsup_{k\to\infty}|\Delta_k|+0=\,0.
\end{align*}
Thus, by Theorem \ref{thm:solution limit}, $\mathbf{y}$ is a strict singular characteristic for the pair $(\phi,H)$.
\end{proof}

%
%

\section{Propagation of singularities}

In this section, we study the problem of global propagation of singularities for a weak KAM solution $\phi$ of \eqref{eq:HJ_wk}. We first recall the following result.

\begin{The}[\cite{Albano2014_1} Theorem 1.2]\label{thm:global_prop}
Suppose $H:T^*M\to\R$ is a Tonelli Hamiltonian, $\phi$ is a weak KAM solution of \eqref{eq:HJ_wk}. If $\gamma:[0,+\infty)\to M$ is a Lipschitz solution of the differential inclusion
\begin{equation}\label{eq:gc51}
\dot{\gamma}(t)\in\text{\rm co}\,H_p(\gamma(t),D^+\phi(\gamma(t))),\qquad a.e.\ t\in[0,+\infty),\\
\end{equation}
and $\gamma(0)\in\overline{\text{\rm Sing}\,(\phi)}$, then we have $\gamma(t)\in\overline{\text{\rm Sing}\,(\phi)}$ for all $t\geqslant0$.
\end{The}

Notice that for every weak KAM solution $\phi$ of \eqref{eq:HJ_wk}, we have $\text{\rm Sing}\,(\phi)\subset\text{\rm Cut}\,(\phi)\subset\overline{\text{\rm Sing}\,(\phi)}$. Now, we will do further analysis by intrinsic approach. Our aim is to prove the result of global propagation of cut points $\text{\rm Cut}\,(\phi)$.

\subsection{Local $C^{1,1}$ regularity in the complement of cut locus}

In this section, we suppose $M$ is a compact manifold and $H:T^*M\to\R$ is a Tonelli Hamiltonian. 
Given $(x,p)\in T^*M$, we consider the following Cauchy problem for Hamiltonian ODEs in a local chart:
\begin{equation}\tag{H}\label{eq:Hamiltonian_ODE}
	\begin{cases}
		\dot{X}(t)=H_p(X(t),P(t)),\\
		\dot{P}(t)=-H_x(X(t),P(t)),
	\end{cases}
	\quad t\in\R, \qquad\text{with}\qquad
	\begin{cases}
		X(0)=x,\\
		P(0)=p.
	\end{cases}
\end{equation}
We denote by $(X(t;x,p),P(t;x,p))$ the solution of \eqref{eq:Hamiltonian_ODE} and set
\begin{align*}
	U(t;x,p):=
	\begin{cases}
		\int^t_0\langle P,\dot{X}\rangle-H(X,P)\ ds,&t\geqslant0;\\
		\int^0_t\langle P,\dot{X}\rangle-H(X,P)\ ds,&t<0.
	\end{cases}
\end{align*}
For any $x\in M$ and $\alpha>0$, set
\begin{align*}
	V(x,\alpha)=\{p\in T^*_xM: |H_p(x,p)|<\alpha\}.
\end{align*}

\begin{Pro}\label{pro:reg1}
For any $\alpha>0$ there exists $\tau=\tau_\alpha>0$ and $D_i=D_{\alpha,i}>0$, $i=1,\ldots,6$, such that for any $x\in M$ and $0<|t|\leqslant\tau$ the following statements hold.
\begin{enumerate}[\rm (1)]
	\item $p\mapsto X(t;x,p)$ is a $C^1$-diffeomorphism on $V(x,\alpha)$ such that
	\begin{align*}
		D_1|t|\cdot|p_1-p_2|\leqslant|X(t,x,p_1)-X(t,x,p_2)|\leqslant D_2|t|\cdot|p_1-p_2|,\qquad\forall p_1,p_2\in V(x,\alpha).
	\end{align*}
	\item $p\mapsto P(t;x,p)$ is a $C^1$-diffeomorphism on $V(x,\alpha)$ such that
	\begin{align*}
		D_3\cdot|p_1-p_2|\leqslant|P(t,x,p_1)-P(t,x,p_2)|\leqslant D_4\cdot|p_1-p_2|,\qquad\forall p_1,p_2\in V(x,\alpha).
	\end{align*}
	\item The map $X\mapsto U(t;x,X^{-1}(t;x,X))$ is a $C^2$-function on $X(t;x,V(x,\alpha))$ such that
	\begin{gather*}
		\frac{\partial U}{\partial X}(t,x,X^{-1}(t;x,X))=
		\begin{cases}
			P(t;x,X^{-1}(t;x,X)),&t>0,\\
			-P(t;x,X^{-1}(t;x,X)),&t<0,
		\end{cases}
		\qquad\forall X\in X(t;x,V(x,\alpha)),\\
		\frac{D_5}{|t|}\leqslant\frac{\partial^2U}{\partial X^2}(t;x,X^{-1}(t;x,X))\leqslant\frac{D_6}{|t|},\qquad \qquad\forall X\in X(t;x,V(x,\alpha)).
	\end{gather*}
\end{enumerate}
\end{Pro}

\begin{proof}
We work on Euclidean case. We only prove the statements for $t>0$ since one can argue similarly for $t<0$. Consider the variational equation of \eqref{eq:Hamiltonian_ODE}
\begin{equation}\label{eq:VE}\tag{VE}
	\begin{cases}
		\dot{X}_p(t;x,p)=H_{px}\cdot X_p+H_{pp}\cdot P_p,\\
		\dot{P}_p(t;x,p)=-H_{xx}\cdot X_p-H_{xp}\cdot P_p,
	\end{cases}
	\quad t\in\R,\qquad\text{with}\qquad
	\begin{cases}
		X_p(0;x,p)=0,\\
		P_p(0;x,p)=Id.
	\end{cases}
\end{equation}
Observe that the initial data of the linear equation \eqref{eq:VE} is constant and the coefficients matrix 
\begin{align*}
	\begin{bmatrix}
		H_{px}(X,P)&H_{pp}(X,P)\\
		-H_{xx}(X,P)&-H_{xp}(X,P)
	\end{bmatrix}
\end{align*}
is continuous with respect to $X=X(t;x,p)$ and $P=P(t;x,p)$. Thus, both $(X_p,P_p)$ and $(\dot{X}_P,\dot{P}_p)$ are uniformly continuous on the bounded set
\begin{align*}
	\{(t,x,p): t\in[0,1], x\in M, p\in V(x,\alpha)\}
\end{align*}
It follows that
\begin{equation}\label{eq:initial_condition1}
    \begin{split}
    	\lim_{t\to0^+}\frac 1tX_p(t;x,p)=&\,\lim_{t\to0^+}\frac{X_p(0,x,p)+\int^t_0\dot{X}_p(s;x,p)\ ds}t=\dot{X}_p(0;x,p)=H_{pp}(x,p),\\
    	\lim_{t\to0^+}P_p(t;x,p)=&\,P_p(0;x,p)=Id,
    \end{split}
\end{equation}
uniformly for $x\in M$ and $p\in V(x,\alpha)$.
Due to the Tonelli condition (H1') there exists $D'_1,D_2'\geqslant0$ such that
\begin{equation}\label{eq:5-01}
	D'_1Id\leqslant H_{pp}\leqslant D'_2Id,\qquad x\in M, p\in V(x,\alpha).
\end{equation}
Together with \eqref{eq:initial_condition1}, the implicit function theorem leads to (1) and (2).

To prove (3) we observe that for any $t\in\R$, $(x,p)\in T^*M$,
\begin{align*}
	U_p(t;x,p)=&\,\int^t_0P_p\cdot \dot{X}+P\cdot \dot{X}_p-H_x\cdot X_p-H_p\cdot P_p\ ds\\
	=&\,\int^t_0P\cdot \dot{X}_p+\dot{P}\cdot X_p=(P\cdot X_p)\big|^t_0\\
	=&\,P(t;x,p)\cdot X_p(t;x,p).
\end{align*}
Then, for $\tau\ll1$ and $x\in M$, $t\in(0,\tau]$ we have
\begin{align*}
	\frac{\partial U}{\partial X}(t;x,X^{-1}(t;x,X))=&\,\frac{\partial U}{\partial p}(t;x,X^{-1}(t;x,X))\cdot\frac{\partial X^{-1}}{\partial X}(t;x,X)\\
	=&\,P(t;x,X^{-1}(t;x,X))\cdot \frac{\partial X}{\partial p}(t;x,X^{-1}(t;x,X))\cdot\frac{\partial X^{-1}}{\partial X}(t;x,X)\\
	=&\,P(t;x,X^{-1}(t;x,X)),\qquad \forall X(t;x,V(x,\alpha)),
\end{align*}
and $X\mapsto P(t;x,X^{-1}(t;x,X))$ is a $C^1$-diffeomorphism on $X(t;x,V(x,\alpha))$ by (1) and (2). Thus, $X\mapsto U(t;x,X^{-1}(t;x,X))$ is a $C^2$-diffeomorphism on $X(t;x,V(x,\alpha))$. Moreover,
\begin{equation}\label{eq:5-02}
	\begin{split}
		\frac{\partial^2U}{\partial X^2}(t;x,X^{-1}(t;x,X))=&\,\frac{\partial P}{\partial X}(t;x,X^{-1}(t;x,X))\\
	=&\,\frac{\partial P}{\partial p}(t;x,X^{-1}(t;x,X))\cdot\frac{\partial X^{-1}}{\partial X}(t;x,X^{-1}(t;x,X)),
	\end{split}
\end{equation}
for all $X\in X(t;x,V(x,\alpha))$. Taking inverse in \eqref{eq:initial_condition1} and \eqref{eq:5-01} we conclude that for any $x\in M$ and $p\in V(x,\alpha)$
\begin{equation}\label{eq:5-03}
	\begin{split}
	\lim_{t\to0^+}t\cdot X^{-1}_p(t;x,p)=H^{-1}_{pp}(x,p)\\
	(D'_2)^{-1}Id\leqslant H^{-1}_{pp}\leqslant(D'_1)^{-1}Id
    \end{split}
\end{equation}
Then \eqref{eq:initial_condition1}, \eqref{eq:5-02} and \eqref{eq:5-03} leads to our conclusion.
\end{proof}

\begin{Pro}\label{pro:reg2}
For any $\lambda>0$ there exists $t_\lambda>0$ and $C_{\lambda,i}>0$, $i=1,2,3,4$, such that for any $x\in M$ and $t\in(0,t_\lambda]$
\begin{enumerate}[\rm (1)]
	\item For any $y\in B(x,\lambda t)$ there exists a unique minimal curve $\xi_{t,x,y}\in\Gamma^t_{x,y}$ for $A_t(x,y)$.
	\item Let $p_{t,x,y}(\cdot)=L_v(\xi_{t,x,y}(\cdot),\dot{\xi}_{t,x,y}(\cdot))$ be the dual arc of $\xi_{t,x,y}$. For any $s\in[0,t]$, $y\mapsto p_{t,x,y}(s)$ is a $C^1$-diffeomorphism on $B(x,\lambda t)$ and
	\begin{align*}
		\frac{C_{\lambda,1}}{t}|y_1-y_2|\leqslant|p_{t,x,y_1}(s)-p_{t,x,y_2}(s)|\leqslant\frac{C_{\lambda,2}}{t}|y_1-y_2|,\qquad\forall y_1,y_2\in B(x,\lambda t).
	\end{align*}
	\item $y\mapsto A_t(x,y)$ is a $C^2$ function on $B(x,\lambda t)$, and
	\begin{gather*}
		\frac{\partial}{\partial y}A_t(x,y)=p_{t,x,y}(t),\\
		\frac{C_{\alpha,3}}{t}Id\leqslant\frac{\partial^2}{\partial y^2}A_t(x,y)\leqslant\frac{C_{\alpha,4}}{t}Id,
	\end{gather*}
	for all $y\in B(x,\lambda t)$.
\end{enumerate}

Similar statements also hold for $A_t(\cdot,x)$ instead of $A_t(x,\cdot)$.
\end{Pro}

\begin{proof}
We only prove the statements for $A_t(x,\cdot)$ since one can argue similarly for $A_t(\cdot,x)$. By classical Tonelli theory and Hamiltonian dynamical systems, for any $x,y\in M$ and $t>0$, any minimizer $\xi_{t,x,y}$ for $A_t(x,y)$ determines the pair $(\xi_{t,x,y},p_{t,x,y})$ satisfying Hamiltonian system \eqref{eq:Hamiltonian_ODE}, and
\begin{align*}
	A_t(x,y)=&\,\int^t_0L(\xi_{t,x,y}(s),\dot{\xi}_{t,x,y}(s))\ ds=\int^t_0\langle p_{t,x,y}(s),\dot{\xi}_{t,x,y}(s)\rangle-H(\xi_{t,x,y}(s),p_{t,x,y}(s))\ ds\\
	=&\,U(t;x,p_{t,x,y}(0)).
\end{align*}
It is clear that there exists a non-decreasing function $F:[0,+\infty)\to[0,+\infty)$ such that $|\dot{\xi}_{t,x,y}(s)|\leqslant F(|y-x|/t)$ for $s\in[0,t]$. Let $t_\lambda=\tau_{F(\lambda)}$ in Proposition \ref{pro:reg1} and set
\begin{align*}
	C_{\lambda,1}=\frac{D_{F(\lambda),3}}{D_{F(\lambda),2}},\quad C_{\lambda,2}=\frac{D_{F(\lambda),4}}{D_{F(\lambda),1}},\quad C_{\lambda,3}=D_{F(\lambda),5},\quad C_{\lambda,4}=D_{F(\lambda),6},
\end{align*}
then all the statements are direct consequences of Proposition \ref{pro:reg1}.
\end{proof}

\begin{Lem}\label{lem:twist1}
For any $\lambda$ let $t_\lambda$ and $C_{\lambda,2}$ be as in Proposition \ref{pro:reg2}. If $t\in(0,t_\lambda]$ and $x_1,x_2,y_1,y_2\in M$ satisfy $|x_i-y_i|/t<\lambda$, $i,j=1,2$, then 
\begin{align*}
	|A_t(x_1,y_1)+A_t(x_2,y_2)-A_t(x_1,y_2)-A_t(x_2,y_1)|\leqslant\frac{C_{\lambda,2}}{t}|x_1-y_1|\cdot|y_1-y_2|.
\end{align*}
\end{Lem}

\begin{proof}
Invoking Proposition \ref{pro:reg2} we have that
\begin{align*}
	&\,|A_t(x_1,y_1)+A_t(x_2,y_2)-A_t(x_1,y_2)-A_t(x_2,y_1)|\\
	=&\,|[A_t(x_1,y_1)-A_t(x_1,y_2)]-[A_t(x_2,y_1)-A_t(x_2,y_2)]|\\
	=&\,\Big|\Big\langle y_1-y_2,\int^1_0p_{t,x_1,y_2+r(y_1-y_2)}(t)\ dr\Big\rangle-\Big\langle y_1-y_2,\int^1_0p_{t,x_2,y_2+r(y_1-y_2)}(t)\ dr\Big\rangle\Big|\\
	\leqslant&\,|y_1-y_2|\cdot\int^1_0|p_{t,x_1,y_2+r(y_1-y_2)}(t)-p_{t,x_2,y_2+r(y_1-y_2)}(t)|\ dr\\
	\leqslant&\,\frac{C_{\lambda,2}}{t}|x_1-x_2|\cdot|y_1-y_2|.
\end{align*}
This completes the proof.
\end{proof}

\begin{Lem}\label{lem:convex1}
For any $\lambda$ there exists $t'_\lambda\in(0,t_\lambda]$	and $C_{\lambda,5}>0$ such that, if $t\in(0,t'_\lambda]$ and $x,y,z\in M$ satisfy $|y-z|/t<\lambda$ and $|z-x|/t<\lambda$, then
\begin{align*}
	A_t(x,z)+A_t(z,y)-A_{2t}(x,y)\geqslant C_{\lambda,5}t\cdot|p_{t,x,z}(t)-p_{t,z,y}(0)|^2.
\end{align*}
\end{Lem}

\begin{proof}
By the triangle inequality
\begin{align*}
	\frac{|y-x|}{2t}\leqslant\frac{|y-z|}{2t}+\frac{|z-y|}{2t}<\lambda
\end{align*}
with $x,y,z\in M$ such that $|y-z|/t<\lambda$ and $|z-x|/t<\lambda$, and $t\in(0,t_\lambda]$. Thus we have $|\dot{\xi}_{2t,x,y}(s)|<F(\lambda)$ for all $s\in[0,2t]$ where $\dot{\xi}_{2t,x,y}$ is determined in Proposition \ref{pro:reg2} (1) and $F:[0,+\infty)\to[0,+\infty)$ is the function used in the proof of Proposition \ref{pro:reg2}. Set $z'=\xi_{2t,x,y}(t)$ and let $t'_\lambda=t_{F(\lambda)}\leqslant t_\lambda$. By Proposition \ref{pro:reg2}
\begin{equation}\label{eq:low_bound1}
	\begin{split}
		&\,A_t(x,z)+A_t(z,y)-A_{2t}(x,y)\\
	=&\,[A_t(x,z)-A_t(x,z')]+[A_t(z,y)-A_t(z',y)]\\
	=&\,\Big\langle z-z',p_{2t,x,y}(t)+\int^1_0s(z-z')\cdot\int^1_0\frac{\partial^2}{\partial  z^2}A_t(x,z'+rs(z-z'))\ drds\Big\rangle\\
	&\,\ +\Big\langle z-z',-p_{2t,x,y}(t)+\int^1_0s(z-z')\cdot\int^1_0\frac{\partial^2}{\partial  z^2}A_t(z'+rs(z-z'),y)\ drds\Big\rangle\\
	\geqslant&\,2\int^1_0s\cdot\frac{C_{F(\lambda),3}}{t}|z-z'|^2\ ds=\frac{C_{F(\lambda),3}}{t}|z-z'|^2.
	\end{split}
\end{equation}
On the other hand, applying Proposition \ref{pro:reg2} again we obtain
\begin{equation}\label{eq:low_bound2}
	\begin{split}
		|z-z'|^2\geqslant&\,\frac{t^2}{2C^2_{F(\lambda),2}}(|p_{2t,x,y}(t)-p_{t,x,z}(t)|^2+|p_{2t,x,y}(t)-p_{t,z,y}(0)|^2)\\
	\geqslant&\,\frac{t^2}{4C^2_{F(\lambda),2}}|p_{t,x,z}(t)-p_{t,z,y}(0)|^2.
	\end{split}
\end{equation}
Combining \eqref{eq:low_bound1} and \eqref{eq:low_bound2} yields
\begin{align*}
	&\,A_t(x,z)+A_t(z,y)-A_{2t}(x,y)\\
	\geqslant&\,\frac{C_{F(\lambda),3}}{t}\cdot\frac{t^2}{4C^2_{F(\lambda),2}}|p_{t,x,z}(t)-p_{t,z,y}(0)|^2=C_{\lambda,5}t\cdot|p_{t,x,z}(t)-p_{t,z,y}(0)|^2,
\end{align*}
with $C_{\lambda,5}:=C_{F(\lambda),3}/(4C^2_{F(\lambda),2})$.
\end{proof} 

Now, for any weak KAM solution $\phi$ of \eqref{eq:HJ_wk} and $t>0$, let
\begin{align*}
	E(\phi,t):=\{x\in M: \tau_\phi(x)\geqslant t\},
\end{align*}
where $\tau_\phi$ is the cut time function associated with $\phi$. Notice that $\phi$ is differentiable at each point in $E(\phi,t)$. The following theorem ensure that $\phi$ has certain $C^{1,1}$-regularity on $E(\phi,t)\subset M\setminus\text{Cut}\,(\phi)$.

\begin{The}\label{thm:C11}
There exists positive constants $\lambda=
\lambda_H,T=T_H,C=C_H$ such that for any weak KAM solution $\phi$ of \eqref{eq:HJ_wk}, $t\in(0,T]$ and $x\in E(\phi,t)$
\begin{align*}
	|p-D\phi(x)|\leqslant\frac Ct|y-x|,\qquad\forall y\in B(x,\lambda t), p\in D^+\phi(y).
\end{align*}
\end{The}

\begin{proof}
Notice that $M$ is compact and any weak KAM solution $\phi$ of \eqref{eq:HJ_wk} is uniformly Lipschitz. Hence there exists $\lambda=\lambda_H>0$ such that for all $(\phi,H)$ calibrated curve $\xi:[0,t]\to M$, we have that $|\dot{\xi}(s)|<\lambda$ for all $s\in[0,t]$. Set $T=T_H=t'_{2\lambda_H}$ as in Proposition \ref{pro:reg2}, Lemma \ref{lem:twist1}, and Lemma \ref{lem:convex1}. Let $\phi$ be any weak KAM solution of \eqref{eq:HJ_wk}, $t\in(0,T]$ and $x\in E(\phi,t)$. We consider three extremals:
\begin{enumerate}[(i)]
	\item Let $\xi_1:[0,2t]\to M$ be the unique $(\phi,H)$ calibrated curve satisfying $\xi_1(t)=x$;
	\item For any $y\in B(x,\lambda t)$ and $p^*\in D^*\phi(y)$ there exists a unique $(\phi,H)$ calibrated curve $\xi_2:[0,t]\to M$ satisfying $\xi_2(t)=y$, $L_v(y,\dot{\xi}_2(t))=p^*$;
	\item Let $\xi_3:[0,t]\to M$ be the unique minimal curve for $A_t(\xi_2(0),x)$,
\end{enumerate}
and let $p_i(\cdot)=L_v(\xi_i(\cdot),\dot{\xi_i}(\cdot))$ be the associated dual arc of $\xi_i(\cdot)$, $i=1,2,3$, respectively. Since
\begin{gather*}
	\frac{|x-\xi_2(0)|}{t}\leqslant\frac{|x-y|}{t}+\frac{|y-\xi_2(0)|}{t}<\lambda+\lambda=2\lambda,\\
	\frac{|y-\xi_2(0)|}{t}\leqslant\frac{|y-x|}{t}+\frac{|x-\xi_2(0)|}{t}<\lambda+\lambda=2\lambda,
\end{gather*}
by Proposition \ref{pro:reg2} and Lemma \ref{lem:twist1} we conclude that
\begin{equation}\label{eq:5-31}
	\begin{split}
		&\,[\phi(\xi_2(0))+A_t(\xi_2(0),x)]-[\phi(\xi_1(0))+A_t(\xi_1(0),x)]\\
	=&\,\big([\phi(\xi_2(0))+A_t(\xi_2(0),y)]-[\phi(\xi_1(0))+A_t(\xi_1(0),y)]\big)\\
	&\,\ +\big(A_t(\xi_1(0),y)+A_t(\xi_2(0),x)-A_t(\xi_1(0),x)-A_t(\xi_2(0),y)\big)\\
	\leqslant&\,\frac{C_{2\lambda_H,2}}{t}|y-x|\cdot|\xi_2(0)-\xi_1(0)|\leqslant \frac{C_{2\lambda_H,2}}{C_{2\lambda_H,1}}|y-x|\cdot|p_3(t)-p_1(t)|.
	\end{split}
\end{equation}
On the other hand, by Lemma \ref{lem:convex1} we get
\begin{equation}\label{eq:5-32}
	\begin{split}
		&\,[\phi(\xi_2(0))+A_t(\xi_2(0),x)]-[\phi(\xi_1(0))+A_t(\xi_1(0),x)]\\
		=&\,[\phi(\xi_2(0))+A_{2t}(\xi_2(0),\xi_1(2t))]-[\phi(\xi_1(0))+A_{2t}(\xi_1(0),\xi_1(2t))]\\
		&\,\ +A_{2t}(\xi_1(0),\xi_1(2t))-A_{t}(\xi_1(0),x)+A_{t}(\xi_2(0),x)-A_{2t}(\xi_2(0),\xi_1(2t))\\
		\geqslant&\,A_{t}(x,\xi_1(2t))+A_{t}(\xi_2(0),x)-A_{2t}(\xi_2(0),\xi_1(2t))\\
		\geqslant&\,C_{2\lambda_H,5} t|p_3(t)-p_1(t)|^2.
	\end{split}
\end{equation}
Combining the two inequalities \eqref{eq:5-31} and \eqref{eq:5-32} above we have that
\begin{align*}
	|p_3(t)-p_1(t)|\leqslant\frac{C_{2\lambda_H,2}}{C_{2\lambda_H,1}\cdot C_{2\lambda_H,5}}\cdot\frac 1t|y-x|.
\end{align*}
Applying Proposition \ref{pro:reg2} again we obtain
\begin{align*}
	|p_2(t)-p_3(t)|\leqslant\frac{C_{2\lambda_H,2}}{t}|y-x|.
\end{align*}
It follows that
\begin{align*}
	|p^*-D\phi(x)|\leqslant&\,|p_2(t)-p_1(t)|\leqslant|p_2(t)-p_3(t)|+|p_3(t)-p_1(t)|\\
	\leqslant&\,\frac 1t\cdot\Big(C_{2\lambda_H,2}+\frac{C_{2\lambda_H,2}}{C_{2\lambda_H,1}\cdot C_{2\lambda_H,5}}\Big)\cdot|y-x|=\frac{C_H}{t}|y-x|.
\end{align*}
This completes the proof by recalling that $D^{+}\phi(y)=\mbox{\rm co}\,D^{*}\phi(y)$ (see Proposition \ref{pro:scsv} (3)).
\end{proof}

\subsection{Global propagation of cut points along generalized characteristics}

Now, we will show that any generalized characteristic propagates cut points globally. 

\begin{The}\label{thm:propagation_cut}
Suppose $H:T^*M\to\R$ is a Tonelli Hamiltonian, and $\phi$ is a weak KAM solution of \eqref{eq:HJ_wk}. Then the following hold
\begin{enumerate}[\rm (1)]
	\item For any $x\in M\setminus\text{\rm Cut}\,(\phi)$, let $\gamma_x:(-\infty,\tau_\phi(x)]\to M$ be the unique $(\phi,H)$-calibrated curve with $\gamma_x(0)=x$. Then $\gamma_x$ is the unique solution of the differential inclusion
	\begin{equation}\label{eq:GCI}\tag{GC}
		\begin{cases}
			\dot{\gamma}(t)\in\text{\rm co}\,H_p(\gamma(t),D^+\phi(\gamma(t))),\qquad a.e.\  t\in\R,\\
			\gamma(0)=x
		\end{cases}
	\end{equation}
	on $(-\infty,\tau_\phi(x)]$.
	\item If $\gamma:[0,+\infty)\to M$ satisfies \eqref{eq:GCI} and $\gamma(0)\in \text{\rm Cut}\,(\phi)$, then $\gamma(t)\in\text{\rm Cut}\,(\phi)$ for all $t\geqslant0$.
\end{enumerate}
\end{The}

\begin{proof}
Recall that $\gamma_x$ is a backward classical characteristic and so it satisfies \eqref{eq:GCI} on $(-\infty,\tau_\phi(x)]$. Now, we turn to prove the uniqueness. Indeed, suppose $\gamma:(-\infty,\tau_\phi(x)]\to M$ is a solution of \eqref{eq:GCI} with $\gamma(0)=x$. Set
\begin{gather*}
	\tau_+=\sup\{\tau\geqslant0: \gamma(t)=\gamma_x(t),\ \forall t\in[0,\tau]\},\\
	\tau_-=\inf\{\tau\leqslant0: \gamma(t)=\gamma_x(t),\ \forall t\in[0,\tau]\}.
\end{gather*}
It is obvious that $0\leqslant\tau_+\leqslant\tau_\phi(x)$ and $-\infty\leqslant\tau_-\leqslant0$.

If $\tau_+<\tau_\phi(x)$, set $\tau_{\Delta+}=\min\{\tau_\phi(x)-\tau_+,T_H\}>0$ where $T_H$ is given in Theorem \ref{thm:C11}. Notice that $\gamma(\tau_+)=\gamma_x(\tau_+)$ and it follows that
\begin{align*}
	\gamma(t),\gamma_x(t)\in B(\gamma_x(\tau_+),\frac 14\lambda_H\tau_{\Delta+}),\qquad\forall t\in[\tau_+,\tau_++\frac 14\tau_{\Delta+}),
\end{align*}
with $\lambda_H$ also determined in Theorem \ref{thm:C11}. Together with Theorem \ref{thm:C11} we obtain that
\begin{align*}
	|p-D\phi(\gamma_x(t))|\leqslant\frac{C_H}{\tau_{\Delta+}/2}|\gamma(t)-\gamma_x(t)|,\qquad\forall t\in[\tau_+,\tau_++\frac 14\tau_{\Delta+}),\ p\in D^+\phi(\gamma(t)).
\end{align*}
Since $\gamma$ satisfies \eqref{eq:GCI} we have that
\begin{align*}
	\frac{d}{dt}|\gamma(t)-\gamma_x(t)|\leqslant&\,|\dot{\gamma}(t)-\dot{\gamma}_x(t)|=\Big|\int_{D^+\phi(\gamma(t))}H_p(\gamma(t),p)\ d\nu_t-H_p(\gamma_x(t),D\phi(\gamma_x(t)))\Big|\\
	\leqslant&\,C'_H|\gamma(t)-\gamma_x(t)|+C'_H\int_{D^+\phi(\gamma(t))}|p-D\phi(\gamma_x(t))|\ d\nu_t\\
	\leqslant&\,C'_H\Big(1+\frac{2C_H}{\tau_{\Delta+}}\Big)|\gamma(t)-\gamma_x(t)|
\end{align*}
for almost all $t\in[\tau_+,\tau_++\frac 14\tau_{\Delta+})$, where $\{\nu_t\}$ is a family of probability measure supported on $D^+\phi(\gamma(t))$ and
\begin{align*}
	C'_H=\sup\{|DH(x,p)|: x\in M, p\in D^+\phi(x), \phi\ \text{is a weak KAM solution of \eqref{eq:HJ_wk}}\}<\infty.
\end{align*}
Gronwall's inequality implies
\begin{align*}
	|\gamma(t)-\gamma_x(t)|\leqslant\exp\Big(C'_H\Big[1+\frac{2C_H}{\tau_{\Delta+}}\Big](t-\tau_+)\Big)|\gamma(\tau_+)-\gamma_x(\tau_+)|=0
\end{align*}
for all $t\in[\tau_+,\tau_++\frac 14\tau_{\Delta+})$. This contradicts the definition of $\tau_+$. A similar argument shows that $\tau_-=-\infty$. This completes the proof of (1).

Now, we prove (2) by contradiction. Suppose that $\gamma:[0,+\infty)\to M$ satisfies \eqref{eq:GCI} and $\gamma(0)\in \text{\rm Cut}\,(\phi)$, and there exists $t>0$ such that $\gamma(t)\in M\setminus\text{Cut}\,(\phi)$. Then
\begin{align*}
	\gamma_x(s)=
	\begin{cases}
		\gamma_1(s),& s\in(-\infty,0],\\
		\gamma(s),& s\in(0,\tau_{\phi}(x)],
	\end{cases}
\end{align*}
define a curve $\gamma_x:(-\infty,\tau_{\phi}(x)]\to M$, where $\gamma_1:(-\infty,0]\to M$ is any solution of \eqref{eq:GCI} on $(-\infty,0]$. From (2) we conclude that $\gamma_x$ is the unique solution of \eqref{eq:GCI} on $(-\infty,\tau_{\phi}(x)]$, i.e., $\gamma_x$ is exactly a $(\phi,H)$-calibrated curve on $(-\infty,\tau_{\phi}(x)]$. It follows that $\gamma\vert_{[0,t]}$ is a $(\phi,H)$-calibrated curve. This contradicts to the assumption that $\gamma(0)\in\text{Cut}\,(\phi)$. Thus, $\gamma(t)\in\text{\rm Cut}\,(\phi)$ for all $t\geqslant0$.
\end{proof}


\subsection{Propagation of singular points}

In this section, we will do some further analysis to study the propagation of singular points $\text{\rm Sing}\,(\phi)$. We begin with the local propagation problem for a pair $(\phi,H)$ with $\phi\in\text{\rm SCL}\,(M)$ and $H$ a Tonelli Hamiltonian. 

\begin{The}\label{thm:local_prop}
	Suppose $H$ is a Tonelli Hamiltonian and $\phi\in\text{\rm SCL}\,(M)$. If $x\in\text{\rm Sing}\,(\phi)$ satisfies
	\begin{equation}\label{eq:condition _sing}
		\mathbf{p}^\#_{\phi,H}(x)\not\in D^*\phi(x)
	\end{equation}
	then there exists a strict singular characteristic $\gamma:[0,+\infty)\to M$ with $\gamma(0)=x$ for the pair $(\phi,H)$ such that for some $\delta>0$ we have
	\begin{align*}
		\gamma(t)\in\text{\rm Sing}\,(\phi),\qquad \forall t\in[0,\delta].
	\end{align*}
\end{The}

\begin{Rem}
	If $\phi$ is a weak KAM solution of \eqref{eq:HJ_wk}, condition \eqref{eq:condition _sing} holds true automatically because $H(x,\mathbf{p}^\#_{\phi,H}(x))<0$ for all $x\in\text{Sing}\,(\phi)$ and $H(x,p)=0$ for all $p\in D^*\phi(x)$ (See, Theorem 6.4.9 in \cite{Cannarsa_Sinestrari_book}). It is interesting to compare this theorem with Lemma 4.1 and Theorem 4.2 in \cite{Cannarsa_Yu2009}.
\end{Rem}

\begin{proof}
	We just use the strict singular characteristic $\gamma$ obtained in Theorem \ref{thm:right_continuous2}. We argue by contradiction. Suppose there exists a sequence of positive reals $\{t_k\}$, $t_k\searrow 0^+$ as $k\to\infty$,  such that $\phi$ is differentiable at every $\gamma(t_k)$. Then, by Theorem \ref{thm:right_continuous2} we have $\mathbf{p}^\#_{\phi,H}(x)=\lim_{k\to\infty}D\phi(\gamma(t_k))\in D^*\phi(x)$. This leads to a contradiction to condition \eqref{eq:condition _sing}.
\end{proof}

\begin{The}\label{thm:global pro}
	Suppose $H:T^*M\to\R$ is a Tonelli Hamiltonian, and $\phi$ is a weak KAM solution of \eqref{eq:HJ_wk}. If $\gamma:[0,+\infty)\to M$ is a strict singular characteristic for the pair $(\phi,H)$ and $\gamma(0)\in\text{\rm Cut}\,(\phi)$, then the following hold
	\begin{enumerate}[\rm (1)]
		\item $\gamma(t)\in\text{\rm Cut}\,(\phi)$ for all $t\geqslant0$.
		\item $\text{\rm supp}\,(\chi_{\text{\rm Sing}\,(\phi)}(\gamma)\mathscr{L}^1)=[0,+\infty)$, where $\chi_{\text{\rm Sing}\,(\phi)}$ is the indicator of the set $\text{\rm Sing}\,(\phi)$, and $\mathscr{L}^1$ stands for the Lebesgue measure on $[0,+\infty)$.
		\item If $\gamma$ satisfies \eqref{eq:lambda_phi_H} as in Theorem \ref{thm:right_continuous2}, then $\text{\rm int}\,(\{t\in[0,+\infty):\gamma(t)\in\text{\rm Sing}\,(\phi)\})$ is dense in $[0,+\infty)$, where $\text{\rm int}\,(A)$ stands for the interior of $A\subset\R$.
	\end{enumerate}
\end{The}

\begin{proof}
Assertion (1) is a direct consequence of Theorem \ref{thm:propagation_cut}. For (2), let $(t_1,t_2)\subset[0,+\infty)$ be any open interval. Suppose that $\chi_{\text{\rm Sing}\,(\phi)}(\gamma)\mathscr{L}^1((t_1,t_2))=0$, that is, $\phi$ is differentiable at $\gamma(s)$ for almost all $s\in(t_1,t_2)$. Then, since $\gamma:[0,+\infty)\to M$ is a strict singular characteristic and $\phi$ is a weak KAM solution of \eqref{eq:HJ_wk}, we have that
\begin{align*}
	\phi(\gamma(t_2))-\phi(\gamma(t_1))=\int_{t_1}^{t_2}\Big\{L(\gamma(s),\dot{\gamma}(s))+H(\gamma(s),D\phi(\gamma(s)))\Big\}\ ds=\int_{t_1}^{t_2}L(\gamma(s),\dot{\gamma}(s))\ ds.
\end{align*}
This implies $\gamma\vert_{[t_1,t_2]}$ is a calibrated curve for $(\phi,H)$ and leads to a contradiction to assertion (1). Thus, we have $\chi_{\text{\rm Sing}\,(\phi)}(\gamma)\mathscr{L}^1((t_1,t_2))>0$. This completes the proof of (2).

Now we turn to prove (3). For any open interval $(t_1,t_2)\subset[0,+\infty)$, by assertion (2), there exists $t_3\in(t_1,t_2)$ such that $\gamma(t_3)\in\text{\rm Sing}\,(\phi)$. Then Theorem \ref{thm:local_prop} implies there exists $t_3<t_4<t_2$ such that $\gamma(t)\in\text{\rm Sing}\,(\phi)$ for all $t\in[t_3,t_4]$. Thus, the set $\text{\rm int}\,(\{t\in[0,+\infty):\gamma(t)\in\text{\rm Sing}\,(\phi)\})$ is dense in $[0,+\infty)$.
\end{proof}

\subsection{Concluding remarks}

The uniqueness of strict singular characteristic is still open in general. 
We raise the following problems:
\begin{enumerate}[(A)]
	\item Does the uniqueness of strict singular characteristic $\gamma:[0,+\infty)\to M$ from a given initial point $\gamma(0)=x$ holds in general? or
	\item can one have an example where the two methods used in the proofs of Theorem \ref{thm:msc exs} and Theorem \ref{thm:pos neg}, respectively, lead to two distinct strict singular characteristics?
\end{enumerate}

We remark that for mechanical Hamiltonians, the associated strict singular characteristic is exactly the solution of the generalized gradient system. In Euclidean case, Albano proved in \cite{Albano2016_1} a global propagation result in the following sense: if $x\in\text{Sing}\,(\phi)$, then the unique solution $\gamma:[0,\infty)\to\R^n$ of the generalized gradient system propagates singularities of the weak KAM solution $\phi$, i.e., $\gamma(t)\in\text{Sing}\,(\phi)$ for all $t\geqslant0$. Thus, it is natural to raise the following problem:
\begin{enumerate}[(C)]
	\item Does global propagation of $\text{Sing}\,(\phi)$ hold in general? 
\end{enumerate}
\begin{enumerate}[(D)]
	\item At least for the Hamiltonian $H(x,p)$ with quadratic dependence on $p$, is it true that $\text{Cut}\,(\phi)=\overline{\text{Sing}\,(\phi)}$?
\end{enumerate}

We finally remark that the energy dissipation term in the discrete scheme comes from the non-commutativity of the positive and negative Lax-Oleinik semigroups. Let us recall the explanation of cut time function in terms of the commutators of $T^\pm_t$ in \cite{Cannarsa_Cheng_Hong2023}. For any weak KAM solution $\phi$ of \eqref{eq:HJ_wk}, the cut time function is given by
\begin{align*}
	\tau_{\phi}(x)=\sup\{t\geqslant0: (T^-_t\circ T^+_t-T^+_t\circ T^-_t)\phi(x)=0\},
\end{align*}
and the cut locus by $\text{Cut}\,(\phi)=\{x\in M:\tau_{\phi}(x)=0\}$. On the other hand, for any $0<t<\tau(\phi)$ and $x\in M$ we have that
\begin{align*}
	(T^-_t\circ T^+_t-T^+_t\circ T^-_t)\phi(x)=&\,\phi(x)-T^+_t\phi(x)=-\int^t_0\frac{d}{ds}T^+_s\phi(x)\ ds\\
	=&\,-\int^t_0H(x,DT^+_s\phi(x))\ ds.
\end{align*}
So, the term $H(x,\mathbf{p}^\#_{\phi,H}(x))$ embodies the non-commutativity of Lax-Oleinik operators $T^\pm$, and the creation and evolution of singularities.

\section{Mass transport aspect of strict singular characteristic}

To study the mass transport along strict singular characteristics, one cannot apply the standard DiPerna-Lions theory to deduce the continuity equation for which the solution is a curve of probability measures determined by the measures driven by the flow of the vector field, since the collision and focus of the classical characteristics. However, one can use a lifting method to consider a dynamical system on the set of strict singular characteristics even without uniqueness. 



Let $\phi:M\to\R$ be a semiconcave function and $H$ be a Hamiltonian satisfying \text{\rm (H1)-(H3)}. Fix $T>0$, let $\mathscr{S}_T$ be the family of strict singular characteristics $\gamma:[0,T]\to M$ for the pair $(\phi,H)$. Consider the set-valued map $\Gamma:M\rightrightarrows\mathscr{S}_T$ defined by
\begin{align*}
	x\mapsto\Gamma(x)=\{\gamma\in\mathscr{S}_T: \gamma(0)=x\}.
\end{align*}

\begin{Pro}\label{pro:measuable}
Suppose $\phi$ is a semiconcave function on $M$ and $H$ is a Hamiltonian satisfying \text{\rm (H1)-(H3)}. Then the following holds:
\begin{enumerate}[\rm (1)]
	\item $\mathscr{S}_T$ is a separable and compact metric subspace of $C^0([0,T],M)$ under the $C^0$-norm.
	\item The set-valued map $\Gamma$ is nonempty, compact-valued and measurable.
	\item The set-valued map $\Gamma$ admits a measurable selection $\gamma:M\to\mathscr{S}_T$, $x\mapsto\gamma(x,\cdot)$	such that $\gamma(x,0)=x$ for all $x\in M$.
\end{enumerate}
\end{Pro}

\begin{proof}
The assertion of (1) is immediate from Corollary \ref{cor:stability}. Now we turn to the proof of (2). For any $x\in M$, by the existence of strict singular characteristic from $x$ (Theorem \ref{thm:msc exs}) and the stability property (Theorem \ref{thm:stability}), $\Gamma(x)\not=\varnothing$ and $\Gamma(x)$ is compact. Let $K\subset\mathscr{S}_T$ be any closed subset, $K$ is compact. The pre-image
\begin{align*}
	\{x\in M: \Gamma(x)\cup K\not=\varnothing\}=\{\gamma(0): \gamma\in K\}
\end{align*}
is also compact, since the continuity of the map $\gamma\mapsto\gamma(0)$. This implies $x\mapsto\Gamma(x)$ is also measurable. Finally, (3) is a direct consequence of Proposition \ref{pro:selection}.
\end{proof}

Now for any measurable map $\gamma$ in Proposition \ref{pro:measuable} we introduce a map
\begin{align*}
	\Phi_\gamma^t(x)=\gamma(x,t),\qquad t\in[0,T],\ x\in M.
\end{align*}
For any $t\in[0,T]$, $\Phi_{\gamma}^t:M\to M$ is measurable, and $\Phi^0_\gamma=\text{id}$.

\begin{The}\label{thm:CE}
Suppose $\phi$ is a semiconcave function on $M$, $H:T^*M\to\R$ is a Hamiltonian satisfying \text{\rm (H1)-(H3)}, and $\gamma:M\to\mathscr{S}_T$ satisfies the condition in Proposition \ref{pro:measuable} (3). Then for any $\bar{\mu}\in\mathscr{P}(M)$, the set of Borel probability measures on $M$, the curve $\mu_t=(\Phi^t_\gamma)_{\#}\bar{\mu}$, $t\in[0,T]$ in $\mathscr{P}(M)$ satisfies the continuity equation
\begin{equation}\label{eq:CE}\tag{CE}
	\begin{cases}
		\frac d{dt}\mu+\text{\rm div}(H_p(x,\mathbf{p}^{\#}_{\phi,H}(x))\cdot\mu)=0,\\
		\mu_0=\bar{\mu}.
	\end{cases}
\end{equation}
and we have the following properties:
\begin{enumerate}[\rm (1)]
	\item $\int_M\phi\mu_T-\int_M\phi\mu_0=\int^T_0\int_M\big\{L(x,H_p(x,\mathbf{p}^\#_{\phi,H}(x)))+H(x,\mathbf{p}^\#_{\phi,H}(x))\big\}\ d\mu_sds$.
	\item For any $t\in[0,T)$, $g\in C^\infty(M)$,
	\begin{align*}
		\frac{d^+}{dt}\int_Mg\ d\mu_t=\int_M\nabla g(x)\cdot H_p(x,\mathbf{p}^\#_{\phi,H}(x))\ d\mu_t.
	\end{align*}
	\item For any $t\in[0,T)$ and $f\in C_c(T^*M,\R)$
	\begin{align*}
		\int_Mf(x,\mathbf{p}^\#_{\phi,H}(x))\ d\mu_t=\lim_{\tau\to t^+}\frac 1{\tau-t}\int^\tau_t\int_Mf(x,\mathbf{p}^\#_{\phi,H}(x))\ d\mu_sds.
	\end{align*}
\end{enumerate}
\end{The}

\begin{proof}
	Since $\mu_0=(\Phi^0_{\gamma})_{\#}\mu_0=(\text{id})_{\#}\mu_0=\bar{\mu}$, to prove \eqref{eq:CE} it suffices to check the continuity equation. For any $g\in C^{\infty}(M)$ and $0\leqslant r<t\leqslant T$ we have
	\begin{align*}
		&\,\int_Mg\ d\mu_t-\int_Mg\ d\mu_r=\int_Mg\circ\Phi_\gamma^t\ d\bar{\mu}-\int_Mg\circ\Phi_\gamma^r\ d\bar{\mu}\\
		=&\,\int_Mg(\gamma(x,t))-g(\gamma(x,r))\ d\bar{\mu}=\int_M\int^t_r\frac d{ds}g(\gamma(x,s))\ dsd\bar{\mu}\\
		=&\,\int_M\int^t_r\nabla g(\gamma(x,s))\cdot\dot{\gamma}(x,s)\ dsd\bar{\mu}=\int_M\int^t_r\nabla g(\gamma(x,s))\cdot H_p(\gamma(x,s),\mathbf{p}^{\#}_{\phi,H}(\gamma(x,s)))\ dsd\bar{\mu}\\
		=&\,\int^t_r\int_M\nabla g(\gamma(x,s))\cdot H_p(\gamma(x,s),\mathbf{p}^{\#}_{\phi,H}(\gamma(x,s)))\ d\bar{\mu}ds\\
		=&\,\int^t_r\int_M\nabla g(x)\cdot H_p(x,\mathbf{p}^{\#}_{\phi,H}(x))\ d\mu_sds.
	\end{align*}
	It follows that $t\mapsto\int_Mg\ d\mu_t$ is absolutely continuous on $[0,T]$ and
	\begin{align*}
		\frac d{dt}\int_Mg\ d\mu_t=\int_M\nabla g(x)\cdot H_p(x,\mathbf{p}^{\#}_{\phi,H}(x))\ d\mu_t,\qquad a.e.\ t\in[0,T].
	\end{align*}
	Now, suppose $\alpha\in C_c^{\infty}((0,T))$. We find
	\begin{align*}
		&\,\int_{[0,T]\times M}\Big\{\frac d{dt}[\alpha(t)g(x)]+H_p(x,\mathbf{p}^{\#}_{\phi,H}(x))\cdot\nabla[\alpha(t)g(x)]\Big\}\ d\mu\\
		=&\,\int^T_0\int_M\big\{\alpha'(t)g(x)+\alpha(t) H_p(x,\mathbf{p}^{\#}_{\phi,H}(x))\cdot\nabla g(x)\big\}\ d\mu_tdt\\
		=&\,\int^T_0\Big\{\alpha'(t)\int_Mg(x)\ d\mu_t+\alpha(t)\int_M\nabla g(x)\cdot H_p(x,\mathbf{p}^{\#}_{\phi,H}(x))\ d\mu_t\Big\}\ dt\\
		=&\,\Big[\alpha(t)\int_Mg(x)\ d\mu_t\Big]\Big|^T_0=0.
	\end{align*}
	Finally, we note that 
	\begin{align*}
		\Big\{\sum^N_{i=1}\alpha_i(t)g_i(x): \alpha_i\in C^{\infty}_c((0,T)), g_i\in C^{\infty}(M), i=1,\ldots,N\Big\}
	\end{align*}
	is a dense subset of $C^{\infty}_c((0,T)\times M)$ in the topology of uniform convergence on compact subsets by the Stone-Weierstrass theorem. It follows that
	\begin{align*}
		\int_{[0,T]\times M}\Big\{\frac d{dt}f(t,x)+H_p(x,\mathbf{p}^{\#}_{\phi,H}(x))\cdot\nabla f(t,x)\Big\}\ d\mu=0,\qquad \forall f\in C^{\infty}_c((0,T)\times M).
	\end{align*}

    Now we turn to prove (1). Recall that $\gamma(x,\cdot)$ is a strict singular characteristic for every $x\in M$. Then, by Fubini's theorem we have that
    \begin{align*}
    	&\,\int_M\phi\ d\mu_T-\int_M\phi\ d\mu_0\\
    	=&\,\int_M\big\{\phi(\gamma(x,T))-\phi(x)\big\}\ d\bar{\mu}\\
    	=&\,\int_M\int^T_0\Big\{L(\gamma(x,s),H_p(\gamma(x,s),\mathbf{p}^\#_{\phi,H}(\gamma(x,s))))+H(\gamma(x,s),\mathbf{p}^\#_{\phi,H}(\gamma(x,s)))\Big\}\ dsd\bar{\mu}\\
    	=&\,\int^T_0\int_M\Big\{L(\gamma(x,s),H_p(\gamma(x,s),\mathbf{p}^\#_{\phi,H}(\gamma(x,s))))+H(\gamma(x,s),\mathbf{p}^\#_{\phi,H}(\gamma(x,s)))\Big\}\ d\bar{\mu}ds\\
    	=&\,\int^T_0\int_M\Big\{L(x,H_p(x,\mathbf{p}^\#_{\phi,H}(x)))+H(x,\mathbf{p}^\#_{\phi,H}(x))\Big\}\ d\mu_sds.
    \end{align*}
    The proof of (2) is a consequence of the property of broken characteristics. Indeed,
    \begin{align*}
    	\frac{d^+}{dt}\int_Mg\ d\mu_t=&\,\lim_{\tau\to t^+}\frac 1{\tau-t}\Big(\int_Mg\ d\mu_\tau-\int_Mg\ d\mu_t\Big)\\
    	=&\,\lim_{\tau\to t^+}\int_M\frac 1{\tau-t}(g(\gamma(x,\tau))-g(\gamma(x,t)))\ d\bar{\mu}\\
    	=&\,\int_M\lim_{\tau\to t^+}\frac 1{\tau-t}(g(\gamma(x,\tau))-g(\gamma(x,t)))\ d\bar{\mu}\qquad(\text{Dominated convergence theorem}) \\ 
    	=&\,\int_M\nabla g(\gamma(x,t)))\cdot H_p(\gamma(x,t)),\mathbf{p}^\#_{\phi,H}(\gamma(x,t))))\ d\bar{\mu}\qquad(\text{Theorem \ref{thm:right_derivative}}) \\
    	=&\,\int_M\nabla g(x)\cdot H_p(x,\mathbf{p}^\#_{\phi,H}(x))\ d\mu_t.
    \end{align*}
    Finally, for (3),
    \begin{align*}
    	&\,\lim_{\tau\to t^+}\frac 1{\tau-t}\int^\tau_t\int_Mf(x,\mathbf{p}^\#_{\phi,H}(x))\ d\mu_sds\\
    	=&\,\lim_{\tau\to t^+}\frac 1{\tau-t}\int^\tau_t\int_Mf(\gamma(x,s),\mathbf{p}^\#_{\phi,H}(\gamma(x,s)))\ d\bar{\mu}ds\\
    	=&\,\lim_{\tau\to t^+}\int_M\frac 1{\tau-t}\int^\tau_tf(\gamma(x,s),\mathbf{p}^\#_{\phi,H}(\gamma(x,s)))\ dsd\bar{\mu}\qquad(\text{Fubini theorem})\\
    	=&\,\int_M\lim_{\tau\to t^+}\frac 1{\tau-t}\int^\tau_tf(\gamma(x,s),\mathbf{p}^\#_{\phi,H}(\gamma(x,s)))\ dsd\bar{\mu}\qquad(\text{Dominated convergence theorem})\\
    	=&\,\int_Mf(\gamma(x,t),\mathbf{p}^\#_{\phi,H}(\gamma(x,t)))\ d\bar{\mu}\qquad(\text{Theorem \ref{thm:measure_H}})\\
    	=&\,\int_Mf(x,\mathbf{p}^\#_{\phi,H}(x))\ d\mu_t
    \end{align*}
    This completes the proof.
\end{proof}

\begin{The}\label{thm:mass sing}
Under the assumption of Theorem \ref{thm:CE}, if $H$ is a Tonelli Hamiltonian and $\phi$ is a weak KAM solution of \eqref{eq:HJ_wk}, we have that
\begin{align*}
	\mu_{t_1}(\text{\rm Cut}\,(\phi))\leqslant\mu_{t_2}(\text{\rm Cut}\,(\phi)),\qquad \mu_{t_1}(\overline{\text{\rm Sing}\,(\phi)})\leqslant\mu_{t_2}(\overline{\text{\rm Sing}\,(\phi)}),\qquad\forall 0\leqslant t_1\leqslant t_2\leqslant T.
\end{align*}
\end{The}

\begin{proof}
	For any $0\leqslant t_1\leqslant t_2\leqslant T$, Theorem \ref{thm:propagation_cut} implies
	\begin{align*}
		\{x\in M: \gamma(x,t_1)\in\text{Cut}\,(\phi)\}\subset\{x\in M: \gamma(x,t_2)\in\text{Cut}\,(\phi)\}.
	\end{align*}
	It follows that
	\begin{align*}
		\mu_{t_1}(\text{\rm Cut}\,(\phi))=&\,\bar{\mu}((\Phi^{t_1}_\gamma)^{-1}(\text{Cut}\,(\phi)))=\bar{\mu}(\{x\in M: \gamma(x,t_1)\in\text{Cut}\,(\phi)\})\\
		\leqslant&\,\bar{\mu}(\{x\in M: \gamma(x,t_2)\in\text{Cut}\,(\phi)\})=\mu_{t_2}(\text{\rm Cut}\,(\phi))
	\end{align*}
Similarly, by Theorem \ref{thm:global_prop}, we have $\mu_{t_1}(\overline{\text{\rm Sing}\,(\phi)})\leqslant\mu_{t_2}(\overline{\text{\rm Sing}\,(\phi)})$ for any $0\leqslant t_1\leqslant t_2\leqslant T$.
\end{proof}

As a direct consequence of Corollary \ref{cor:stability_H}, we obtain

\begin{Cor}\label{cor:mass H}
Suppose $\phi\in\text{\rm SCL}\,(M)$ and $H:T^*M\to\R$ is a Tonelli Hamiltonian. For any $T>0$, let $\tilde{\mathscr{S}}_T$ be the subset of $\gamma\in\mathscr{S}_T$ satisfying \eqref{eq:lambda_phi_H} for all $0\leqslant t_1\leqslant t_2\leqslant T$. Then $\tilde{\mathscr{S}}_T$ is a nonempty compact subset of $C^0([0,T],M)$ under the $C^0$-norm.
\end{Cor}

From Corollary \ref{cor:mass H}, we can conclude that all the statements for $\mathscr{S}_T$ in Proposition \ref{pro:measuable} also holds true for $\tilde{\mathscr{S}}_T$. Therefore, we can have some refinement of Theorem \ref{thm:CE}.

\begin{Cor}
Under the assumption of Theorem \ref{thm:CE}, if $\phi\in\text{\rm SCL}\,(M)$, $H$ is a Tonelli Hamiltonian, and we take $\tilde{\mathscr{S}}_T$ instead of $\mathscr{S}_T$, then the following holds:
\begin{enumerate}[\rm (3')]
	\item For any $t\in[0,T)$ and $f\in C_c(T^*M,\R)$
	\begin{align*}
		\int_Mf(x,\mathbf{p}^\#_{\phi,H}(x))\ d\mu_t=\lim_{\tau\to t^+}\int_Mf(x,\mathbf{p}^\#_{\phi,H}(x))\ d\mu_\tau.
	\end{align*}
\end{enumerate}
\begin{enumerate}[\rm (4)]
	\item $\int_MH(x,\mathbf{p}^\#_{\phi,H}(x))\ d\mu_{t_2}-\int_MH(x,\mathbf{p}^\#_{\phi,H}(x))\ d\mu_{t_1}\leqslant\lambda_{\phi,H}(t_2-t_1)$ for all $0\leqslant t_1\leqslant t_2\leqslant T$.
\end{enumerate}
\end{Cor}

\begin{proof}
For (3') we note that by Theorem \ref{thm:right_continuous2} and the dominated convergence theorem
\begin{align*}
	\lim_{\tau\to t^+}\int_Mf(x,\mathbf{p}^\#_{\phi,H}(x))\ d\mu_\tau=&\,\lim_{\tau\to t^+}\int_Mf(\gamma(x,\tau),\mathbf{p}^\#_{\phi,H}(\gamma(x,\tau)))\ d\bar{\mu}\\
	=&\,\int_Mf(\gamma(x,t),\mathbf{p}^\#_{\phi,H}(\gamma(x,t)))\ d\bar{\mu}\\
	=&\,\int_Mf(x,\mathbf{p}^\#_{\phi,H}(x))\ d\mu_t.
\end{align*}

For (4), by Theorem \ref{thm:right_continuous2} again, we have that
\begin{align*}
	&\,\int_MH(x,\mathbf{p}^\#_{\phi,H}(x))\ d\mu_{t_2}-\int_MH(x,\mathbf{p}^\#_{\phi,H}(x))\ d\mu_{t_1}\\
	=&\,\int_M\Big\{H(\gamma(x,t_2),\mathbf{p}^\#_{\phi,H}(\gamma(x,t_2)))-H(\gamma(x,t_1),\mathbf{p}^\#_{\phi,H}(\gamma(x,t_1)))\Big\}\ d\bar{\mu}\\
	\leqslant&\,\int_M\lambda_{\phi,H}(t_2-t_1)\ d\bar{\mu}=\lambda_{\phi,H}(t_2-t_1).
\end{align*}
This completes the proof.
\end{proof}

\appendix

\section{Maximal slope curves in the time-dependent case}

In this appendix, we will extend some results in Section 3 to the time-dependent case. We suppose that $H:T^*M\to\R$ satisfies conditions (H1)-(H3) and $\phi:\R\times M\to R$ is a locally Lipschitz function. For any $(t,x)\in\R\times M$, the set $\arg\min\{q+H(x,p): (q,p)\in D^+\phi(t,x)\}$ is a singleton since $D^+\phi(t,x)$ is a compact convex set and $H(x,\cdot)$ is strictly convex. We set
\begin{align*}
	(\mathbf{q}^\#_{\phi,H}(t,x),\mathbf{p}^\#_{\phi,H}(t,x)):=\arg\min\{q+H(x,p): (q,p)\in D^+\phi(t,x)\},\qquad (t,x)\in\R\times M.
\end{align*}

\begin{Lem}\label{lem:Borel2}
Suppose $H$ satisfies conditions (H1)-(H3) and $\phi$ is a locally semiconcave function on $\R\times M$. Then the map $(t,x)\mapsto(\mathbf{q}^\#_{\phi,H}(t,x),\mathbf{p}^\#_{\phi,H}(t,x))$ is Borel measurable.
\end{Lem}

\begin{proof}
The proof is similar to that of Lemma \ref{lem:Borel}.
\end{proof}




\begin{Pro}
Suppose $H$ satisfies conditions (H1)-(H3), $\phi$ is a locally semiconcave function on $\R\times M$ and $(\mathbf{q}(t,x),\mathbf{p}(t,x))$ is a Borel measurable selection of the superdifferential $D^+\phi(t,x)$. Then, for any absolutely continuous curve $\gamma:[0,t]\to M$ we have that
\begin{align*}
	\phi(t,\gamma(t))-\phi(0,\gamma(0))\leqslant\int^{t}_{0}\Big\{L(\gamma(s),\dot{\gamma}(s))+\mathbf{q}(s,\gamma(s))+H(\gamma(s),\mathbf{p}(s,\gamma(s)))\Big\}\ ds.
\end{align*}
The equality above holds if and only if
\begin{align*}
	\dot{\gamma}(s)=H_p(\gamma(s),\mathbf{p}(s,\gamma(s))),\qquad a.e.\ s\in[0,t].
\end{align*}
\end{Pro}

\begin{proof}
By Proposition \ref{pro:scsv} (3), for almost all $s\in[0,t]$ we have that
\begin{align*}
	&\,\min_{(q,p)\in D^+\phi(s,\gamma(s))}\{q+\langle p,\dot{\gamma}(s)\rangle\}=\frac{d^+}{ds}\phi(s,\gamma(s))=\frac{d^-}{ds}\phi(s,\gamma(s))\\
	=&\,-\min_{(q,p)\in D^+\phi(s,\gamma(s))}\{-q+\langle p,-\dot{\gamma}(s)\rangle\}=\max_{(q,p)\in D^+\phi(s,\gamma(s))}\{q+\langle p,\dot{\gamma}(s)\rangle\}.
\end{align*}
Invoking Young's inequality
\begin{align*}
	\phi(t,\gamma(t))-\phi(0,\gamma(0))=&\,\int^{t}_{0}\frac{d}{ds}\phi(s,\gamma(s))\ ds=\int^{t}_{0}\mathbf{q}(s,\gamma(s))+\langle \mathbf{p}(s,\gamma(s)),\dot{\gamma}(s)\rangle\\
	\leqslant&\,\int^{t}_{0}\Big\{L(\gamma(s),\dot{\gamma}(s))+\mathbf{q}(s,\gamma(s))+H(\gamma(s),\mathbf{p}(s,\gamma(s)))\Big\}\ ds
\end{align*}
and the equality holds if and only if $\dot{\gamma}(s)=H_p(\gamma(s),\mathbf{p}(s,\gamma(s)))$ for almost all $s\in[0,t]$.
\end{proof}

\begin{defn}\label{defn:msc2}
Let $\phi$ be a semiconcave function on $\R\times M$, $H:T^*M\to\R$ be a Hamiltonian satisfying \text{\rm (H1)-(H3)} and $(\mathbf{q}(t,x),\mathbf{p}(t,x))$ be a Borel measurable selection of the superdifferential $D^+\phi(t,x)$.
\begin{enumerate}[(1)]
	\item We call a locally absolutely continuous curve $\gamma:I\to M$ a \emph{maximal slope curve for the pair $(\phi,H)$ and the selection $\mathbf{p}(t,x)$}, where $I$ is any interval which can be the whole real line, if for any $t_1,t_2\in I$, $t_1<t_2$, $\gamma$ satisfies
	\begin{equation}\label{eq:msc_H_phi_t}\tag{VIt}
		\phi(t_2,\gamma(t_2))-\phi(t_1,\gamma(t_1))=\int_{t_1}^{t_2}L(\gamma(s),\dot{\gamma}(s))+\mathbf{q}(s,\gamma(s))+H(\gamma(s),\mathbf{p}(s,\gamma(s)))\ ds,
	\end{equation}
	or, equivalently,
	\begin{align*}
		\dot{\gamma}(t)=H_{p}(\gamma(t),\mathbf{p}(t,\gamma(t))),\quad a.e.\ t\in I.
	\end{align*}
	\item For the ``minimal energy'' selection $(\mathbf{q}^{\#}_{\phi,H},\mathbf{p}^{\#}_{\phi,H})$, we call any associated maximal slope curve $\gamma:I\to M$ for $(\phi,H)$ a \emph{strict singular characteristic} for the pair $(\phi,H)$. 
\end{enumerate}
\end{defn}

Notice that we use $q+H(x,p)$ as a new energy term for the time-dependent case (see, for instance, \cite{Cheng_Hong2022a}). We emphasize all the results in this paper for the time-dependent case can be proved in a similar way as those for the time-independent case.

\bibliographystyle{plain}
\bibliography{mybib}

\end{document}